\documentclass[11pt]{amsart} 
\usepackage[bottom=1in]{geometry}
\usepackage{esint} 
\usepackage{graphicx, amssymb, hyperref }
\usepackage[mathscr]{euscript}
\usepackage{todonotes}
\usepackage[english]{babel}
\usepackage{mathrsfs}
\usepackage{enumitem}
\usepackage[scr=boondoxo]{mathalfa}

\usepackage{pdftricks}
\begin{psinputs}
   \usepackage{pstricks}
   \usepackage{multido}
\end{psinputs}
 \usepackage{epsfig}
 \usepackage{pst-grad} 
 \usepackage{pst-plot} 
 \usepackage{xcolor}
\usepackage{caption}  
 \usepackage{mathtext}

\newtheorem{definition}{Definition}
\newcommand{\ssize}{\text{size}\,}
\newcommand{\eenergy}{\text{energy}\,}
\newtheorem{lemma}{Lemma}
\newtheorem{corollary}{Corollary}
\newtheorem{proposition}{Proposition}
\newtheorem{theorem}{Theorem}
\newtheorem{rremark}{Remark}

\newcommand{\sssize}{\widetilde{\text{size}\,}}

\newcommand{\one}{\mathbf{1}}
\newcommand{\dist}{\text{ dist }}
\newcommand{\rr}{\mathbb}
\newcommand{\ii}{\mathscr}
\newcommand{\ic}{\mathcal}
\newcommand{\ci}{\tilde{\chi}}

\newcommand{\ds}{\displaystyle}
\newcommand{\two}{\left( 2 \right)}
\newcommand{\di}{\left( d \right)}

\newtheorem{question*}{Question}
\newtheorem{claim*}{Claim}
\newtheorem*{main*}{\underline{Induction statement}}

\newcommand{\lft}{\big|}
\newcommand{\rg}{\big|}
\newcommand{\supp}{\text{supp\,}}

\newenvironment{thmbis}[1]
  {\renewcommand{\theproposition}{\ref{#1}$'$}%
   \addtocounter{proposition}{-1}%
   \begin{proposition}}
  {\end{proposition}}

\def\Xint#1{\mathchoice
   {\XXint\displaystyle\textstyle{#1}}%
   {\XXint\textstyle\scriptstyle{#1}}%
   {\XXint\scriptstyle\scriptscriptstyle{#1}}%
   {\XXint\scriptscriptstyle\scriptscriptstyle{#1}}%
   \!\int}
\def\XXint#1#2#3{{\setbox0=\hbox{$#1{#2#3}{\int}$}
     \vcenter{\hbox{$#2#3$}}\kern-.5\wd0}}

\def\aver#1{\Xint-_{#1}}

\author{Cristina Benea }
\address{Cristina Benea, Universit\'{e} de Nantes, Laboratoire Jean Leray, Nantes 44322, France}
\email{cristina.benea@univ-nantes.fr}

\author[Camil Muscalu]{Camil Muscalu*}
\thanks{$^*$The author is also a Member of the ``Simion Stoilow" Institute of Mathematics of the Romanian Academy}
\address{Camil Muscalu, Department of Mathematics, Cornell University, Ithaca, NY 14853, USA}
\email{camil@math.cornell.edu}

 \title{The Helicoidal Method
} 
\begin{document}
\begin{abstract}
We present a method for proving multiple vector-valued inequalities for several operators in harmonic analysis, based on stopping times and localizations. As it turns out, the local estimate can be used for proving sparse domination for the scalar operator and its multiple vector-valued extensions, and hence also weighted estimates.
\end{abstract}
\maketitle

\section{Introduction}
\label{sec:intro}

The present work is an expository paper on the helicoidal method, which we introduced in \cite{vv_BHT}, \cite{quasiBanachHelicoid}, \cite{sparse-hel}. 
Put in a shell, the helicoidal method comes down to encapsulating very precise data at a \emph{local} level, and subsequently unfolding the (possibly superfluous) information on certain collections of intervals that are chosen through a careful stopping time. 

\smallskip 

The helicoidal method was developed with the purpose of dealing with vector-valued (depth $1$) extensions of operators, and later with multiple vector-valued extensions (depth $n$, where $n$ represents the number of iterated spaces considered). The local estimates at any level /  of any depth are very similar in nature, and they all resemble the scalar (level $0$) one, which is essential to our approach. Hence the helicoid is a metaphor for this iterative procedure, where estimates at the $n^{\text{th}}$ level can be achieved only through estimates from the previous level, and they all are of the same kind. See also Remark \ref{remark:expl-name-hel} below.
\smallskip

For the reader's convenience, we will try to leave out the technical details and instead present the development of the ideas and their implications, from the scalar to the multiple vector-valued setting. With this in mind, we will take the time to review certain results which are classical in the field of time-frequency analysis, in spite of them being obvious to experts.

Our prime example is the bilinear Hilbert transform, an operator given by
\begin{equation}
\label{eq:def-BHT}
BHT(f, g)(x)=p.v. \int_{\rr R} f(x-t) g(x+t)\frac{dt}{t}, 
\end{equation}
which falls beyond the scope of classical Calder\'on-Zygmund theory. It was introduced by Calder\'on in his attempt to understand the Cauchy integral on Lipschitz curves and what are now termed ``Calder\'on's commutators".

The prototypical localized inequality obtained through the helicoidal method is
{\fontsize{10}{10}\begin{equation}
\label{eq:loc-BHT-first}
\vert \langle BHT_{I_0}(f, g ), h \rangle \vert \lesssim  \sup_{I \subseteq I_0} \big( \frac{1}{|I|} \int_{\rr R} |f|^{s_1} \ci_I dx \big)^\frac{1}{s_1} \,  \sup_{I \subseteq I_0} \big( \frac{1}{|I|} \int_{\rr R} |g|^{s_2} \ci_I dx \big)^\frac{1}{s_2} \, \sup_{I \subseteq I_0} \big(  \frac{1}{|I|} \int_{\rr R} |h|^{s_3} \ci_I dx \big)^\frac{1}{s_3} \cdot |I_0|,
\end{equation}}
\smallskip
where $I_0$ is some (dyadic) interval contained in $\rr R$, $\ci_I(x): =\big( 1+\frac{\dist (x, I)}{|I|}  \big)^{-100}$ and $s_1, s_3, s_3 \in (1, \infty)$ are Lebesgue exponents such that $ \frac{1}{s_1}+\frac{1}{s_2}+\frac{1}{s_3}<2$. In fact, it will be clear later on that it suffices to find $0 \leq \theta_1, \theta_2, \theta_3 <1$ with $\theta_1+\theta_2 +\theta_3=1$, such that for every $1 \leq j \leq 3$,  $\frac{1}{s_j} < \frac{1+\theta_j}{2}$ (in particular, we can replace the tuple $(s_1, s_2, s_3)$ by any $(\tilde{s}_1, \tilde{s}_2, \tilde{s}_3)$ with $s_j<\tilde{s}_j$ for all $1 \leq j \leq 3$).

It is the scalar case and the iterative procedure of the helicoidal method that imply also the depth-$n$ vector-valued version of \eqref{eq:loc-BHT-first}. Such a local estimate, together with a stopping time of type VVST implies the corresponding result of depth $n+1$ (the induction step is presented in detail in Section \ref{sec:method-proof}), while a stopping time of type SST (as described in the proof of Theorem \ref{thm:local->sparse}) implies a multiple vector-valued sparse domination result. The two stopping time procedures (of type VVST or SST) are discussed in more detail in Remark \ref{remark:stopping-times}. The multiple vector-valued extensions of $BHT$ can be regarded as $BHT \otimes Id \otimes \ldots \otimes Id$, and it can be seen that the structure of the operator plays a role only in proving the scalar local estimate (taking the form of a localized trilinear form as in \eqref{eq:loc-BHT-first}, or formulated in the quasi-Banach context).

The bilinear Hilbert transform operator maps $L^{q_1} \times L^{q_2}$ into $L^{q}$ for any $1< q_1, q_2 \leq \infty$, $\frac{2}{3}< q <\infty$ satisfying the H\"older condition $\frac{1}{q_1}+\frac{1}{q_2}+\frac{1}{q'}=1$. This is a classical result due to Lacey and Thiele from the '$90 \,$s \cite{initial_BHT_paper}. If $q \geq 1$, this result can be rephrased as
\[
\vert \langle  BHT(f, g), h   \rangle  \vert \lesssim \|f\|_{q_1} \, \| g  \|_{q_2} \, \|h \|_{q'}, \quad \text{whenever    } \frac{1}{q_1}+\frac{1}{q_2}+\frac{1}{q'}=1.
\]

In contrast, the estimate \eqref{eq:loc-BHT-first} of the trilinear form, though local, seems to yields more information than one might need: the estimate holds provided the Lebesgue exponents verify $\ds \frac{1}{s_1}+\frac{1}{s_2}+\frac{1}{s_3}<2$. Indeed, we eventually obtain a stronger result (which first appeared in \cite{sparse-hel}) in the form of a Fefferman-Stein inequality for $BHT$:
\begin{theorem}
\label{thm:Fefferman-Stein-BHT}
Let $\ic M_{s_j}(f):=\big(\ic M (|f| ^{s_j}) \big)^\frac{1}{s_j}$, where $\ic M$ denotes the usual Hardy-Littlewood maximal function. Then if $0<q \leq 1$ and $1<s_1, s_2 < \infty$ with $\frac{1}{s_1}+\frac{1}{s_2}<\frac{3}{2}$, we have 
\[
\| BHT(f, g)\|_q \lesssim \| \ic M_{s_1} (f)  \cdot \ic M_{s_2}(g) \|_q.
\]

If $1<q < \infty$, $\frac{1}{s_1}+\frac{1}{s_2}<\frac{3}{2}$ and  $\ds \frac{1}{s_1}+\frac{1}{s_2}< 1+\frac{1}{q}$, then a similar inequality holds. 
\end{theorem}

The dichotomy between the cases $q \in (0, 1]$ and $q \in (1, \infty)$ is due to a subadditivity condition and Theorem \ref{thm:sparse-BHT}. It can also be rephrased, for all $0<q <\infty$ as
\[
\| BHT(f, g)\|_q \lesssim \| \ic M_{s_1} (f)  \cdot \ic M_{s_2}(g) \|_q, \quad \text{provided} \quad \frac{1}{s_1}+\frac{1}{s_2}< \min \big( \frac{3}{2}, 1+\frac{1}{q} \big),
\]
but then the role of subadditivity is concealed. This is discussed in more detail in Remark \ref{remark:rangeFeffStein} \eqref{remark:rangeFS-item}, and we note in passing that a weighted version was stated in \cite{sparse-hel}.

Nevertheless, the result of Theorem \ref{thm:Fefferman-Stein-BHT} is powerful enough as it allows us to recover the full range of boundedness of $BHT$ from \cite{initial_BHT_paper}. We illustrate this for $\frac{2}{3}<q<1$: first, by H\"older, we immediately obtain
\[
\| BHT(f, g)\|_q \lesssim \| \ic M_{s_1} (f) \|_{q_1}  \cdot  \| \ic M_{s_2}(g) \|_{q_2}, \quad \text{where} \quad \frac{1}{q_1}+\frac{1}{q_2}=\frac{1}{q}.
\]
Hence $BHT: L^{q_1} \times L^{q_2} \to L^{q}$ provided we can find $1<s_1, s_2<\infty$ with  $\frac{1}{s_1}+\frac{1}{s_2}<\frac{3}{2}$, and $s_1<q_1$ and $s_2<q_2$. It suffices to take $s_j=q_j-\epsilon$, with $\epsilon$ sufficiently small.

The result of Theorem \ref{thm:Fefferman-Stein-BHT} loses its relevance if $q$ is too small: the norms of the maximal operators on the right hand side become infinite. This will not happen on the torus $\rr T$ though. There the bilinear Hilbert transform is defined as
\[
BHT \,_{\rr T}(f, g)(t)=\int_{\rr T} f(t - \theta) g(t+ \theta) \cot (\pi \theta) d \, \theta,
\]
and Theorem \ref{thm:Fefferman-Stein-BHT} admits a corresponding formulation in the setting of periodic functions.

The Fefferman-Stein inequality is deduced from \eqref{eq:loc-BHT-first} through a careful selection of intervals $I_0$, thanks to the scattered structure of the collection of such intervals. This procedure is known in literature as ``sparse domination" and it resembles in spirit the ``good-$\lambda$ inequalities" which are commonly used for showing that Calder\'on-Zygmund operators are controlled (in norm) by the Hardy-Littlewood maximal function. For the sparse domination, the unfolding of the information given by \eqref{eq:loc-BHT-first} occurs on disjoint sets, at the price of making appear Hardy-Littlewood maximal functions.

For the vector-valued inequalities, the unfolding process is slightly different. In this case, the goal is to convert the local $L^{s_j}$ maximal averages into $L^{r_j}$ averages and later into $L^{q_j}$ averages. This will yield ultimately the vector-valued extension $$BHT: L^{q_1}(\ell^{r_1}) \times L^{q_2}(\ell^{r_2}) \to L^{q}(\ell^r),$$ so long as we can find $s_1, s_2, s_3$ as in \eqref{eq:loc-BHT-first} satisfying simultaneously 
{\fontsize{10}{10}\[
\frac{1}{q_1}, \frac{1}{r_1}<\frac{1}{s_1}, \quad \frac{1}{q_2}, \frac{1}{r_2}<\frac{1}{s_2}  \quad \text{and} \quad \frac{1}{q'}, \frac{1}{r'}<\frac{1}{s_3}.
\]}

Weighted or vector-valued extensions of the above Fefferman-Stein inequality will be discussed later on. Moreover, the product of two maximal functions can be replaced by a bilinear maximal function (of the type presented in \cite{NewMaxFnMultipleWeights}).  We want to point out that Theorem \ref{thm:Fefferman-Stein-BHT} holds for all $0<q<\infty$, and that the quasi-Banach case (i.e., when $q<1$) is especially relevant for the $BHT$ operator: the condition $\frac{2}{3}< q$, which is not known to be necessary, hints to the importance of orthogonality in the proof, which is indicated by the fact that at least one of $q_1$ and $q_2$ is greater than $2$. 

The Fefferman-Stein inequality of Theorem \ref{thm:Fefferman-Stein-BHT} also points out to the controlling role that the Hardy-Littlewood maximal functions have over the bilinear Hilbert transform operator. Although somehow concealed, it already represents a key point in the proof (see \cite{LaceyThieleBHTp>2}, \cite{initial_BHT_paper}, \cite{biest}). That its importance becomes more visible in the context of vector-valued and weighted extensions, the area where the Fefferman-Stein inequalities originally appeared for Calder\'on-Zygmund operators \cite{FefStein_vvmaximal}, should not come as a surprise.

The inequality \eqref{eq:loc-BHT-first} appears in \cite{vv_BHT}, in the special case of restricted-type functions, i.e. functions which are bounded above by characteristic functions of sets of finite measure. Considering such particular functions is natural for the $BHT$ operator, since at some point we need to attune $L^1$ and $L^2$ information. Outside the local $L^2$ setting, this step of combining $L^1$ and $L^2$ data cannot be avoided; interpolation is required sooner or later. By working with restricted-type functions, interpolation is reduced to a technical issue which is performed last.

Later on, we will present some applications to the helicoidal method, including sparse vector-valued estimates for several operators in time-frequency analysis. Of particular importance are the multiple vector-valued extensions, which enabled us to answer several open questions related to multi-parameter operators in \cite{vv_BHT}.

The helicoidal method provides a self-inclusive, robust method for proving vector-valued extensions for operators in harmonic analysis. The vector spaces considered are iterated $L^p$ spaces: if $n \geq 1$, given an $n$-tuple $R=(r^1, \ldots, r^n)$ and $(\ii W, \Sigma, \mu):=\big( \prod_{j=1}^n \ii W_j, \prod_{j=1}^n \Sigma_j, \prod_{j=1}^n \mu_j   \big)$ a product of totally $\sigma$-finite measure spaces, the $L^R$ norm is defined as
\[
\| \vec f  \|_{L^R(\ii W, \mu)}:= \big(  \int_{\ii W_1} \ldots \big( \int_{\ii W_n} \vert \vec f(w_1, \ldots , w_n)   \vert^{r^n} d \mu_n(w_n)     \big)^{r^{n-1}/r^n} \ldots d \mu_1(w_1) \big)^{1/{r^1}}.
\]
If $S$ is a (sub)linear operator, its range will be the set $\ds Range(S):=\lbrace p : S:L^p \to L^p \rbrace$. Similarly, for an $m$-(sub)linear operator $T$, its range will consist of
\[
Range(T):=\lbrace (p_1, \ldots, p_m, p): T:L^{p_1}\times \ldots \times L^{p_m} \to L^p \rbrace.
\]

As a matter of fact, we want to find the range of the ``depth $n$" vector-valued extension, i.e. all possible Lebesgue exponents for which
\begin{equation}
\label{eq:vv-depth-n}
\vec{T}_n : L^{p_1}\left( \rr{R}; L^{R_1}(\ii W, \mu) \right) \times \ldots \times L^{p_m}\left( \rr{R}; L^{R_m}(\ii W, \mu) \right) \to L^{p}\left( \rr{R}; L^{R}(\ii W, \mu) \right),
\end{equation}
where the $n$-tuples $R_k=\big( r^1_k, \ldots, r^n_k \big), R=\left( r^1, \ldots, r^n \right)$ satisfy for every $1 \leq k \leq m$,  $1 \leq j \leq n$
\begin{equation}
\label{eq:cond-r-tuples}
1< r_k^j \leq \infty, \quad \frac{1}{m}<r^j<\infty \quad \text{and}\quad \frac{1}{r_1^j}+\ldots +\frac{1}{r_m^j}=\frac{1}{r^j}.
\end{equation}

In the case of linear operators, it was long understood that the vector-valued extensions are related to weighted inequalities: the former are implied by the latter via extrapolation theory. Knowing that $T:L^{p_0}(w) \to L^{p_0}(w)$ for the correct class of weights $w$ (which depends on $p_0$), one can deduce that $T: L^p(w) \to L^p(w)$ for any $p \in Range(T)$, with the associated class of weights (which now depends on $p$). Furthermore, extrapolation also yields weighted vector-valued extensions for any $n$-tuple $R$ with the $r^j \in Range(T)$ \cite{RFCuervaBook}.

This theory is also available for multilinear Calder\'on-Zygmund operators (\cite{extrapolation_multilinear}), provided no $L^\infty$ spaces are involved. If $T$ is such an $m$-linear operator, then 
\[
T: L^{p_1} \times \ldots \times L^{p_m} \to L^p \qquad \text{for all         } 1<p_j \leq \infty, \, \frac{1}{m}<p<\infty.
\]
By means of extrapolation, it follows that 
\[
\vec{T}_n : L^{p_1}\left( \rr{R}; L^{R_1}(\ii W, \mu) \right) \times \ldots \times L^{p_m}\left( \rr{R}; L^{R_m}(\ii W, \mu) \right) \to L^{p}\left( \rr{R}; L^{R}(\ii W, \mu) \right),
\]
for all $\ds 1<p_j < \infty, \, \frac{1}{m}<p<\infty,$ and all $n$-tuples $R_1, \ldots, R_m, R$ so that $1<r^j_k, r^j<\infty$. However, our method allows us to deal with $L^\infty$ spaces as well: see Theorem \ref{thm:main-thm-paraprod} below.

Throughout the paper, we chose not to investigate the endpoint behavior of the operator (that is, when the target space is $L^{p, \infty}$), though it is traceable without difficulty.

For more general multilinear operators, in particular for the bilinear Hilbert transform, a suitable theory of extrapolation was only recently developed in \cite{extrap-BHT}. For the moment, this extrapolation theory doesn't seem to match perfectly the vector-valued results from \cite{vv_BHT}, \cite{quasiBanachHelicoid}.

\psscalebox{.6 .6} 
{
\begin{pspicture}(-3,-5.2)(11.520909,3)
\rput[bl](4,-2){ \Large $\left(1, 0, 0 \right)$}
\rput[bl](4.55,-3.73205080756888){\Large $\left( 1, \frac{1}{2}, -\frac{1}{2}\right)$}
\rput[b](8,1.5){\Large $\left( 0, 0, 1\right)$}
\rput[br](11.8,-2){ \Large $\left( 0, 1, 0 \right)$}
\rput[bl](9.4,-3.73205080756888){\Large  $\left( \frac{1}{2}, 1, -\frac{1}{2}\right)$}
\pscustom[linecolor=black, linewidth=0.04]
{
\newpath
\moveto(1.2857143,-5.7638702)
}
\pscustom[linecolor=black, linewidth=0.04]
{
\newpath
\moveto(5.0,-5.7638702)
}
\psdiamond[linecolor=black, linewidth=0.04, dimen=outer](8,-2)(2,3.46410161513775)
\pspolygon[linecolor=black, linewidth=0.04, fillstyle=solid,fillcolor=lightgray](8,1.46410161513775)(10,-2)(9,-3.7)(7,-3.7)(6,-2)
\psline[linecolor=black, linewidth=0.04](6,-2)(10,-2)
\label{fig:rangeBHT}
\end{pspicture}
}

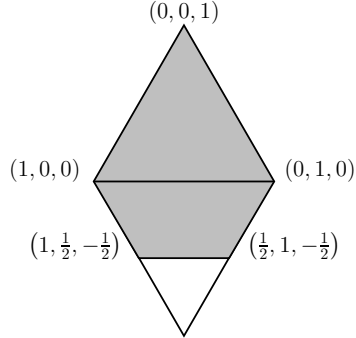
\captionof{figure}{Range for $BHT$ operator : $\big( \frac{1}{p}, \frac{1}{q}, \frac{1}{s'} \big)$ lies in the shaded region}

\smallskip
The extrapolation apparatus requires the existence of suitable weighted estimates for a certain tuple $(p_{1,0}, \ldots, p_{m, 0}, p_0)$ and all weights in the corresponding class. For the $BHT$ operator, weighted results were obtained in \cite{weighted_BHT} in the Banach case. The limited range multilinear extrapolation theory from \cite{extrap-BHT} extends these results to the quasi-Banach case as well, but it doesn't recover all the known vector valued estimates for $BHT$ from \cite{vv_BHT} and \cite{quasiBanachHelicoid}. The ``correct" class of weights for $BHT$, which will establish the appropriate vector-valued extensions, is still to be determined. 

\smallskip
We will see that the (vector-valued) local estimate for $BHT$ implies a (vector-valued) sparse domination result, which further produces weighted vector-valued estimates. We will present in detail the helicoidal method from \cite{vv_BHT}, the extension to quasi-Banach spaces from \cite{quasiBanachHelicoid}, and later on, the sparse domination results of \cite{sparse-hel}. The central study case will be the bilinear Hilbert transform, but the method can easily be adapted to other operators in harmonic analysis that allow for a wave packet decomposition; among these: Carleson and variational Carleson operators, paraproducts, etc. 

Sparse domination for the trilinear form associated to $BHT$ and for the bilinear form of the variational Carleson operator have been obtained in \cite{weighted_BHT} and \cite{sparse-var-Carleson}. Essential to the authors' method are the Carleson embeddings of \cite{Carleson-embedding-belowL2}, \cite{outer-measure-var-Carleson} (Carleson embeddings represent a way of interpreting time-frequency information exclusively as spatial information and they are ubiquitous, although implicitly, in time-frequency analysis). The ``localized outer $L^p$ embeddings", formulated in the language of outer measures of \cite{outer_measures}, sit at the core of \cite{weighted_BHT} and \cite{sparse-var-Carleson}, and are similar in spirit to the localization principle developed in \cite{vv_BHT}, \cite{quasiBanachHelicoid}. More details describing connections between the outer measure approach (of \cite{outer_measures}) and the general size and energy one (introduced in \cite{biest}) will be presented in Section \ref{sec:outer-measures}.

\smallskip
Before we proceed, we want to mention a few applications of the vector-valued inequalities implied by our method.

\subsection{A Rubio de Francia inequality for iterated Fourier integrals}
\label{sec:RF-for-BHT}
Our incipient project was the study of an $\ell^r$ inequality for iterated Fourier integrals associated to an arbitrary collection $\ds \lbrace [a_k, b_k]  \rbrace_k$ of disjoint intervals:
\begin{equation}
\label{eq:def-RF-it-integrals}
T_r(f, g)(x):=\big(  \sum_{k=1}^N \big \vert  \int_{a_k<\xi_1 <\xi_2<b_k} \hat{f}(\xi_1) \hat{g}(\xi_2) e^{2 \pi i x \left( \xi_1+\xi_2 \right)}  d \xi_1 d \xi_2 \big \vert^r \big)^\frac{1}{r}.
\end{equation}

We called $T_r$ a \emph{Rubio de Francia operator for iterated Fourier integrals} since it represents a bilinear version of Rubio de Francia's square function for Fourier projections associated to an arbitrary collection of disjoint intervals from \cite{RF}. 

\begin{theorem}
\label{thm-RF-itF-integrals}
If  $1 \leq r < 2$, then
\[
\| T_r(f,g) \|_s \lesssim \|f\|_p  \|g\|_q
\]
whenever $\dfrac{1}{p}+\dfrac{1}{q}=\dfrac{1}{s}$, and $p, q, s$ satisfy
\[
0 < \frac{1}{p}, \frac{1}{q}<\frac{1}{2}+\frac{1}{r}, \qquad -\frac{1}{r'}<\frac{1}{s'}<1.
\]

On the other hand, if $r \geq 2$, $T_r$ is a bounded operator with the same range as $BHT$ ($L^\infty$ spaces excluded), i.e. $T_r : L^p \times L^q \to L^s$ for all $1< p, q <\infty, \frac{2}{3}< s <\infty$ so that $\frac{1}{p}+\frac{1}{q}=\frac{1}{s}$.
\end{theorem}

Our initial proof of Theorem \ref{thm-RF-itF-integrals} was based on localizations and stopping times with respect to averages of Rubio de Francia's square function associated to $f$ and $g$ respectively. Later on, we realized that the same scheme yields vector-valued extensions for $BHT$, so in fact Theorem \ref{thm-RF-itF-integrals} becomes a consequence of the inequalities that will be presented in Section \ref{sec-vv-ineq-BHT}.
\smallskip
Localization is natural when dealing with arbitrary Fourier projections: an appropriate maximal operator can control, at a local level, all the different frequency scales. The frequency information is reduced entirely to spatial information (through the boundedness of certain Fourier multipliers or square functions) and our maximal operator is in most cases represented as a ``$\sssize$" (see \eqref{def:sssize}).
\smallskip

The study of the operator $T_r$ was motivated by a particular question related to an $AKNS$ system. These are systems of differential equations of the form
\[
u'=i \lambda Du +Au,
\]
where $u=[u_1,\ldots, u_n]^t$ is a vector-valued function defined on $\mathbb{R}$, $D$ is a diagonal $n \times n$ matrix with real and distinct entries $d_1,d_2, \ldots d_n$, and $A=\left( a_{jk} (\cdot)  \right)_{j,k=1}^n$ is a matrix valued function defined on $\mathbb{R}$, and so that $a_{jj} \equiv 0$ for all $1 \leq j \leq n$. 

The operator $T_r$ turns out to be useful in the case when $A$ can be written as the sum of an upper-triangular matrix with entries in $L^2$ and a lower-triangular matrix with entries in $L^p$ spaces with $1\leq p <2$.

\subsection{Vector-valued inequalities for $BHT$}
\label{sec-vv-ineq-BHT}
Prior to \cite{vv_BHT}, vector-valued extensions for $BHT$ where known from \cite{vv_bht-Prabath} and \cite{vv_quartile}, within a rather restricted range. We managed to remove some of these constraints, even though, as of this moment, the exact domain of boundedness for $BHT$ or its vector-valued extensions is not known precisely. We recall the multiple vector-valued extensions for $BHT$ from \cite{vv_BHT} and \cite{quasiBanachHelicoid}: 

\begin{theorem}
\label{thm-main-vvBHT-1vs}
For any triple $( r_1, r_2, r)$ with $1<r_1, r_2 \leq \infty, \, \frac{2}{3}<r <\infty$ and so that $\frac{1}{r_1}+\frac{1}{r_2}=\frac{1}{r}$, there exists a non-empty set $\ii D_{ r_1, r_2, r}$ of triples $(p, q, s)$ satisfying $\frac{1}{p}+\frac{1}{q}=\frac{1}{s}$, for which
\[
\big \|  \big\| BHT(\vec f, \vec g)\big\|_{L^r\left( \ii W, \mu \right)} \big\|_{L^s\left( \rr R \right)} \lesssim \big \|  \big\|  \vec f  \big\|_{L^{r_1}\left( \ii W, \mu \right)}  \big\|_{L^p\left( \rr R \right)} \big \|  \big\|  \vec g \big\|_{L^{r_2}\left( \ii W, \mu \right)}  \big\|_{L^q\left( \rr R \right)}.
\]
The set $\ii D_{ r_1, r_2, r}$ can be described as follows: $(p, q, s) \in \ii D_{r_1, r_2, r }$ if there exist $\theta_1, \theta_2, \theta_3$ so that $0 \leq \theta_1, \theta_2, \theta_3 <1, \, \theta_1+\theta_2+\theta_3 =1$ and 
\[
\frac{1}{r_1}, \frac{1}{p} <\frac{1+\theta_1}{2}, \quad \frac{1}{r_2}, \frac{1}{q} <\frac{1+\theta_2}{2} \quad \frac{1}{r'}, \frac{1}{s'} <\frac{1+\theta_3}{2}.
\]
\end{theorem}

More generally, we can consider extensions to multiple vector-valued spaces:

\begin{theorem}
\label{thm:main-thm-BHT}
Consider the tuples $R_1=\left(r_1^1, \ldots, r_1^n \right), R_2=\left(r_2^1, \ldots, r_2^n \right)$ and $R=\left(r^1, \ldots, r^n \right)$ satisfying for every $1 \leq j \leq n: 1<r_1^j, r_2^j \leq \infty, \, \frac{2}{3}<r^j <\infty$, and $\frac{1}{r_1^j}+\frac{1}{r_2^j}=\frac{1}{r^j}$. Then the bilinear Hilbert transform allows for a multiple vector-valued extension 
\begin{equation}
\label{eq:BHT-multi-vec}
\overrightarrow{BHT}^n_{\vec R}: L^p\left( \rr R; L^{R_1}\left( \ii W, \mu \right) \right) \times L^q\left( \rr R; L^{R_2}\left( \ii W, \mu \right) \right) \to L^s\left( \rr R; L^{R}\left( \ii W, \mu \right) \right)
\end{equation}
provided there exist $\theta_1, \theta_2, \theta_3$ so that $0 \leq \theta_1, \theta_2, \theta_3 <1,  \theta_1+\theta_2+\theta_3  =1$ and, for all $ 1 \leq j \leq n$, we have 
\begin{equation}
\label{eq:constraints-r_j-p}
\frac{1}{r_1^j}, \frac{1}{p} <\frac{1+\theta_1}{2}, \quad \frac{1}{r_2^j}, \frac{1}{q} <\frac{1+\theta_2}{2} \quad \frac{1}{\left(r^j\right)'}, \frac{1}{s'} <\frac{1+\theta_3}{2}.
\end{equation}

In this case we denote by $$\ii D_{R_1, R_2, R}:=Range(\overrightarrow{BHT}^n_{\vec R})=\lbrace (p, q, s): \eqref{eq:BHT-multi-vec} \text{ holds  } \rbrace,$$
and we note that $\ii D_{R_1, R_2, R} =\bigcap\limits_{1 \leq j \leq n} \ii D_{r_1^j, r_2^j, r^j}$. 
\end{theorem}

It is not difficult to see that in the ``local $L^2$" situation ($2 \leq r^j_1, r^j_2, (r^j)' \leq \infty$), the conditions on $r^j_1, r^j_2$ and $(r^j)'$ from \eqref{eq:constraints-r_j-p} are automatically satisfied (assuming that $0<\theta_1, \theta_2, \theta_3<1$), and hence the range of vector-valued extensions $\overrightarrow{BHT}^n_{\vec R}$ coincides with the range of the scalar operator:
\begin{corollary}
If the tuples $R_1=\left(r_1^1, \ldots, r_1^n \right), R_2=\left(r_2^1, \ldots, r_2^n \right)$ and $R=\left(r^1, \ldots, r^n \right)$ satisfy for every $1 \leq j \leq n: \, 2 \leq r_1^j, r_2^j, \left( r^j  \right)' \leq \infty$, then 
\[
\ii D_{R_1, R_2, R}:=Range(\overrightarrow{BHT}^n_{\vec R})=Range(BHT).
\]
\end{corollary}

\subsection{Boundedness of multi-parameter operators and Leibniz rules}

As applications to our vector-valued inequalities, we can prove the boundedness (sometimes in mixed norm $L^p$ spaces) of multi-parameter operators. While this is not the case for $BHT \otimes BHT$, which, as claimed in \cite{bi-parameter_paraproducts}, does not satisfy any $L^p$ estimates, it is true for $\Pi \otimes \Pi, \Pi \otimes \ldots \otimes \Pi$ and $BHT \otimes \Pi \ldots \otimes \Pi$.

The operator $\Pi$ is a \emph{paraproduct}, i.e. a bilinear operator of the form
\begin{equation*}
\label{eqn:paraprod 1}
(f,g) \mapsto \sum_{k } \left(\left(f \ast \psi_k   \right)   \cdot \left(g \ast \psi_k  \right) \right) \ast \varphi_k(x)= \sum_{k}P_k(Q_kf \cdot Q_k g)
\end{equation*}
\begin{equation*}
\label{eqn:paraprod 2}
(f,g) \mapsto \sum_{k } \left(\left(f \ast \varphi_k   \right)   \cdot \left(g \ast \psi_k  \right) \right) \ast \psi_k(x)= \sum_{k} Q_k(P_kf \cdot Q_k g).
\end{equation*}
\begin{equation*}
\label{eqn:paraprod 3}
(f,g) \mapsto \sum_{k } \left(\left(f \ast \psi_k   \right)   \cdot \left(g \ast \varphi_k  \right) \right) \ast \psi_k(x)= \sum_{k}Q_k(Q_kf \cdot P_k g).
\end{equation*}

Here $\psi_k(x)=2^k \psi(2^k x)$, $\varphi_k(x)=2^k \varphi(2^k x)$,  $\hat{\varphi}(\xi) \equiv 1$ on $\ds \left[-1/2, 1/2 \right]$, is supported on $[-1,1]$  and $\hat{\psi}(\xi)=\hat{\varphi}(\xi/2)-\hat{\varphi}(\xi) $. The $\ds \lbrace Q_k \rbrace_k$ represent Littlewood-Paley projections onto the frequency $|\xi| \sim 2^k$, while $\ds \lbrace P_k \rbrace_k$ are convolution operators associated with dyadic dilations of a nice bump function of integral 1.

The study of a general bilinear \emph{Mikhlin} multiplier $T_m$ defined by
\begin{equation*}
T_m(f, g)(x):= \int_{\rr R^2} m(\xi, \eta) \, \hat{f}(\xi) \, \hat{g}(\eta) \, e^{2 \pi i x \left(\xi+\eta  \right)} d \xi \, d \eta,
\end{equation*}
where the symbol $m$ satisfies 
\begin{equation}
\label{eq:Mikhlin-one-param-condition}
\vert  \partial_{\xi}^{\alpha}\partial_{\eta}^{\beta} m(\xi, \eta) \vert \lesssim \vert (\xi, \eta ) \vert^{-(\alpha+\beta)},
\end{equation}
can be reduced to the study of (frequency modulated) paraproducts. This is done by restricting the symbol $m$ to Whitney cubes with respect to the origin in the $\xi \eta$-plane and using Fourier series decompositions. The procedure, which is described in \cite{multilinear_harmonic}, allows to write $T_m$ as a finite sum of terms of the form
\[
\sum_{n_1, n_2 \in \rr Z} \sum_{k \in \rr Z} C_{n_1, n_2}^k Q_{k} ( P_{k, n_1}f \cdot Q_{k, n_2}g), 
\] 
where $\ds \widehat{ P_{k, n_1}f}(\xi)= \hat{\phi}_k(\xi) \, e^{\frac{2 \pi i \, n_1 \, \xi}{2^k}}, \, \widehat{ Q_{k, n_2}g}(\eta)= \hat{\psi}_k(\eta) \, e^{\frac{2 \pi i \, n_2 \, \eta}{2^k}}$, while the Fourier coefficients $C_{n_1, n_2}^k$ satisfy, uniformly in $k$,
\begin{equation}
\label{eq:cond-Fourier-coeff}
\vert C_{n_1, n_2}^k  \vert \lesssim  \left( 1+ \vert n_1 \vert  \right)^{-\alpha} \, \left( 1+ \vert n_2 \vert  \right)^{-\beta}.
\end{equation}

Similarly, a bilinear bi-parameter Mikhlin multiplier $T_m^{\two}$ defined by 
\begin{equation}
\label{eq:two-param-multipl}
T_m^{\two}(f, g)(x,y):= \int_{\rr R^4} m(\xi_1, \xi_2, \eta_1, \eta_2) \, \hat{f}(\xi_1, \xi_2) \, \hat{g}(\eta_1, \eta_2) \, e^{2 \pi i x \left(\xi_1+\eta_1  \right)} \, e^{2 \pi i y \left(\xi_2+\eta_2 \right)} d \xi_1 \, d \xi_2 \, d \eta_1\, d \eta_2,
\end{equation}
where the symbol $m$ satisfies 
\begin{equation}
\label{eq:Mikhlin-two-param-condition}
\vert  \partial_{\xi_1}^{\alpha_1}\partial_{\xi_2}^{\alpha_2}\partial_{\eta_1}^{\beta_1} \partial_{\eta_2}^{\beta_2} m(\xi_1, \xi_2, \eta_1, \eta_2) \vert \lesssim \vert (\xi_1, \eta_1 ) \vert^{-(\alpha_1+\beta_1)} \, \vert (\xi_2, \eta_2 ) \vert^{-(\alpha_2+\beta_2)} ,
\end{equation}
can be written as a finite sum of terms of the form
\[
\sum_{n_1, n_2 \in \rr Z} \sum_{m_1, m_2 \in \rr Z} \sum_{k \in \rr Z} \sum_{\ell \in \rr Z} C_{n_1, n_2; m_1, m_2}^{k, \ell} Q_{k} \otimes Q_\ell ( P_{k, n_1} \otimes Q_{\ell, m_1} f \cdot Q_{k, n_2} P_{\ell, m_2} g), 
\]
with the Fourier coefficients $C_{n_1, n_2; m_1, m_2}^{k, \ell}$ satisfying, uniformly in $k$ and $\ell$
\begin{equation}
\label{eq:cond-Fourier-coeff-bi}
\vert C_{n_1, n_2; m_1, m_2}^{k, \ell}  \vert \lesssim  \left( 1+ \vert n_1 \vert  \right)^{-\alpha_1} \, \left( 1+ \vert n_2 \vert  \right)^{-\beta_1}\left( 1+ \vert m_1 \vert  \right)^{-\alpha_2} \, \left( 1+ \vert m_2 \vert  \right)^{-\beta_2}.
\end{equation}

The boundedness of $\Pi \otimes \Pi$ (or more generally, of a bilinear bi-parameter Mikhlin multiplier satisfying conditions \eqref{eq:Mikhlin-two-param-condition}) was proved in \cite{bi-parameter_paraproducts}, and in \cite{multi-parameter_paraproducts} the same was shown to hold for $d$-parameter paraproducts $\Pi \otimes \ldots \otimes \Pi$ and implicitly for bilinear $d$-parameter Mikhlin multipliers $T^{\di}_m$, which behave like bilinear Mikhlin multipliers in each parameter. 

We were able to recover these results in \cite{vv_BHT} and also, for $n=2$, we obtained a mixed norm estimate for $\Pi \otimes \Pi$ (in \cite{vv_BHT} and the case $\frac{1}{2}<s_2<1$ in \cite{quasiBanachHelicoid}). Via the reduction described above, these results transfer right away to bilinear bi-parameter multipliers:
\begin{theorem}
\label{thm:kenig-two-param}
Let $1<p_i, q_i \leq \infty$ and $\frac{1}{2}<s_i<\infty$, be so that $\ds \frac{1}{p_i}+\frac{1}{q_i}=\frac{1}{s_i}$ for any index $i=1, 2$. Then the bilinear bi-parameter Mikhlin multiplier $T_m^{\two}$ defined by \eqref{eq:two-param-multipl} satisfies the following mixed norm estimate:
\begin{equation}
\label{eq:mixed-norms-quasiBanach}
T_m^{\two} : L^{p_1}_xL^{p_2}_y \times  L^{q_1}_xL^{q_2}_y \to L^{s_1}_xL^{s_2}_y.
\end{equation}
Moreover, for any $d \geq 3$ a similar result holds for the bilinear $d$-parameter Mikhlin multiplier $T_m^{\di}$, provided $1<p_i, q_i<\infty$, $\frac{1}{2}<s_i<\infty$ and $\ds \frac{1}{p_i}+\frac{1}{q_i}=\frac{1}{s_i}$ for all $1 \leq i \leq d$:
\[
T_m^{\di} : L^{p_1}_{x_1} \ldots  L^{p_d}_{x_d} \times  L^{q_1}_{x_1} \ldots L^{q_d}_{x_d} \to L^{s_1}_{x_1} \ldots L^{s_d}_{x_d}.
\]
\end{theorem}

\begin{rremark}
\label{remark-d-param-mixed}
The $d$-parameter case, for $d \geq 3$, in the quasi-Banach setting (i.e. when at least one of $s_i<1$) is implied by a multiple vector-valued extension of the Chang-Fefferman inequality from \cite{mixed-norm-square-function}.

We also note that in contrast to the case $d=2$, we cannot allow for $p_i=\infty$ or $q_i=\infty$ if $1 \leq i \leq d-1$.

\end{rremark}

A similar result for the bi-parameter operator, corresponding to reflexive Banach spaces (i.e. $1<p_2, q_2, s_2<\infty$)  was proved using different techniques in \cite[Corollary 4]{francesco_UMDparaproducts}. Both approaches invoke vector-valued inequalities in the study of multi-parameter multilinear operators. 

Theorem \ref{thm:kenig-two-param} implies a new Leibniz rule in mixed-norm $L^p$ spaces, in the spirit of the one employed in \cite{kenig2004_vvLeibniz}:
\begin{theorem}
\label{thm:Leibniz}
For any $\alpha, \beta >0$
\begin{align*}
\Big \| D_1^\alpha D^\beta_2 (f \cdot g)  \Big \|_{L^{s_1}_xL^{s_2}_y} &\lesssim \Big \| D_1^\alpha D^\beta_2 f  \Big \|_{L^{p_1}_xL^{p_2}_y} \cdot \|g\|_{L^{q_1}_xL^{q_2}_y} +\|f\|_{L^{p_3}_xL^{p_4}_y} \cdot \Big \| D_1^\alpha D^\beta_2 g  \Big \|_{L^{q_3}_xL^{q_4}_y} \\
&+\Big \| D_1^\alpha f  \Big \|_{L^{p_5}_xL^{p_6}_y} \cdot \Big \| D^\beta_2 g  \Big \|_{L^{q_5}_xL^{q_6}_y} +\Big \| D_2^\beta f  \Big \|_{L^{p_7}_xL^{p_8}_y} \cdot \Big \| D^\alpha_1 g  \Big \|_{L^{q_7}_xL^{q_8}_y},
\end{align*}
whenever $1< p_j,q_j \leq \infty$, $\frac{1}{1+\alpha} <s_1<\infty$, $\max \big( \frac{1}{1+\alpha}, \frac{1}{1+\beta}  \big) < s_2<\infty$, and the indices satisfy the natural H\"{o}lder-type conditions.
\end{theorem}

The proof of the Leibniz rule above (as presented in \cite{vv_BHT} and \cite{quasiBanachHelicoid}) relies on multiple-vector-valued inequalities for paraproducts (that were obtained in \cite{vv_BHT} and \cite{quasiBanachHelicoid}), which we formulate in the following way:
\begin{theorem}
\label{thm:main-thm-paraprod}
Consider the tuples $R_1=\left(r_1^1, \ldots, r_1^n \right), R_2=\left(r_2^1, \ldots, r_2^n \right)$ and $R=\left(r^1, \ldots, r^n \right)$ satisfying for every $1 \leq j \leq n: 1<r_1^j, r_2^j \leq \infty, \frac{1}{2}<r^j <\infty$, and $\frac{1}{r_1^j}+\frac{1}{r_2^j}=\frac{1}{r^j}$. Then the paraproduct $\Pi$ satisfies the estimates
\[
\Pi : L^p\left( \rr R; L^{R_1}\left( \ii W, \mu  \right)  \right) \times L^q\left( \rr R; L^{R_2}\left( \ii W, \mu  \right)  \right) \to L^s\left( \rr R; L^{R}\left( \ii W, \mu  \right)  \right),
\]  
for any $1<p, q \leq \infty, \frac{1}{2}<s<\infty$ with $\frac{1}{p}+\frac{1}{q}=\frac{1}{s}$. 
\end{theorem}
\begin{rremark}
\label{remark:new-param-vv}
More generally, if $ \lbrace \Pi_w \rbrace_{w \in \ii W}$ is a family of paraproducts uniformly bounded with respect to the parameter $w \in \ii W$, then
\begin{equation}
\label{eq:mvv-ineq-param-w-parap}
 \big \|  \| \Pi_w (\vec f, \vec g)\|_{L^{R}\left( \ii W, \mu  \right) } \big\|_{L^s(\rr R)} \lesssim  \big \|  \| \vec f\|_{L^{R_1}\left( \ii W, \mu  \right) } \big\|_{L^p(\rr R)} \, \big \|  \| \vec g\|_{L^{R_2}\left( \ii W, \mu  \right) } \big\|_{L^q(\rr R)},
\end{equation}
for any $R_1, R_2, R, p, q$ and $s$ as above. 

The same is true for the operators of Theorem \ref{thm:main-thm-BHT} (and within the same range as that of Theorem \ref{thm:main-thm-BHT}): we can consider a family $\lbrace BHT_w \rbrace_{w \in \ii W}$ of operators whose symbols are singular along a given line, as long as we have uniform boundedness in the parameter $w \in \ii W$:
\begin{equation}
\label{eq:mvv-ineq-param-w}
\big \|  \big\| BHT_w(\vec f, \vec g)\big\|_{L^R\left( \ii W, \mu \right)} \big\|_{L^s\left( \rr R \right)} \lesssim \big \|  \big\|  \vec f  \big\|_{L^{R_1}\left( \ii W, \mu \right)}  \big\|_{L^p\left( \rr R \right)} \big \|  \big\|  \vec g \big\|_{L^{R_2}\left( \ii W, \mu \right)}  \big\|_{L^q\left( \rr R \right)}.
\end{equation}
\end{rremark}

Leibniz-type rules have been investigated also in \cite{bi-parameter_paraproducts}, \cite{multilinear_harmonic}, \cite{graf-Leibniz_rules}, \cite{TorresWardLeibniz}, and mixed  norm estimates were considered recently in  \cite{HartTorresWu-smoothing-bil}, though in the setting of reflexive Banach spaces. In \cite{BourgainLi-kato} and \cite{GrafakosMaldonadoNaibo-Kato-endpoint} $L^\infty$ estimates are studied (but not in mixed norm spaces).

The multiple vector-valued inequality for $BHT$ from Theorem \ref{thm:main-thm-BHT} implies the boundedness of $BHT \otimes \Pi \otimes \ldots \otimes \Pi$ within the same range as that of $BHT$, as proved in \cite{vv_BHT}:
\begin{theorem}
\label{tensor product BHT d-paraproducts_intro}
For any $p, q, r$ with $\ds \frac{1}{p}+\frac{1}{q}=\frac{1}{r}$, with $1< p, q \leq \infty$ and $2/3 < r < \infty$:
\[
\|  BHT \otimes \Pi \otimes \ldots \otimes \Pi(f, g)   \|_{L^r(\rr{R}^{d+1})} \lesssim \|f\|_{L^{p}(\rr{R}^{d+1})} \| g  \|_{L^q(\rr{R}^{d+1})}.
\]
The same is true for $\ds \Pi \otimes \ldots \otimes \Pi \otimes BHT \otimes \Pi \otimes \ldots \otimes \Pi$.
\end{theorem} 

Our result gives a definite answer to a question from \cite{bi-parameter_paraproducts}, which was partially resolved in \cite{vv_bht-Prabath}. The same method of the proof and the multiple vector-valued inequalities \eqref{eq:mvv-ineq-param-w-parap} and \eqref{eq:mvv-ineq-param-w} yield a similar result for a more general $(d+1)$-parameter multiplier $T^{(d+1)}_m$, whose symbol is singular along a (non-degenerate) line in one of the parameters (hence, it behaves like the bilinear Hilbert transform) and acts as a bilinear Mikhlin multiplier in each of the remaining parameters. 

This can be extended to more singular $n$-linear operators, as those studied in \cite{multilinearMTT}, for which we need the corresponding vector-valued estimates from \cite{sparse-hel}.

\begin{theorem}
\label{thm:d-param-k-singular}
Let $d \geq 1$, $n \geq 2$ and $1 \leq k < \frac{n+1}{2}$. For any $1 \leq j \leq d+1$, let  $\Gamma_j'$ be a subspace of $\Gamma_j:=\lbrace (\xi_1^j, \ldots, \xi_n^j, \xi_{n+1}^j) \in \rr R^{n+1} : \xi_1^j + \ldots +\xi_n^j + \xi_{n+1}^j =0   \rbrace$ of dimension $k_j$, where $k_j=k$ for an index $1\leq j' \leq d+1$ and $k_j=0$ for the remaining indices. Assume that $\Gamma_{j'}'$ is not degenerate (in the sense of \cite{multilinearMTT}) and that $m: \rr R^{\left( d+1 \right) \left( n+1 \right)} \to \rr C$ satisfies the regularity condition
\[
\partial_{(\xi_1^1, \ldots, \xi_{n+1}^1)}^{\alpha_1} \ldots \partial_{(\xi_1^{d+1}, \ldots, \xi_{n+1}^{d+1})}^{\alpha_{d+1}} \, m(\vec{\xi}_1, \ldots, \vec{\xi}_{n+1})    \lesssim \prod_{j=1}^{d+1} \dist ((\xi_1^j, \ldots, \xi_{n+1}^j), \Gamma_j')^{- |\alpha_j|}
\]
for all partial derivatives $\partial_{(\xi_1^j, \ldots, \xi_{n+1}^j)}^{\alpha_j}$ on $\Gamma_j$ up to some finite order. Let $T^{\left( d+1 \right)}_m$ be an $n$-linear, $(d+1)$-parameter operator associated to the $(n+1)$-linear form
\[
\Lambda(f_1, \ldots, f_{n+1})=\int_{\vec \xi_1+\ldots + \vec \xi_{n+1}=\vec 0} m(\vec{\xi}_1, \ldots, \vec{\xi}_{n+1}) \hat{f}_1(\vec \xi_1) \cdot \ldots \cdot \hat{f}_{n+1}(\vec \xi_{n+1}) d \vec \xi_1 \ldots \vec \xi_{n+1}.
\]
Then $T^{\left( d+1 \right)}_m$ is a bounded operator from $L^{p_1}(\rr R^{d+1}) \times \ldots \times L^{p_n}(\rr R^{d+1})$ into $L^{p'_{n+1}}(\rr R^{d+1})$ for $(p_1, \ldots, p'_{n+1})$ H\"older tuples satisfying
\[
\frac{1}{p_1}+\ldots +\frac{1}{p_{n+1}}=1, \quad 1<p_1, \ldots, p_n \leq \infty, \quad \frac{1}{n}<p_{n+1}' <\infty, \quad \frac{1}{p_j} < 1-\alpha_j \,\, \forall \, 1 \leq j \leq n+1
\]
where the exponents $\alpha_j \in (0, \frac{1}{2})$ for all $1 \leq j \leq n+1$ are defined by
\begin{equation}
\label{def:exp-alpha_j}
\alpha_j:=\sum_{\substack{1 \leq i_1<\ldots < i_k \leq n+1 \\ i_s =j \text{  for some   } 1\leq  s \leq k}} \theta_{i_1, \ldots, i_k}
\end{equation}
with $ 0 \leq \theta_{i_1, \ldots, i_k} \leq 1$ positive numbers indexed by ordered $k$-tuples such that 
\[
\sum_{1 \leq i_1 < \ldots < i_k \leq n+1} \theta_{i_1, \ldots, i_k}=1.
\]
\end{theorem}

In short, the above theorem says that the tensor product of an $n$-linear operator given by a symbol singular along a $k$-dimensional space (for $1 \leq k<\frac{n+1}{2}$) with a $d$-parameter Mikhlin multiplier has the same range of boundedness as the $k$-singular operator.

\subsection{Sparse Domination and weighted estimates}
 The link between weighted and vector-valued estimates is not fully understood in the particular case of the bilinear Hilbert transform and other similar operators. Nevertheless, the local estimate \eqref{eq:loc-BHT-first} (and its quasi-Banach version) implies the following weighted result:

 \begin{theorem}
 \label{thm:weighted-BHT}
Let $(q_1, q_2, q) \in Range(BHT)$, $q^*:=\min (1, q)$ and $q^{**}=\max(1, q)$. Then for any weights $w_1, w_2, w$ so that $w=w_1 \cdot w_2$, we have
\[
BHT: L^{q_1}(w_1^{q_1}) \times  L^{q_2}(w_2^{q_2}) \to  L^{q}(w^{q}),
\]
provided there exist $0 \leq \theta_1, \theta_2, \theta_3 <1$ with $\theta_1+\theta_2+\theta_3 =1$ and $s_1, s_2, s_3$ satisfying
\[
\frac{1}{q_1}<\frac{1}{s_1} <\frac{1+\theta_1}{2}, \quad \frac{1}{q_2}<\frac{1}{s_2} <\frac{1+\theta_2}{2}, \quad \frac{1}{s_3} < \frac{1}{q^*} -\frac{\theta_1+\theta_2}{2},
\]
 and provided the vector weight condition
\begin{equation*}
\sup_{Q \subset \rr R} \Big(\aver{Q} w_1^{\frac{1}{\frac{1}{q_1}-\frac{1}{s_1}}}\Big)^{\frac{1}{s_1}-\frac{1}{q_1}} \cdot \Big(\aver{Q} w_2^{\frac{1}{\frac{1}{q_2}-\frac{1}{s_2}}}\Big)^{\frac{1}{s_2}-\frac{1}{q_2}} \cdot \Big(\aver{Q} w^{\frac{1}{\frac{1}{s_3}-\frac{1}{\left( q^{**} \right)'}}}\Big)^{\frac{1}{s_3}-\frac{1}{\left(q^{**}\right)'}} <+\infty.
\end{equation*}
 is satisfied. 
 \end{theorem}

We note that, for $q \leq 1$, we have $\frac{1}{\left( q^{**} \right)'}=0$, and the right-most expression is just $ \big( \aver{Q} w^{s_3} \big)^\frac{1}{s_3}$. The reflexive Banach case, i.e. corresponding to $1<q_1, q_2, q<\infty$, was proved in \cite{weighted_BHT}.

Extrapolation provides a means for extending an operator's range of boundedness. Hence, in \cite{extrap-BHT}, a limited range multilinear extrapolation theory was developed. This yields certain weighted vector-valued extensions for $BHT$, but surprisingly, not all those of Theorem \ref{thm:main-thm-BHT}. Instead, all the weighted estimates for $BHT$ from \cite{extrap-BHT} follow directly from Theorem \ref{thm:weighted-BHT}: we start by fixing $0 \leq \theta_1, \theta_2, \theta_3 <1$ with $\theta_1+ \theta_2+\theta_3=1$, and we set
\[
\frac{1}{r_1^-}:=\frac{1+\theta_1}{2}, \quad \frac{1}{r_2^-}:=\frac{1+\theta_2}{2},
\]
while $0 <r_1^+, r_2^+ <\infty$ are so that $\ds \frac{1}{r_1^+}+\frac{1}{r_2^+}=\frac{\theta_1+\theta_2}{2}=1-\frac{1+\theta_3}{2}$. Then we have

\begin{corollary}
\label{cor:weighted-est}
If $(q_1, q_2, q) \in Range(BHT)$ with $1<q_1, q_2<\infty, \, r_1^-<q_1 <r_1^+, \, r_2^-<q_2 <r_2^+$, $w_1, w_2$ and $w=w_1 \cdot w_2$ are weights so that
\[
w_1^{q_1} \in A_{\frac{q_1}{r_1^-}} \cap RH_{\big( \frac{r_1^+}{q_1} \big)'}, \quad w_2^{q_2} \in A_{\frac{q_2}{r_2^-}} \cap RH_{\big( \frac{r_2^+}{q_2} \big)'},
\]
then we have that 
\[
BHT: L^{q_1}(w_1^{q_1}) \times L^{q_2}(w_2^{q_2}) \to L^q(w^q).
\]
\end{corollary}

The above classes of weights are the Muckenhoupt weights $A_p$ and Reverse H\"older $RH_q$ (for more details, see \cite{RFCuervaBook}).

If we consider weighted vector-valued results (which in general are direct consequences of extrapolation), we obtain through the helicoidal method the following:
\begin{corollary}
\label{cor:vv--weighted-est}
Let $0 \leq \theta_1, \theta_2, \theta_3<1$ with $\theta_1+\theta_2+\theta_3=1$, and consider $n$-tuples $(R_1, R_2, R)$, and $(q_1, q_2, q) \in Range(BHT)$ so that $\ds \frac{1}{r_1^j}+\frac{1}{r_2^j}=\frac{1}{r^j}$ for all $1 \leq j \leq n$, $1<q_1, q_2<\infty$ and 
\[
\frac{1}{q_1}, \frac{1}{r_1^j} < \frac{1+\theta_1}{2}, \quad \frac{1}{q_2}, \frac{1}{r_2^j} < \frac{1+\theta_2}{2}, \quad  \frac{1}{q'}, \frac{1}{(r^j)'} < \frac{1+\theta_3}{2}, \quad \forall \, 1 \leq j \leq n.
\]
Let $r_1^-, r_1^+, r_2^-, r_2^+$ be defined by
\[
\frac{1}{r_1^-}:=\frac{1+\theta_1}{2}, \quad \frac{1}{r_2^-}:=\frac{1+\theta_2}{2},
\]
while $0 <r_1^+, r_2^+ <\infty$ are so that $\ds \frac{1}{r_1^+}+\frac{1}{r_2^+}=\frac{\theta_1+\theta_2}{2}=1-\frac{1+\theta_3}{2}$. Then, provided
\[
 r_1^-<q_1 <r_1^+, r_2^-<q_2 <r_2^+,
\]
 we have the vector-valued weighted estimates
\[
BHT: L^{q_1}\big(  \rr R; L^{R_1}(\ii W, \mu) \big)(w_1^{q_1}) \times L^{q_2}\big(  \rr R; L^{R_2}(\ii W, \mu) \big)(w_2^{q_2})  \to L^{q}\big(  \rr R; L^{{R}}(\ii W, \mu) \big)(w^{q}), 
\]
where $w_1, w_2$ and $w$ are weights so that $w=w_1 \cdot w_2$ and
\[
w_1^{q_1} \in A_{\frac{q_1}{r_1^-}} \cap RH_{\big( \frac{r_1^+}{q_1} \big)'}, \quad w_2^{q_2} \in A_{\frac{q_2}{r_2^-}} \cap RH_{\big( \frac{r_2^+}{q_2} \big)'}.
\]
\end{corollary}

If $q_1, q_2, q, R_1, R_2, R$ are as above, unweighted vector valued estimates
\[
BHT: L^{q_1}\big(  \rr R; L^{R_1}(\ii W, \mu) \big) \times L^{q_1}\big(  \rr R; L^{R_2}(\ii W, \mu) \big)  \to L^{q}\big(  \rr R; L^{{R}}(\ii W, \mu) \big) 
\]
exist, as a consequence of Theorem \ref{thm:main-thm-BHT}.

The extrapolation results from \cite{extrap-BHT} require that the $n$-tuples $R_1$, $R_2$ satisfy
\[
 r_1^-< q_1, r_1^j <r_1^+, \quad r_2^-<q_2, r_2^j <r_2^+, \quad \forall \, 1 \leq j \leq n.
\]

In contrast, even though such a condition needs to be true for $q_1$ and $q_2$, we only demand of the $n$-tuples $R_1, R_2$ to verify
\[
 r_1^-<r_1^j , \quad r_2^-<r_2^j , \quad \forall \, 1 \leq j \leq n.
\]

The weighted results of Theorem \ref{thm:weighted-BHT}, Corollary \ref{cor:weighted-est} and \ref{cor:vv--weighted-est} follow from a multiple vector-valued sparse domination. We recall the following definition:
 \begin{definition}
 \label{def-sparse-collection}
Let $0< \eta <1$. A collection $\ic S$ of dyadic intervals is called \emph{$\eta$-sparse} if one can choose pairwise disjoint measurable sets $E_Q \subseteq Q$ with $\vert E_Q \vert \geq \eta \vert Q \vert$ for all $Q \in \ic S$. 
 \end{definition} 

From the localized vector-valued estimates, we deduced in \cite{sparse-hel}
\begin{theorem}
\label{thm:sparse-BHT}
Let $0 \leq \theta_1, \theta_2, \theta_3<1$ with $\theta_1+\theta_2+\theta_3=1$, and consider $n$-tuples $(R_1, R_2, R)$ so that $\ds \frac{1}{r_1^j}+\frac{1}{r_2^j}=\frac{1}{r^j}$, $(r_1^j, r_2^j, r^j) \in Range(BHT)$  for all $1 \leq j \leq n$,  and we have simultaneously
\begin{equation}
\label{eq:cond-r_j-theta}
 \frac{1}{r_1^j} < \frac{1+\theta_1}{2}, \quad  \frac{1}{r_2^j} < \frac{1+\theta_2}{2}, \quad  \frac{1}{(r^j)'} < \frac{1+\theta_3}{2}, \quad \forall 1 \leq j \leq n.
\end{equation}

Let $s_1, s_2, s_3$ be so that
\begin{equation}
\label{eq:condition-on-s}
\frac{1}{s_1} <\frac{1+\theta_1}{2}, \,\frac{1}{s_2} <\frac{1+\theta_2}{2}, \quad \frac{1}{s_3} < \frac{1}{q} -\frac{\theta_1+\theta_2}{2}.
\end{equation}

Then for any vector-valued functions $\vec f, \vec g$ so that $\| \vec f (x, \cdot) \|_{L^{R_1}(\ii W, \mu)}, \| \vec g (x, \cdot) \|_{L^{R_2}(\ii W, \mu)}$ are locally integrable, and any locally $q$-integrable function $v$, we can construct a sparse collection $\ic S$ of dyadic intervals, depending on the functions $\vec f, \vec g$ and $v$, and on the exponents $s_1, s_2, s_3, q$, for which
{\fontsize{10}{10}\begin{align}
\label{eQ:sparse-dom-Lq}
\big\| \| BHT_{\rr P}(\vec f, \vec g)  \|_{L^{R}(\ii W, \mu)}  \cdot v   \big\|_q^q  \lesssim \sum_{Q \in \ic S} &\big( \frac{1}{|Q|} \int_{\rr R} \| \vec f(x, \cdot) \|_{L^{R_1}(\ii W, \mu)}^{s_1} \cdot \ci_{Q} dx \big)^\frac{q}{s_1} \cdot \big( \frac{1}{|Q|} \int_{\rr R} \| \vec g(x, \cdot) \|_{L^{R_2}(\ii W, \mu)}^{s_2} \cdot \ci_{Q} dx \big)^\frac{q}{s_2}  \\ &\cdot \big( \frac{1}{|Q|} \int_{\rr R} |v(x)|^{s_{3}} \cdot \ci_{Q} dx \big)^\frac{q}{s_{3}} \cdot |Q|. \nonumber
\end{align}}
\end{theorem}

We note that the case $q=1$ is equivalent to a sparse domination of the trilinear form $\Lambda_{BHT; \rr P}$. Here $BHT_{\rr P}$ represents a discretized model operator (see the definition \eqref{eq:BHT-model} below), hence there is a fixed dyadic grid associated to the spatial discretization.

\begin{rremark}
The Lebesgue exponents $s_1, s_2, s_3$ and $q$ appearing in the sparse domination result must satisfy condition \eqref{eq:condition-on-s}. We note that $s_3 \in (1, \infty)$ implies  the restriction 
\[
\frac{1}{s_1}+\frac{1}{s_2}<1+\frac{1}{q},
\]
and hence an $L^q$-sparse domination result as in \eqref{eQ:sparse-dom-Lq} valid for all $0<q<\infty$ is only possible if $\frac{1}{s_1}+\frac{1}{s_2}<1$.
\end{rremark}

\smallskip
The sparse domination theorem above implies the following vector-valued Fefferman-Stein type estimate:
\begin{corollary}
\label{corollary-Fefferman-Stein}
If the Lebesgue exponents $q, s_1, s_2$ and the $n$-tuples $R_1, R_2, R$ are as in Theorem \ref{thm:sparse-BHT} (that is, they are conditioned by certain $0\leq \theta_1, \theta_2, \theta_3<1$ with $\theta_1+\theta_2+\theta_3=1$ and also $\ds \frac{1}{s_1}+\frac{1}{s_2}<1+\frac{1}{q}$), then 
\begin{equation}
\label{eq:Fefferman-Stein}
\big\| \| BHT_{\rr P}(\vec f, \vec g)  \|_{L^{R}(\ii W, \mu)} \big\|_q \lesssim \big\| \ic M_{s_1}  \big( \| \vec f(x, \cdot) \|_{L^{R_1}(\ii W, \mu)}  \big)  \cdot  \ic M_{s_2}  \big( \| \vec g(x, \cdot) \|_{L^{R_2}(\ii W, \mu)}  \big)  \big\|_q.
\end{equation}

In particular, if $2\leq r^j_1, r^j_2 \leq \infty$ for all $1 \leq j \leq n$ (i.e., if we are in the locally $L^2$ case), then 
\[
\big\| \| BHT_{\rr P}(\vec f, \vec g)  \|_{L^{R}(\ii W, \mu)} \big\|_q \lesssim \big\| \ic M_{2}  \big( \| \vec f(x, \cdot) \|_{L^{R_1}(\ii W, \mu)}  \big)  \cdot  \ic M_{2}  \big( \| \vec g(x, \cdot) \|_{L^{R_2}(\ii W, \mu)}  \big)  \big\|_q.
\]
Also, if $2 \leq r^j_1 \leq \infty$ for all $1 \leq j \leq n$ and $0<q<2$, then for a certain $\epsilon>0$ we have
\[
\big\| \| BHT_{\rr P}(\vec f, \vec g)  \|_{L^{R}(\ii W, \mu)} \big\|_q \lesssim \big\| \ic M_{2}  \big( \| \vec f(x, \cdot) \|_{L^{R_1}(\ii W, \mu)}  \big)  \cdot  \ic M_{1+\epsilon}  \big( \| \vec g(x, \cdot) \|_{L^{R_2}(\ii W, \mu)}  \big)  \big\|_q.
\]

If $\ic M_{s_1, s_2}$ denotes the bilinear maximal operator
\[
\ic M_{s_1, s_2}(f, g)(x):=\sup_{Q \ni x} \big( \frac{1}{|Q|} \int_Q |f(y)|^{s_1} dy  \big)^\frac{1}{s_1} \big( \frac{1}{|Q|} \int_Q |g(y)|^{s_2} dy  \big)^\frac{1}{s_2},
\]
we can also formulate the Fefferman-Stein inequality as:  
\[\big\| \| BHT_{\rr P}(\vec f, \vec g)  \|_{L^{R}(\ii W, \mu)} \big\|_q \lesssim \big\| \ic M_{s_1, s_2}  \big( \| \vec f(x, \cdot) \|_{L^{R_1}(\ii W, \mu)}, \| \vec g(x, \cdot) \|_{L^{R_2}(\ii W, \mu)}  \big)  \big\|_q,
\]
for any $0 < q< \infty$, $s_1, s_2$ and $n$-tuples $R_1, R_2, R$ as in Theorem \ref{thm:sparse-BHT}.
\end{corollary}

Such Fefferman-Stein inequalities imply at once the $L^{q_1} \times L^{q_2} \to L^q$ boundedness of the (multiple vector-valued) bilinear Hilbert transform operator, provided $s_1<q_1$ and $s_2<q_2$. From here, the range of boundedness is also visible: $ \frac{1}{q}=\frac{1}{q_1}+\frac{1}{q_2}< \frac{1}{s_1}+\frac{1}{s_2}<1+\frac{\theta_1+\theta_2}{2}<\frac{3}{2}$, and hence $\frac{2}{3}<q$. 

Similar multiple vector-valued and weighted results hold also for paraproducts (multilinear Mikhlin multipliers whose symbol is singular at a point). In the scalar case, Fefferman-Stein inequalities were obtained in \cite{CoifmanMeyer-Commutators}, and a pointwise sparse domination in \cite{Nazarov-Lerner-DyadicCalculus}. For a comprehensive discussion of the general multiple vector-valued case (which includes $L^\infty$ spaces), see Section 6.4 of \cite{sparse-hel}.

\smallskip
The paper is organized as follows: in Section \ref{sec:BHTsurvey} we present an overview of the study of the bilinear Hilbert transform operator, elaborating from the local $L^2$ case of \cite{LaceyThieleBHTp>2} to the localization principle from \cite{vv_BHT}, discussing on the way connections with the outer measure theory of \cite{outer_measures}. In Section \ref{sec:var-Carleson} we illustrate the localization principle for the variational Carleson operator of \cite{variational_Carleson}, which yields a significant simplification of the proof and a Fefferman-Stein inequality similar to Theorem \ref{thm:Fefferman-Stein-BHT}. Sections \ref{sec:method-proof} and \ref{sec:sparse-dom} are devoted to presenting the methods of the proof, for the vector-valued estimates, and sparse vector-valued estimates, respectively. 

\subsection*{Acknowledgments}

C. B. wishes to express her gratitude to Fr\'ed\'eric Bernicot and Teresa Luque for several discussions on sparse domination and weighted theory.

C. B. was partially supported by NSF grant DMS 1500262  and ERC project FAnFArE no. 637510. C. M. was partially supported by NSF grant DMS 1500262; he also acknowledges partial support from a grant from the Ministry of Research and Innovation of Romania, CNCS--UEFISCDI, project PN--III--P4--ID--PCE--2016--0823 within PNCDI--III.

During the Spring Semester of 2017, C. M. was a member of the MSRI in Berkeley, as part of the Program in Harmonic Analysis, and during the Fall semester of 2017, he was visiting the Mathematics Department of the Universit\'e Paris-Sud Orsay, as a Simons Fellow. He is very grateful to both institutions for their hospitality, and to the Simons Foundation for their generous support.

\section{A survey of the bilinear Hilbert transform operator} 
\label{sec:BHTsurvey}

We recall a few properties of the bilinear Hilbert transform operator, following the presentation in \cite{multilinear_harmonic}. In its multiplier form, it can be represented as
\[
BHT(f, g)(x):= \int_{\xi_1<\xi_2} \hat{f}(\xi_1) \hat{g}(\xi_2) e^{2 \pi i x \left( \xi_1+\xi_2  \right)} d \xi_1 d \xi_2.
\]

In consequence, $BHT$ can be written as a superposition of discrete model operators, each of which allows for a decomposition over a rank-1 family $\rr P$ of tri-tiles:
\begin{equation}
\label{eq:BHT-model}
BHT_{\rr{P}}(f,g)(x)=\sum_{P \in \rr{P}} \frac{1}{|I_P|^{1/2}} \langle f, \phi_{P_1}^1 \rangle  \langle g, \phi_{P_2}^2 \rangle \phi^3_{P_3}(x).
\end{equation}

The classical proof consists in using a careful stopping time in order to decompose the collection $\rr P$ into subcollections that are easier to estimate (\emph{trees} or \emph{Carleson-Fefferman sets}, as referred in \cite{LaceyThieleBHTp>2}). Later on, in \cite{biest}, the \emph{bi-est} operator defined by
\begin{equation}
\label{def:bi-est}
(f_1, f_2, f_3) \mapsto \int_{\xi_1< \xi_2<\xi_3} \hat{f}_1(\xi_1)\,  \hat{f}_2(\xi_2) \,  \hat{f}_3(\xi_3) \, e^{2 \pi i x \left( \xi_1+\xi_2+\xi_3  \right)} d\xi_1 d \xi_2 d \xi_3
\end{equation}
was investigated. The problem can be reduced to the study of a local composition of two bilinear Hilbert transforms, and a global estimate for the trilinear form $\Lambda_{BHT; \rr P}$ provides an efficient way of keeping track of the local compositions. Such a global estimate is the \emph{generic size and energy} estimate (formulated in \eqref{eq:gen-size-energy}), which can summed up in the following way: the trilinear form associated to the $BHT_{\rr P}$ operator satisfies
\begin{equation}
\label{eq:BHT-en-size}
\vert \Lambda_{BHT; \rr P} (f, g, h) \vert \lesssim \big( \sssize_{\rr P} \,f \big)^{\theta_1} \cdot \big( \sssize_{\rr P}\, g \big)^{\theta_2} \cdot  \big( \sssize_{\rr P} \, h \big)^{\theta_3} \cdot  \|f\|_2^{1-\theta_1} \cdot  \|g\|_2^{1-\theta_2} \cdot   \|h\|_2^{1-\theta_3},
\end{equation}
where $0 \leq \theta_1, \theta_2, \theta_3<1$ and $\theta_1+\theta_2+\theta_3=1$.

The quantity ``$\sssize_{ \rr P} \, f $" is similar to a maximal operator:
{\fontsize{10}{10}\begin{equation}
\label{def:sssize}
\sssize_{\rr P} \, f :=\sup_{P \in \rr P} \frac{1}{\vert I_P  \vert} \int_{\rr R} \vert f(x)  \vert \cdot \ci_{I_P}^{M}(x) dx, \quad \sssize^s_{\rr P} \, f := \sup_{P \in \rr P} \big( \frac{1}{\vert I_P  \vert} \int_{\rr R} \vert f(x)  \vert^s \cdot \ci_{I_P}^{M}(x) dx \big)^\frac{1}{s}.
\end{equation}}
Since it involves $L^1$ (or $L^s$) averages, it represents an $L^1$ quantity (respectively $L^s$). On the other hand, at the level of the collection of tiles $\rr P$, it should be regarded as an $\ell^\infty$ quantity.

With the purpose of reassembling $\|f\|_p$ out of an $L^1$ quantity and an $L^2$ norm, it is natural to work with restricted type functions: that is, functions that are bounded above by characteristic functions of finite-measure sets. We recall that interpolation theory (more precisely \emph{generalized restricted type interpolation}, as it was called in \cite{wave_packet}) reduces the general case to that of restricted type functions, a simplification often used in the field.

In tackling the $L^{q_1} \times L^{q_2} \to L^q$ boundedness of the bilinear Hilbert transform, there are two cases that are worth a closer look:
\begin{enumerate}[label=(\alph*), ref=\alph*]
\item \label{item-introd-local-L2} the ``local $L^2$" case, when $2<q_1, q_2, q'<\infty$ 
\item\label{item-introd-whole-range} the general case, when $1<q_1, q_2 \leq \infty, \, \frac{2}{3}<q<\infty$.
\end{enumerate}
The first case was treated in \cite{LaceyThieleBHTp>2}, and the second in \cite{initial_BHT_paper}; the methods that apply to situation \eqref{item-introd-local-L2} needed an extra push in order to deal with \eqref{item-introd-whole-range}, and one of the additional features was interpolation between the estimates for the bilinear Hilbert transform and those for its adjoints. As we sketch the proof of the boundedness of $BHT$, we will present the two cases separately, as well as a third situation, when the interpolation of the adjoint operators can be avoided by using a localization technique and additional stopping times.

\subsection{Generic ``size and energy" estimates}
\label{sec:generic-size-energy-est}
Regarded as a bilinear multiplier, the bilinear Hilbert transform is defined by the symbol $sgn (\xi_1-\xi_2)$, or equivalently (up to adding $f(x) \cdot g(x)$ and multiplying by a constant) by the symbol $\one_{\lbrace \xi_1<\xi_2 \rbrace }$. Hence the model operator defined in \eqref{eq:BHT-model} is indexed after a family $\rr P$ of tri-tiles $P=(P_1, P_2, P_3)$ which is said to have rank $1$. More exactly, each $P_j=I_P \times \omega_{P_j}$ is a rectangle of area $1$, where $I_P$ is a dyadic interval, while $\omega_{P_j}$ is in a (possibly) translated dyadic grid, which is considered to be fixed. The \emph{rank $1$} condition refers to the fact that once one of the $\omega_{P_{j_0}}$ is fixed, we know the position of the other $\omega_{P_j}$s as well. In fact, $\ds \lbrace   \omega_{P_1} \times \omega_{P_2} \rbrace_{P \in \rr P}$ represents a Whitney decomposition of the frequency region $\lbrace \xi_1<\xi_2  \rbrace$, and in consequence $\xi_1 \in \omega_{P_1}, \xi_2 \in \omega_{P_2}$, while $\xi_1+\xi_2 \in \omega_{P_3}$.

An order relation can be introduced on the collection of time-frequency tiles, which will further allow us to define \emph{tree structures} on the rank-$1$ collection $\rr P$ of tri-tiles.

\begin{definition}
If $P=I_P \times \omega_P$ and $P'=I_{P'}\times \omega_{P'}$ are two tiles, we say that
\begin{itemize}
\item  $P' \leq P$, if $I_{P'} \subseteq I_P$ and $\omega_P \subseteq 3 \omega_{P'}$;
\item  $P' \lesssim P$, if $I_{P'} \subseteq I_P$ and $\omega_P \subseteq 100 C_0 \omega_{P'}$, for a fixed constant $C_0$;
\item  $P' \lesssim ' P$, if $P' \lesssim P$ but $P' \nleq P$.
\end{itemize}

If $j \in \lbrace 1, 2, 3 \rbrace$, a \emph{$j$-tree} with top $P_T=(P_{T,1}, P_{T,2}, P_{T,3} )$ is a subcollection $T \subseteq \rr P$ of tri-tiles such that
\[
P_j \leq P_{T,j} \quad \text{for all} \quad P \in T.
\]

As a consequence of the rank-$1$ property, if $T$is a $j$-tree and $\jmath \in \lbrace 1, 2, 3 \rbrace, \jmath \neq j$, we have that $P_{\jmath} \lesssim' P_{T, \jmath}$ for all $P \in T$. Sometimes, we say that $T$ is a \emph{$\jmath$-lacunary} tree.
\end{definition}

A key step in proving the boundedness of $BHT$ in \cite{initial_BHT_paper} consists in organizing $\rr P$ into collections of trees (which can be thought of as \emph{forests}, in the spirit of \cite{fefferman-Carleson}), according to their sizes. This motivates the concept of \emph{size} associated to a subcollection $\rr P' \subseteq \rr P$: for any $1 \leq j \leq 3$,
\begin{equation}
\label{eq:def-size}
\ssize_{\rr P'}^j (f):=\sup_{T \subseteq \rr P'}  \big( \frac{1}{|I_T|} \sum_{P \in T} | \langle f, \phi_{P_j}^j    \rangle|^2    \big)^\frac{1}{2},
\end{equation}
where $T$ ranges over all $j$-lacunary trees $T \subseteq \rr P'$ (or equivalently, over all trees $T \subseteq \rr P'$ which are $i$-trees for some $1 \leq i \leq 3, i \neq j$).

Due to John-Nirenberg inequality, the $L^2$-defined $\ssize_{\rr P'}^j (f)$ is equivalent to an $L^{1, \infty}$ quantity
 \[
 \ssize_{\rr P'}^j (f):=\sup_{T \subseteq \rr P'} \frac{1}{|I_T|^{1/2}} \big\|   \big( \sum_{P \in T} \frac{| \langle f, \phi_{P_j}^j    \rangle |^2}{|I_P|} \one_{I_P}  \big)^\frac{1}{2} \big\|_2 \sim \sup_{T \subseteq \rr P'} \frac{1}{|I_T|} \big\|   \big( \sum_{P \in T} \frac{| \langle f, \phi_{P_j}^j    \rangle |^2}{|I_P|} \one_{I_P}  \big)^\frac{1}{2} \big\|_{1, \infty},
 \]
where the sup is considered as before over all $j$-lacunary trees $T \subseteq \rr P'$. Because of this, and due to the weak type $(1, 1)$ boundedness of the square function, we have that
\[
 \ssize_{\rr P'}^j (f) \lesssim \sssize_{\rr P'}(f)=\sup_{P \in \rr P'} \frac{1}{\vert I_P  \vert} \int_{\rr R} \vert f(x)  \vert \cdot \ci_{I_P}^{M}(x) dx,
\]
hence $\ssize_{\rr P'}^j (f)$, acting on functions, is indeed an $L^1$ quantity.

In \cite{biest}, where local compositions of two $BHT$-like operators were investigated, a concept dual to the $\ssize$ was introduced: given a subcollection $\rr P' \subseteq \rr P$, 
\begin{equation}
\label{eq:def-energy}
\eenergy_{\rr P'}^j (f):=\sup_{n \in \rr Z} \sup_{\mathbf{T}} 2^n \big( \sum_{T \in \mathbf{T}} |I_T|  \big)^\frac{1}{2}, 
\end{equation}
where $\mathbf{T}$ ranges over all subcollections of $j$-strongly disjoint trees $T \subseteq \rr P'$ (which are $i$-trees for some $1 \leq i \leq 3, i \neq j$)  with the property that 
\[
\big( \frac{1}{|I_T|} \sum_{P \in T} | \langle f, \phi_{P_j}^j    \rangle|^2    \big)^\frac{1}{2} \geq 2^{n-1},
\]
and at the same time, for any subtree $T' \subseteq T, T \in \mathbf{T}$, we have 
\[
\big( \frac{1}{|I_{T'}|} \sum_{P \in T'} | \langle f, \phi_{P_j}^j    \rangle|^2    \big)^\frac{1}{2} \leq 2^n.
\]

As the name suggests, the $\eenergy$ is an $L^2$ quantity and in fact we have that
\[
\eenergy_{\rr P'}^j (f) \lesssim \|f\|_2.
\]

On this account, the estimate \eqref{eq:BHT-en-size} of the trilinear form associated with $BHT$ is a consequence of the more general
\begin{align}
\label{eq:gen-size-energy}
\vert \Lambda_{BHT;\rr P}(f, g, h)\vert  & \lesssim \big( \ssize_{\rr P}^1\, f \big)^{\theta_1} \cdot \big( \ssize_{\rr P}^2 \, g \big)^{\theta_2} \cdot \big( \ssize_{\rr P}^3 \, h \big)^{\theta_3} \\
& \cdot \big( \eenergy_{\rr P}^1\, f  \big)^{1-\theta_1} \cdot \big( \eenergy_{\rr P}^2 \, g  \big)^{1-\theta_2} \cdot \big( \eenergy_{\rr P}^3 h \, \big)^{1-\theta_3},\nonumber
\end{align}
which holds true whenever $0 \leq \theta_1, \theta_2, \theta_3 <1$ with $\theta_1+ \theta_2+\theta_3 =1$.

Strictly speaking, an earlier version of the energy was introduced in \cite{biest-Walsh}, where a Walsh model of the \emph{bi-est} operator was studied. This energy was defined as
\[
``\eenergy_{\rr P}(f)":=\sup_{\mathbf{T}} \big( \sum_{T \in \mathbf{T}} \sum_{P \in T} \vert  \langle f, \phi_P \rangle\vert^2 \big)^\frac{1}{2},
\]
(where supremum is taken over collections of strongly disjoint trees), and clearly it is larger than $\eenergy_{\rr P}(f)$ from \eqref{eq:def-energy}. In the Walsh case, disjointness of tiles immediately entails orthogonality and hence $``\eenergy_{\rr P}(f)" \lesssim \|f\|_2$. But this need not be the case in the Fourier setting: $``\eenergy_{\rr P}(f)"$ it is not necessarily bounded above by $\|f\|_2$ (see \cite[Chapter~5]{wave_packet}). This issue led to the definition of energy as in \eqref{eq:def-energy}. The lower bound condition which needs to be satisfied by each of the considered trees makes the $\eenergy$ a weak-$L^2$ rather than an $L^2$ quantity.

 \smallskip
Now we recall how the general size and energy estimate taking the form of \eqref{eq:BHT-en-size} or \eqref{eq:gen-size-energy} implies the boundedness of $BHT$, when working with restricted type functions.

\subsubsection*{\textbf{The local $L^2$ case \eqref{item-introd-local-L2}:}}
Given that $\eenergy_{\rr P'}^j (f) \lesssim \|f\|_2$, we obtain for functions $f, g, h$ satisfying $|f|\leq \one_F, |g| \leq \one_G$ and $|h| \leq \one_H$
\begin{equation}
\label{eq:discard-size}
\vert \Lambda_{BHT;\rr P}(f, g, h)\vert   \lesssim  |F|^\frac{1-\theta_1}{2} \, |G|^\frac{1-\theta_2}{2} \, |H|^\frac{1-\theta_3}{2}.
\end{equation}
Here we can discard the $\ssize_{\rr P}$ part of \eqref{eq:gen-size-energy}, since the functions are all bounded above by characteristic functions, and in consequence $\ssize_{\rr P} \lesssim 1$. Restricted type interpolation then implies the boundedness of $BHT$ in the ``local $L^2$" range: $BHT:L^p \times L^q \to L^s$, whenever $2< p, q, s' <\infty$ with $\frac{1}{p}+\frac{1}{q}=\frac{1}{s}$.

\subsubsection*{\textbf{Whole range: case \eqref{item-introd-whole-range}:}}
Thanks to restricted weak type interpolation (see \cite{multilinearMTT}, \cite{wave_packet}, \cite{vv_BHT}, \cite{multilinear_harmonic}), the framework can be simplified to the following: given $F, G, H$ measurable sets of finite measure, if suffices to find $H' \subseteq H$ major subset (this means $|H'|>\dfrac{|H|}{2}$) so that
\begin{equation}
\label{eq:tril-form-BHT}
\vert \Lambda_{\rr P}(f, g, h) \vert \lesssim |F|^{a_1} \cdot |G|^{a_2} \cdot  |H|^{a_3}
\end{equation}
for all functions $f, g$ and $h$ satisfying $|f| \leq \one_F, |g| \leq \one_G$ and $|h| \leq \one_{H'}$, and for various parameters $a_1, a_2, a_3$ with the property that $a_1+a_2+a_3=1, \, a_1, a_2 \in (0,1), a_3 \in (-1, 1)$. The numbers $a_1, a_2, a_3$ will stand for reciprocals of Lebesgue exponents $\frac{1}{p}, \frac{1}{q}, \frac{1}{s'}$; since $\frac{2}{3} < s<\infty$, $\frac{1}{s'}$ lies inside the interval $\big(-\frac{1}{2}, 1 \big)$.

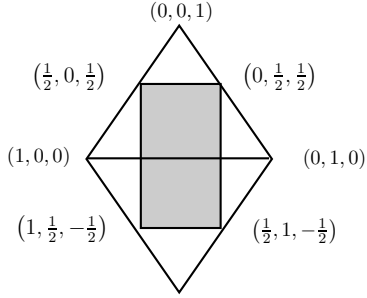
\begin{figure}
\begin{center}
\psscalebox{.7 .7} 
{
\begin{pspicture}(3.5 ,-4.521844)(11.520909,2)
\definecolor{colour0}{rgb}{0.8,0.8,0.8}
\psframe[linecolor=black, linewidth=0.04, fillstyle=solid,fillcolor=colour0, dimen=outer](8.877778,-0.48926708)(7.322222,-3.2670448)
\rput[bl](4.672727,-2.0963378){ $\left(1, 0, 0 \right)$}
\rput[bl](4.9636364,-3.5145197){\large $\left( 1, \frac{1}{2}, -\frac{1}{2}\right)$}
\rput[bl](7.509091,0.66729856){$\left( 0, 0, 1\right)$}
\rput[bl](10.290909,-2.1327014){ $\left( 0, 1, 0 \right)$}
\rput[bl](9.309091,-3.5327015){ $\left( \frac{1}{2}, 1, -\frac{1}{2}\right)$}
\pscustom[linecolor=black, linewidth=0.04]
{
\newpath
\moveto(1.2857143,-5.7638702)
}
\pscustom[linecolor=black, linewidth=0.04]
{
\newpath
\moveto(5.0,-5.7638702)
}
\psline[linecolor=black, linewidth=0.04](6.3333335,-1.9226004)(9.777778,-1.9226004)
\psdiamond[linecolor=black, linewidth=0.04, dimen=outer](8.071428,-1.9352988)(1.7857143,2.5714285)
\rput[bl](5.2636366,-0.6145196){\large $\left(  \frac{1}{2}, 0, \frac{1}{2}\right)$}
\rput[bl](9.263637,-0.6145196){\large $\left( 0, \frac{1}{2}, \frac{1}{2}\right)$}
\end{pspicture}}
\end{center}
\caption{Range of Lebesgue exponents for $BHT$ obtained directly from the generic size-energy estimate}
\label{fig1}
\end{figure}

\smallskip

The major subset $H'$ is constructed by removing a certain ``small" part of $H$ where $\Lambda_{\rr P}(f, g, h)$ is difficult to control: if
\begin{equation}
\label{eq:exceptional-set}
\Omega:=\big\lbrace x: \ic M(\one_F) (x) > C \frac{|F|}{|H|} \big\rbrace \cup \big\lbrace x: \ic M(\one_G) (x) > C \frac{|G|}{|H|} \big\rbrace,
\end{equation}
then we set $H':=H \setminus \Omega$, which is a major subset if $C$ is chosen large enough.

As a consequence, the ``sizes" and ``energies" can be estimated as follows:
\[
\ssize_{\rr P}^1 (f) \lesssim \min \big( \frac{|F|}{|H|}, 1   \big), \quad \ssize_{\rr P}^2 (g) \lesssim \min \big( \frac{|G|}{|H|}, 1   \big), \quad \ssize_{\rr P}^3 (h) \lesssim 1, 
\] 
\begin{flushleft}
\[
\text{   and also} \quad \eenergy_{\rr P}^1 (f) \lesssim |F|^\frac{1}{2}, \quad \eenergy_{\rr P}^2 (g) \lesssim |G|^\frac{1}{2}, \quad \eenergy_{\rr P}^3 (h) \lesssim |H|^\frac{1}{2}. 
\]
\end{flushleft}

All of the above imply, for any $0 \leq a, b \leq 1$, and any $0\leq \theta_1, \theta_2, \theta_3<1$ with $\theta_1+\theta_2+\theta_3=1$, that 
\begin{align*}
\vert  \Lambda_{\rr P}(f, g, h)\vert  & \lesssim \Big( \frac{|F|}{|H|} \Big)^{a \theta_1} \cdot \Big( \frac{|G|}{|H|} \Big)^{b \theta_2} \cdot |F|^\frac{1-\theta_1}{2}\cdot  |G|^\frac{1-\theta_2}{2} |H|^\frac{1-\theta_3}{2}  \\ &  = |F|^{a \theta_1+\frac{1-\theta_1}{2 }} \cdot  |G|^{b \theta_2+\frac{1-\theta_2}{2 }} \cdot  |H|^{\frac{1-\theta_3}{2 }-a \theta_1-b \theta_2}.
\end{align*}

\smallskip

This is similar to the estimate \eqref{eq:tril-form-BHT}, which was our goal. However, if we set 
\[
a_1=a \theta_1+\frac{1-\theta_1}{2 }, \quad a_2=b \theta_2+\frac{1-\theta_2}{2 }, \quad a_3=\frac{1-\theta_3}{2 }-a \theta_1-b \theta_2, 
\]
the conditions we had on $\theta_j, a$ and $b$ imply that 
\begin{equation}
\label{eq:silly-eq}
 \vert a_1-a_2 \vert <\frac{1}{2},
\end{equation}
and we only obtain $L^p$ estimates for the Lebesgue exponents presented in Figure \ref{fig1}. A routine check shows that we can obtain $(a_1,a_2, a_3)$ arbitrarily close to any of the points
\[
\big(\frac{1}{2}, 0, \frac{1}{2} \big), \quad \big(1, \frac{1}{2}, - \frac{1}{2} \big),  \quad \big(\frac{1}{2}, 1, -\frac{1}{2} \big), \quad \big(0, \frac{1}{2}, \frac{1}{2} \big), 
\]
whose convex hull is the region in Figure \ref{fig1}. However, it is impossible to have $(a_1, a_2, a_3)$ close to the points $(1, 0, 0)$ or $(0, 1, 0)$: if $a_1 \sim 1$, then necessarily $a \sim 1, \theta_1 \sim 1$, which implies $\theta_2 \sim 0$ and in consequence $\ds a_2 \sim \frac{1}{2}$.

In order to obtain the rest of the $L^p$ estimates (see Figure \ref{fig:rangeBHT}), we have to use restricted weak type estimates for the adjoints of $BHT$, and then to interpolate carefully between them.
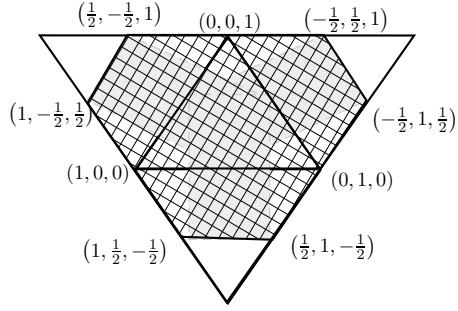
\begin{figure}
\psscalebox{.7 .7} 
{
\begin{pspicture}(0,-5)(12.303637,2)
\definecolor{colour0}{rgb}{0.8,0.8,0.8}
\definecolor{colour1}{rgb}{0.92941177,0.92941177,0.92941177}
\definecolor{colour2}{rgb}{0.93333334,0.92941177,0.92941177}
\pspolygon[linecolor=colour0, linewidth=0.024, linestyle=dashed, dash=0.17638889cm 0.10583334cm, fillstyle=solid,fillcolor=colour1](7.2972975,-0.55923706)(8.054054,-1.8565344)(10.7027025,-0.6132911)(9.8918915,0.62995213)(9.8918915,0.62995213)
\rput[bl](4.8727274,-2.1963377){ $\left(1, 0, 0 \right)$}
\rput[bl](5.163636,-3.7145195){ $\left( 1, \frac{1}{2}, -\frac{1}{2}\right)$}
\rput[bl](7.509091,0.66729856){$\left( 0, 0, 1\right)$}
\rput[bl](9.890909,-2.3327014){ $\left( 0, 1, 0 \right)$}
\rput[bl](9.109091,-3.6327014){ $\left( \frac{1}{2}, 1, -\frac{1}{2}\right)$}
\pscustom[linecolor=black, linewidth=0.04]
{
\newpath
\moveto(1.2857143,-5.7638702)
}
\pscustom[linecolor=black, linewidth=0.04]
{
\newpath
\moveto(5.0,-5.7638702)
}
\psline[linecolor=black, linewidth=0.04](6.3333335,-1.9226004)(9.777778,-1.9226004)
\psdiamond[linecolor=black, linewidth=0.04, dimen=outer](8.071428,-1.9352988)(1.7857143,2.5714285)
\pspolygon[linecolor=black, linewidth=0.04](4.5,0.62184405)(11.609756,0.619405)(8.04878,-4.453766)
\pspolygon[linecolor=colour0, linewidth=0.024, linestyle=dashed, dash=0.17638889cm 0.10583334cm, fillstyle=solid,fillcolor=colour1](8.810811,-0.55923706)(8.054054,-1.9105884)(5.4054055,-0.6132911)(6.1621623,0.57589805)(6.1621623,0.62995213)
\psframe[linecolor=colour0, linewidth=0.024, linestyle=dashed, dash=0.17638889cm 0.10583334cm, fillstyle=solid,fillcolor=colour2, dimen=outer](8.877778,-0.48926708)(7.322222,-3.2670448)
\pspolygon[linecolor=black, linewidth=0.04, fillstyle=crosshatch, hatchwidth=0.008, hatchangle=60, hatchsep=0.2212](5.4054055,-0.66734517)(7.189189,-3.2078857)(8.918919,-3.2619398)(10.7027025,-0.6132911)(9.8918915,0.62995213)(6.1621623,0.62995213)
\pstriangle[linecolor=black, linewidth=0.04, dimen=outer](8.075,-1.928156)(3.55,2.55)
\rput[bl](3.7636364,-1.1145196){ $\small\left( 1, -\frac{1}{2}, \frac{1}{2}\right)$}
\rput[bl](10.663636,-1.2145196){ $ \small\left( -\frac{1}{2}, 1, \frac{1}{2}\right)$}
\rput[bl](9.363636,0.6854804){ $\tiny \left( -\frac{1}{2}, \frac{1}{2}, 1\right)$}
\rput[bl](5.0636363,0.7854804){ $\small \left( \frac{1}{2}, -\frac{1}{2}, 1 \right)$}
\end{pspicture}
}
\caption{Range of Lebesgue exponents for $BHT$ and its adjoints $BHT^{*, 1}$ and $BHT^{*,2}$}
\label{fig-adjoints-interp}
\end{figure}
\smallskip

The adjoint $BHT^{*,1}$, which is defined by
\[
\Lambda_{BHT; \rr P}(f, g, h)= \Lambda_{BHT^{*,1}; \rr P}(h, g, f)=\int_{\rr R} BHT^{*,1}_{\rr P}(h, g)(x) f(x) dx,
\]
is bounded whenever the Lebesgue exponents $\big(\frac{1}{p}, \frac{1}{q}, \frac{1}{s'} \big)$ are located in the convex hull of 
\[
\big( \frac{1}{2}, \frac{1}{2}, 0 \big), \quad \big( -\frac{1}{2}, 1, \frac{1}{2}\big), \quad \big( -\frac{1}{2}, \frac{1}{2}, 1 \big), \quad \big( \frac{1}{2}, 0, \frac{1}{2}\big).
\]

Similarly, $BHT^{*,2}$ is bounded if $\big(\frac{1}{p}, \frac{1}{q}, \frac{1}{s'} \big)$ lie inside the convex hull of 
\[
\big( \frac{1}{2}, \frac{1}{2}, 0 \big), \quad \big( 1, -\frac{1}{2}, \frac{1}{2}\big), \quad \big( \frac{1}{2}, -\frac{1}{2}, 1 \big), \quad \big( 0, \frac{1}{2}, \frac{1}{2}\big).
\]

As a result, the trilinear form is bounded whenever $\big(\frac{1}{p}, \frac{1}{q}, \frac{1}{s'} \big)$ are contained in the convex hull of 
\[
\big( 1, -\frac{1}{2}, \frac{1}{2}, 0 \big), \quad \big( 1, \frac{1}{2}, -\frac{1}{2}\big), \quad \big( \frac{1}{2}, 1, -\frac{1}{2}\big), \quad \big(- \frac{1}{2}, 1, \frac{1}{2}\big), \quad \big(- \frac{1}{2}, \frac{1}{2}, 1\big), \quad \big(\frac{1}{2}, - \frac{1}{2}, 1\big),
\]
which includes $Range(BHT)$ (compare Figure \ref{fig:rangeBHT} to Figure \ref{fig-adjoints-interp}).

\subsection{Localization: the helicoidal method at level $0$}
\label{sec:localization}
We now take a look at the localization principle of \cite{vv_BHT}, which proves to be relevant even in the scalar case: a consequence which will be discussed in Section \ref{sec:BHT-avoid-interpolation} is that all the known $L^p$ estimates for $BHT$ can be obtained without using interpolation of adjoint operators.
\smallskip
This localization technique was first designed in \cite{myphdthesis} for proving the Rubio de Francia inequality for iterated Fourier integrals of Section \ref{sec:RF-for-BHT}. Later it became clear that we can capitalize on the local estimates if we also take into account the operatorial norm, and this led to the helicoidal method. In order to make these ideas clear, we need to introduce the notion of a localized size.

\begin{definition}
If $I_0$ is a fixed dyadic interval and $\ds \rr P(I_0):= \lbrace P \in \rr P: I_P \subseteq I_0   \rbrace $, the localized size is defined by
\[
\sssize_{\rr P(I_0)} \, f :=\max \big( \frac{1}{\vert I_0  \vert} \int_{\rr R} \vert f(x)  \vert \cdot \ci_{I_0}^{M}(x) dx ,\sup_{P \in \rr P(I_0)} \frac{1}{\vert I_P  \vert} \int_{\rr R} \vert f(x)  \vert \cdot \ci_{I_P}^{M}(x) dx\big).
\]
The difference between this and formula \eqref{def:sssize} is that here we take into account the average over the interval $I_0$ as well.
\end{definition}

The generic size and energy estimate \eqref{eq:gen-size-energy} is re-interpreted locally in the form of the following lemma:

\begin{lemma}
\label{lemma-local-gen-size-energy}
Let $I_0$ be a dyadic interval, $F, G, H$ sets of finite measure, $\rr P$ a rank-$1$ collection of tri-tiles and $\ds \rr P(I_0):= \lbrace P \in \rr P: I_P \subseteq I_0   \rbrace$. Then we have
\begin{equation}
\label{eq:local-gen-size-energy}
\vert \Lambda_{\rr P (I_0)} (f, g, h)  \vert \lesssim \big(  \sssize _{\rr P(I_0)} \one_F  \big)^\frac{1+\theta_1}{2} \cdot \big(  \sssize _{\rr P(I_0)} \one_G  \big)^\frac{1+\theta_2}{2} \cdot \big(  \sssize _{\rr P(I_0)} \one_H  \big)^\frac{1+\theta_3}{2} \cdot |I_0|,
\end{equation}
for any functions $f, g, h$ satisfying $|f| \leq \one_F, |g| \leq \one_G, |h| \leq \one_H$, and any $0 \leq \theta_1, \theta_2, \theta_3 <1$, with $\theta_1+\theta_2 +\theta_3 =1$.
\begin{proof}
Since all the time-frequency tri-tiles have their spatial information concentrated inside $I_0$ (the $I_P \subseteq I_0$ assumption), we would expect to have a better estimate for the energy: $\ds \eenergy_{\rr P(I_0)}^1(f) \lesssim \| f \cdot \ci_{I_0}\|_2 $. This is indeed the case, as we will see shortly (although formulated in a different way, this represented a central estimate in \cite{multilinearMTT}). Such an inequality, together with \eqref{eq:gen-size-energy}, immediately imply the local estimate above since $\frac{1-\theta_1}{2}+\frac{1-\theta_2}{2} +\frac{1-\theta_3}{2}=1$ and
\[
\eenergy_{\rr P(I_0)}^1(f) \lesssim \| f \cdot \ci_{I_0}\|_2 \lesssim \| \one_F \cdot \ci_{I_0} \|_2 \lesssim  \big( \sssize_{\rr P(I_0)} \one_F \big)^{\frac{1}{2}} \cdot |I_0|^\frac{1}{2}.
\] 

We outline the ideas behind the estimate $\ds \eenergy_{\rr P(I_0)}(f) \lesssim \|f \cdot \ci_{I_0}\|_2$, as they appeared in \cite{multilinearMTT}. We pick $n \in \rr Z$ and $\mathbf{T}$ a collection of disjoint trees as in the energy definition from \eqref{eq:def-energy}. Then we have 
\[
\big(\eenergy_{\rr P(I_0)}(f) \big)^2 \sim 2^{2n} \, \sum_{T \in \mathbf{T}} |I_T|,
\]
and it will be enough to show 
\[
\sum_{T \in \mathbf{T}} \sum_{P \in T} |\langle f, \phi_P \rangle|^2 \lesssim 2^{n} \big( \sum_{T \in \mathbf{T}} |I_T|  \big)^\frac{1}{2} \, \| f \cdot \ci_{I_0}\|_2.
\]
Following a $T\,T^*$-like argument, this reduces to proving
\begin{equation}
\label{eq:ineq-TT-star}
\| \big( \sum_{T \in \mathbf{T}} \sum_{P \in T} \langle f, \phi_P \rangle \phi_P   \big) \, \ci_{I_0}^{-\frac{N}{2}}    \|_2  \lesssim 2^{n} \big( \sum_{T \in \mathbf{T}} |I_T|  \big)^\frac{1}{2}.
\end{equation}
Note that on the left hand side we are multiplying by $\ci_{I_0}^{-\frac{N}{2}}$, a function which grows like $\big(\frac{\dist(x, I_0)}{|I_0|} \big)^{N}$ away from $I_0$. In the classical case, when we want to prove $\eenergy_{\rr P}(f) \lesssim \|f\|_2$, such an expression doesn't appear and the proof relies only on orthogonality. In order to prove \eqref{eq:ineq-TT-star}, we will also use space or frequency decay.

As prescribed in \cite{multilinearMTT}, we should aim to prove for any $I \subseteq I_0$
\[
\| \big( \sum_{T \in \mathbf{T}} \sum_{\substack{P \in T \\ I_P =I}} \langle f, \phi_P \rangle \phi_P   \big) \, \ci_{I_0}^{-\frac{N}{2}}    \|_{L^2(\rr R \setminus 2\,I_0)}^2  \lesssim \frac{|I|^3}{|I_0|^3} \,  2^{2n}\, \big( \sum_{T \in \mathbf{T}} |I_T|  \big).
\]
At the end of a straightforward calculation, this would immediately yield \eqref{eq:ineq-TT-star}, since on $2\, I_0$ the function $\ci_{I_0}^{-\frac{N}{2}}$ is constant. On $\rr R \setminus 2I_0$ we use a smooth decomposition of $\ci_{I_0}^{-\frac{N}{2}}$ (see for example \cite{multilinear_harmonic}):
\[
\ci_{I_0}^{-\frac{N}{2}}=\sum_{\kappa \geq 1} 2^{\frac{\kappa \, N}{2} } \tilde{\ci}_{I_0, \kappa}, 
\]
where $\tilde \ci_{I_0, \kappa}$ is a smooth function adapted to the set $2^{\kappa+1}I_0 \setminus 2^{\kappa}I_0$. This allows for a simplification in the $\| \cdot \|_{L^2(\rr R \setminus 2\,I_0)}$ norm, since now we need an upper bound for
\[
\| \big( \sum_{T \in \mathbf{T}} \sum_{\substack{P \in T \\ I_P =I}} \langle f, \phi_P \rangle \phi_P   \big) \, \tilde \ci_{I_0}    \|^2_{L^2(\rr R)},
\]
albeit with a $2^{-\kappa (N+1)}$ decay.
Such an expression is easier to work with, as it equals 
\[
\sum_{T \in \mathbf{T}} \sum_{\substack{P \in T \\ I_P =I}} \sum_{T' \in \mathbf{T}} \sum_{\substack{P' \in T' \\ I_{P'} =I}} \langle f, \phi_P  \rangle \overline{\langle f, \phi_{P'}  \rangle}  \int_{\rr R} \phi_P(x) \overline{\phi_{P'}(x)} \tilde \ci_{I_0}^2(x) dx,
\]
and we already know that $|\langle f, \phi_P  \rangle| \leq 2^n |I_P|^{\frac{1}{2}}$ and similarly for $P'$. We fix $P \in T$ a tile; this will imply that all other tiles $P'$ with $I_P'=I=I_P$ come from different trees and all the frequency intervals $\omega_{P'}$ are mutually disjoint. By translation invariance we can also assume $\omega_P$ is centered at $0$. Using integration by parts $M$ times, we obtain that 
{\fontsize{10}{10}\begin{align*}
\big| \int_{\rr R} \phi_P(x) \overline{\phi_{P'}(x)} \tilde \ci_{I_0}^2(x) dx  \big| = \big| \int_{\rr R} \Phi^{[M]}_P(x) \frac{d^M}{d \, x^M} (\tilde \ci_{I_0}^2(x) ) dx  \big| \lesssim (|I|\,\dist( \omega_P, \omega_{P'}))^{-M} \, \big( \frac{|I|}{|I_0|} \big)^{M-1}  2^{-M \,\kappa}.
\end{align*}}

Above $\frac{d^M}{ d\, x^M} \Phi^{[M]}_P(x)= \phi_P(x) \overline{\phi_{P'}(x)}$ and because the Fourier transform of $\phi_P(x) \overline{\phi_{P'}(x)}$ is supported on the set $\lbrace |\xi| \sim  |\omega_{P'}|=\dist(\omega_P, \omega_{P'}) \rbrace$, we have 
\[
|\Phi^{[M]}_P(x)| \lesssim |I|^{-1} \dist(\omega_P, \omega_{P'})^{-M}.
\]

Because of the frequency disjointness, we can clearly sum in $P'$ an eventually we obtain
\[
\| \big( \sum_{T \in \mathbf{T}} \sum_{\substack{P \in T \\ I_P =I}} \langle f, \phi_P \rangle \phi_P   \big) \, \tilde \ci_{I_0}    \|^2_{L^2(\rr R)} \lesssim 2^{\kappa (N-M)} \frac{|I|^3}{|I_0|^3} \,  2^{2n}\, \big( \sum_{T \in \mathbf{T}} |I_T|  \big).
\]
Afterwards we sum in $\kappa$ (we just choose $M > N +10$) and over $I \subseteq I_0$ to obtain \eqref{eq:ineq-TT-star}.
\end{proof}
\end{lemma}

\subsubsection{Avoiding interpolation of adjoint operators}\label{sec:BHT-avoid-interpolation}We take another look at the $L^p$ estimates for $BHT$; that is, we want to deduce the estimate \eqref{eq:tril-form-BHT} by making use of Lemma \ref{lemma-local-gen-size-energy} instead of the generic estimate \eqref{eq:gen-size-energy} (although Lemma \ref{lemma-local-gen-size-energy} is a consequence of the latter). We will see that there is a way of recovering all the known $L^p$ estimates for $BHT$ without using the adjoint operators $BHT^{*,1}$ and $BHT^{*,2}$.

The localization Lemma \ref{lemma-local-gen-size-energy} is efficient only if the intervals $I_0$ satisfy certain conditions. Hence localization coordinates with a (triple) stopping time and generates a new partition of the collection of tiles $\rr P$. Each piece can be estimated in a precise way, and interpolation with the adjoint operators can be avoided. A similar principle of localization (and an underlying stopping time enclosed in the definition of super level measures) is used in \cite{outer-measure-var-Carleson} for defining \emph{iterated outer measures}. This doubles the number of stopping times, but allows for a formulation of Carleson embeddings in any $L^p$ range, $1<p<\infty$.

In order to prove the boundedness of $BHT$, we start with $F, G, H$ sets of finite measure, and we construct the major subset as before $H':=H \setminus \Omega$, where $\Omega$ is again defined by \eqref{eq:exceptional-set}. We want to obtain, for any functions $f, g, h$ satisfying $|f|\leq \one_F, |g|\leq \one_G$ and $|h| \leq \one_{H'}$, the estimate
\[
\vert \Lambda_{\rr P}(f, g, h) \vert \lesssim |F|^{a_1} |G|^{a_2} |H|^{a_3},
\]
where $(a_1, a_2, a_3)$ can be chosen arbitrarily close to any tuple $(\frac{1}{p}, \frac{1}{q}, \frac{1}{s'})$, with $(p, q, s) \in Range(BHT)$.

We are going to describe a procedure for selecting the \emph{relevant} dyadic intervals and the associated collections of tiles. $\rr P_{Stock}$ denotes the collection of tiles available at the moment of the selection. Also, $\ic D_{Stock}$ denotes the collection of dyadic intervals $I$ for which there exists some $P \in \rr P_{Stock}$ with $I_P \subseteq I$.

Here we assume that all the tiles in $\rr P$ satisfy $I_P \cap \Omega^c \neq \emptyset$. This will imply that $ \sssize_{\rr P} \one_F \lesssim  \min \Big(C \frac{|F|}{|H|},  1  \Big)$; if not, we need to make another decomposition $\rr P :=\bigcup_{d \geq 0} \rr P_d$, where
\[
\rr P_d=:\lbrace P \in \rr P: 1+\frac{\dist (I_P, \Omega^c)}{|I_P|} \sim 2^d   \rbrace.
\]
This will imply that
\[
\sssize_{\rr P_d} \one_F \lesssim 2^d \min \big( 1, \frac{|F|}{|H|}  \big), \quad \sssize_{\rr P_d} \one_G \lesssim 2^d \min \big( 1, \frac{|G|}{|H|}  \big), \quad \text{and     } \sssize_{\rr P_d} \one_{H' }\lesssim 2^{-Md}.
\]
Because of the fast decay in the $\sssize_{\rr P_d} \one_{H'}$ estimate, the summation in $d$ is not problematic, except for some technical difficulties which we are going to avoid essentially by assuming that $d=0$.

To start with, we initialize $\rr P_{Stock}:= \rr P$, and let $\ic S_1^1$ be the collection of maximal dyadic intervals $I_0 \in \ic D_{Stock}$ so that 
\[
\frac{1}{|I_0|} \int_{\rr R} \one_F \cdot \ci_{I_0} dx > \frac{1}{2} \min \Big(C \frac{|F|}{|H|},  1  \Big).
\]
Then we reset $\ds\rr P_{Stock}:=\rr P_{Stock}\setminus \big( \bigcup_{I_0 \in \ic S_1^1} \lbrace P \in  \rr P: I_P \subseteq I_0   \rbrace \big)$. 

Next, we define $\ic S_1^2$ to be the collection of maximal intervals $I_0 \in \ic D_{Stock}$ with the property that 
\[
\frac{1}{|I_0|} \int_{\rr R} \one_F \cdot \ci_{I_0} dx > \frac{1}{2^2} \min \Big(C \frac{|F|}{|H|},  1  \Big).
\]
We continue the procedure until the collection $\rr P$ of tri-tiles is exhausted, i.e. until $\rr P_{Stock}=\emptyset$. This will produce collections of intervals $\ic S_1^1, \ldots, \ic S_1^k, \ldots$, whose union we denote by $\ic S_1$.

Similarly, and independently, we construct $\ds \ic S_2:=\bigcup_{k=1}^\infty \ic S_2^k$ associated to the function $\one_G$, and $\ds \ic S_3:=\bigcup_{k=1}^\infty \ic S_3^k$ associated to $\one_{H'}$. Each of $\ic S_1, \ic S_2, \ic S_3$ determines a natural decomposition of the initial collection $\rr P$ as 
\[
\rr P:=\bigcup_{S_j \in \ic S_j}  \rr P_{S_j}, \quad \text{for every}\quad 1\leq j \leq 3,
\]
where by $\rr P_{S_j}$ we mean the collection of tritiles $P \in \rr P_j$ such that $I_P \subseteq S_j$ and there exists no other $S'_j \in \ic S_j$ so that $I_P \subseteq S'_j \subseteq S_j$. If $S_1 \in \ic S_1^k$, we deduce, from the selection algorithm, that
\[
\sssize_{\rr P_{S_1}} \one_F \lesssim  2^{-k} \min \Big( \frac{|F|}{|H|}, 1 \Big).
\]
In the same way, if $S_2 \in S_2^k$, then 
\[
\sssize_{\rr P_{S_2}} \one_G \lesssim  2^{-k} \min \Big( \frac{|G|}{|H|}, 1 \Big),
\]
and if $S_3 \in \ic S_3^k$, then $\ds \sssize_{\rr P_{S_3}} \one_{H'} \lesssim  2^{-k}$.

In addition, each of the collections $\ic S_j$, for $1 \leq j \leq 3$, is a Cantor-like set, or has a \emph{Carleson packing property}  (as referred in \cite{tree_vs_Carleson_box}):
\begin{equation}
\label{eq:Carleson-condition}
\text{if  } S_0 \in \ic S_j , \text{  then    } \sum_{\substack{S \in \ic S_j \\ S \subseteq S_0 }} |S| \leq \tilde C \, |S_0|. 
\end{equation}
This is because every $\ic S_j^k$ represents a part of a maximal covering of the level set of $\ic M(f)$, where $f$ is one of the functions $\one_F, \one_G$ or $\one_{H'}$. A proof can be found in Proposition 14, \cite{sparse-hel}. The geometrical structure of the collection $\ic S_j$ is illustrated in Figure \ref{fig:sparse}.

\bigskip
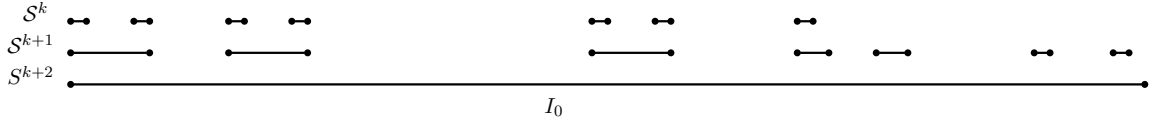
\begin{figure}
\psscalebox{.75 .75} 
{
\begin{pspicture}(0,-0.985)(20.16,0.985)
\psline[linecolor=black, linewidth=0.04, dotsize=0.06cm 1.6]{*-*}(1.12,-0.425)(20.16,-0.425)
\rput[bl](9.52,-0.985){$I_0$}
\psline[linecolor=black, linewidth=0.04, dotsize=0.06cm 1.6]{*-*}(1.12,0.135)(2.52,0.135)
\psline[linecolor=black, linewidth=0.04, dotsize=0.06cm 1.6]{*-*}(3.92,0.135)(5.32,0.135)
\psline[linecolor=black, linewidth=0.04, dotsize=0.06cm 1.6]{*-*}(10.36,0.135)(11.76,0.135)
\psline[linecolor=black, linewidth=0.04, dotsize=0.06cm 1.6]{*-*}(14.0,0.135)(14.56,0.135)
\psline[linecolor=black, linewidth=0.04, dotsize=0.06cm 1.6]{*-*}(15.4,0.135)(15.96,0.135)
\psline[linecolor=black, linewidth=0.04, dotsize=0.06cm 1.6]{*-*}(19.88,0.135)(19.6,0.135)(19.6,0.135)
\psline[linecolor=black, linewidth=0.04, dotsize=0.06cm 1.6]{*-*}(18.48,0.135)(18.2,0.135)
\psline[linecolor=black, linewidth=0.04, dotsize=0.06cm 1.6]{*-*}(1.12,0.695)(1.4,0.695)
\psline[linecolor=black, linewidth=0.04, dotsize=0.06cm 1.6]{*-*}(2.52,0.695)(2.24,0.695)
\psline[linecolor=black, linewidth=0.04, dotsize=0.06cm 1.6]{*-*}(3.92,0.695)(4.2,0.695)
\psline[linecolor=black, linewidth=0.04, dotsize=0.06cm 1.6]{*-*}(5.32,0.695)(5.04,0.695)
\psline[linecolor=black, linewidth=0.04, dotsize=0.06cm 1.6]{*-*}(10.36,0.695)(10.64,0.695)
\psline[linecolor=black, linewidth=0.04, dotsize=0.06cm 1.6]{*-*}(11.76,0.695)(11.48,0.695)
\psline[linecolor=black, linewidth=0.04, dotsize=0.06cm 1.6]{*-*}(14.0,0.695)(14.28,0.695)
\rput[bl](0.28,0.695){$\Huge \mathcal{S}^k$}
\rput[bl](0.0,0.135){$\Huge \mathcal{S}^{k+1}$}
\rput[bl](0.0,-0.425){$\Huge S^{k+2}$}
\end{pspicture}
}
\caption{Geometry of the collections of intervals $\ic S^k$}
\label{fig:sparse}
\end{figure}
\smallskip

The  trilinear form associated to the model operator $BHT_{\rr P}$ can be estimated by
{\fontsize{10}{10}\begin{align*}
\vert \Lambda_{BHT; \rr P} (f, g, h)  \vert &\lesssim \sum_{n_1, n_2, n_3 \geq 0} \sum_{\substack{S_1 \in \ic S_1^{n_1} \\ S_2 \in \ic S_2^{n_2} \\S_3 \in \ic S_3^{n_3}}} \sum_{P \in \rr P_{S_1}\cap \rr P_{S_2} \cap \rr P_{S_3}} \frac{1}{|I_P|^{1/2}} \vert \langle f, \phi_{P_1}^1 \rangle  \langle g, \phi_{P_2}^2 \rangle \langle h, \phi_{P_3}^3 \rangle \vert \\
&\lesssim \sum_{j=1}^3 \sum_{n_j \geq 0} \sum_{S_j \in \ic S_j^{n_j}} 2^{- \frac{n_1 \left( 1+\theta_1 \right)}{2}} \cdot 2^{- \frac{n_2 \left( 1+\theta_2 \right)}{2}} \cdot 2^{- \frac{n_3 \left( 1+\theta_3 \right)}{2}} \\
& \cdot \min \big( C \frac{|F|}{|H|}, 1  \big)^\frac{1+\theta_1}{2} \cdot  \min \big( C \frac{|G|}{|H|}, 1  \big)^\frac{1+\theta_2}{2} \cdot  1^\frac{1+\theta_3}{2} \cdot |S_1 \cap S_2 \cap S_3|.
\end{align*}}

Now we need to estimate the sum of the spatial supports; first, we note that
\[
\sum_{S_j \in \ic S_j^{n_j}} \vert S_1 \cap S_2 \cap S_3 \vert \lesssim \sum_{S_1 \in \ic S_1^{n_1} } |S_1| \lesssim 2^{n_1} \cdot \min \big( C \frac{|F|}{|H|}, 1  \big)^{-1} \, |F|.
\]

Similarly, 
\[
\sum_{S_j \in \ic S_j^{n_j}} \vert S_1 \cap S_2 \cap S_3 \vert \lesssim  2^{n_2} \cdot \min \big( C \frac{|G|}{|H|}, 1  \big)^{-1} \, |G|,  \quad \sum_{S_j \in \ic S_j^{n_j}} \vert S_1 \cap S_2 \cap S_3 \vert \lesssim   2^{n_3}\, |H|.
\]
So in fact, we can use the geometric mean of the three expressions above: if $0 \leq \alpha_1, \alpha_2, \alpha_3 \leq 1$, with $\alpha_1+\alpha_2+ \alpha_3 = 1$, we have 
\[
\sum_{S_j \in \ic S_j^{n_j}} \vert S_1 \cap S_2 \cap S_3 \vert \lesssim 2^{n_1 \cdot \alpha_1} \, 2^{n_2 \cdot \alpha_2} \, 2^{n_3 \cdot \alpha_3} \min \big( C \frac{|F|}{|H|}, 1  \big)^{-\alpha_1} \cdot |F|^{\alpha_1} \cdot \min \big( C \frac{|G|}{|H|}, 1  \big)^{-\alpha_2} \cdot |G|^{\alpha_2} \cdot |H|^{\alpha_3}.
\]

This implies that
\[
\vert \Lambda_{BHT; \rr P} (f, g, h)  \vert \lesssim |F|^{a \left(\frac{1+\theta_1}{2} -\alpha_1  \right)+\alpha_1} \cdot |G|^{b \left(\frac{1+\theta_2}{2} -\alpha_2  \right)+\alpha_2} \cdot |H|^{\alpha_3 - a \left(\frac{1+\theta_1}{2} -\alpha_1  \right)-b \left(\frac{1+\theta_2}{2} -\alpha_2  \right)},
\]
for any $0 \leq a, b \leq 1$, provided that
\begin{equation}
\label{eq:cond-conv}
\frac{1+\theta_1}{2}>\alpha_1, \quad \frac{1+\theta_2}{2}>\alpha_2, \quad \frac{1+\theta_3}{2}>\alpha_3.
\end{equation}
The condition above ensures the convergence of the exponential series, and indeed it is possible to find such $\alpha_j$ because $\ds \frac{1+\theta_1}{2}+\frac{1+\theta_2}{2}+\frac{1+\theta_3}{2}=2$, while $\alpha_1+\alpha_2+ \alpha_3 = 1$.

If $\ds a_1=a \left(\frac{1+\theta_1}{2} -\alpha_1  \right)+\alpha_1, \, a_2=b \left(\frac{1+\theta_2}{2} -\alpha_2  \right)+\alpha_2$ and $a_3=\alpha_3 - a \left(\frac{1+\theta_1}{2} -\alpha_1  \right)-b \left(\frac{1+\theta_2}{2} -\alpha_2  \right)$, it remains to check that we can choose them arbitrarily close to any of the points 
\[
(1, 0, 0),\quad  (0, 1, 0), \quad (0, 0, 1), \quad  \big(1, \frac{1}{2}, - \frac{1}{2} \big), \quad \big(\frac{1}{2}, 1, -\frac{1}{2} \big),
\]
since the convex hull of these points represents exactly $Range(BHT)$.

This is surely possible. For example, in order to have $(a_1, a_2, a_3)$ arbitrarily close to $(1, 0, 0)$, we can choose $a=1,\, b=0,\,  \theta_1 \sim 1$, which implies $\theta_2 \sim 0,\, \theta_3 \sim 0$. Then we further choose $\alpha_2 \sim 0$ and $\alpha_3 \sim 0$, which agrees with the constraint \eqref{eq:cond-conv}.

\subsection{Connections with outer measures}
\label{sec:outer-measures}

We've seen the role played by the generic size and energy estimate
\begin{align*}
\vert \Lambda_{BHT;\rr P}(f_1, f_2, f_3)\vert  & \lesssim \big( \sssize_{\rr P}\, f_1 \big)^{\theta_1} \cdot \big( \sssize_{\rr P} \, f_2 \big)^{\theta_2} \cdot \big( \sssize_{\rr P} \, f_3 \big)^{\theta_3} \\
& \cdot \big( \eenergy_{\rr P}\, f_1 \big)^{1-\theta_1} \cdot \big( \eenergy_{\rr P} \, f_2  \big)^{1-\theta_2} \cdot \big( \eenergy_{\rr P} f_3 \, \big)^{1-\theta_3}\nonumber.
\end{align*}

In particular, for functions bounded above by characteristic functions of sets of finite measure, it implies the following:
\begin{enumerate}[label=(\alph*), ref=\alph*]
\item \label{item-Local l2}In the ``local $L^2$" case, when $2 \leq p, q, s' < \infty$, we can ignore the``size" and deduce the regular estimate $\ds BHT: L^p \times L^q \to L^s$ by using solely restricted type interpolation, as in \eqref{eq:discard-size}; there is no requirement to remove an exceptional set. 
\item \label{item-outside-local-L2} In order to prove all known estimates for $BHT$, we need to make use of both the energy and the size, the latter having to be evaluated on certain level sets. For that reason, we assume $|f_j| \leq \one_{E_j}$, and we carry out the analysis on subcollections $\rr P_{\#}$ having the property that $\sssize_{\rr P_\#} \one_{E_j} \lesssim \# |E_j|$. Nevertheless, we still have the constraint $\big|  \frac{1}{p_1} -\frac{1}{p_2} \big| < \frac{1}{2}$ and in consequence, we need to use multilinear interpolation between adjoint operators.
\item \label{item-local} The local estimate \eqref{eq:local-gen-size-energy} combines together the size and energy information, and the interpolation between adjoint operators is not necessary anymore. Instead, it is replaced by an additional stopping time.
\end{enumerate}

Next, we will rewrite some of the arguments presented earlier in a different language. This is mainly motivated by the remarks \eqref{item-Local l2}-\eqref{item-local}, in an attempt to bring to light certain pre-existing structures. For this purpose, we define
\begin{equation}
\label{def-Lq-mock}
\|F_j\|_{\ii L^{q_j}_{mock}}:= \big(  \sssize_{\rr P} f_j \big)^{\theta_j} \cdot \big( \eenergy_{\rr P}  f_j  \big)^{1-\theta_j},
\end{equation}
where $\frac{1}{q_j}=\frac{\theta_j}{\infty}+\frac{1-\theta_j}{2}$, while for every $1 \leq j \leq 3$, $F_j$ is a function defined on the collection of tiles $\rr P$, and which depends on $f_j$:
\[
F_j(P)=\langle f_j, \phi_P  \rangle, \quad \text{for all   } P \in \rr P.
\]

Notice that $\|F_j\|_{\ii L^{q_j}_{mock}}$ resembles an interpolation quantity between $\sssize_{\rr P} f_j$ and $\eenergy_{\rr P}  f_j $, and this is consistent with the earlier intuition: $\sssize_{\rr P} f_j$ is an $L^\infty$-type quantity as a supremum of averages taken over a certain collection of intervals, while $\eenergy_{\rr P}  f_j $ is an $L^{2, \infty}$-type quantity.

Certainly, the functions $F_j$ need not be defined on the collection $\rr P$ of tri-tiles, but rather on the $j^{\text{th}}$ coordinate projection of $\rr P$:
\[
\rr P_j:=\lbrace I_P \times \omega_{P_j} : P= \left(  I_P \times \omega_{P_1},  I_P \times \omega_{P_2},  I_P \times \omega_{P_2} \right) \in \rr P   \rbrace.
\]
However, we need to understand the interactions of the tri-tiles in the collection $\rr P$ (summarized in the estimate \eqref{eq:gen-size-energy}) before moving on to studying the collections $\rr P_j$. We abuse notation and denote by $\rr P$ any of the collections $\rr P_j$; it will be clear from the context whether it is a collection of tiles (which is the case when considering $\ii L^q_{mock}$ spaces) or tri-tiles (this corresponds to the analysis of the bilinear operator or of its associated trilinear form).

With these notations, the main inequality \eqref{eq:gen-size-energy} reads as
\begin{equation}
\label{eq:tril-form-L-mock}
\vert \Lambda_{BHT; \rr P}(f_1, f_2, f_3)  \vert \lesssim \|F_1\|_{\ii L^{q_1}_{mock}} \|F_2\|_{\ii L^{q_2}_{mock}} \|F_3\|_{\ii L^{q_3}_{mock}},
\end{equation}
and we are left with proving
\begin{equation}
\label{eq:tril-carleson-embeddings}
\|F_1\|_{\ii L^{q_1}_{mock}} \|F_2\|_{\ii L^{q_2}_{mock}} \|F_3\|_{\ii L^{q_3}_{mock}} \lesssim \|f_1\|_{p_1}  \|f_2\|_{p_2}  \|f_3\|_{p_3},
\end{equation}
where $2 \leq q_j \leq \infty$, $\frac{1}{q_1}+\frac{1}{q_2}+\frac{1}{q_3}=1, \frac{1}{p_1}+\frac{1}{p_2}+\frac{1}{p_3}=1$.

\smallskip
We need to understand how $\|F_j\|_{\ii L^{q_j}_{mock}}$ relates to $\|f_j\|_{p_j}$. But this is precisely what we described earlier for restricted-type functions, in the proofs of the cases \eqref{item-Local l2} and \eqref{item-outside-local-L2}: under the assumption that $|f| \leq \one_E$, we showed estimates of the type
\[
(\sssize_{\rr P} f )^{\theta} (\eenergy_{\rr P} f)^{1-\theta} \lesssim |E|^\frac{1}{p}=\| \one_E\|_p.
\]

As it often happens when studying the behavior of an operator through the associated multilinear form, the last function $f_3$ plays the role of a testing function, used to liniarize an $L^{s, \infty}$ norm of the operator (see Lemma 2.6 of \cite{multilinear_harmonic}). In fact, we require that $1<p_1, p_2 \leq \infty$, and $p_3$ is so that $\frac{1}{p_1}+\frac{1}{p_2}+\frac{1}{p_3}=1$, $-\frac{1}{2}<p_3<\infty$. For these reasons, it suffices to examine the $L^{p_j} \mapsto \ii L^{q_j}_{mock}$ Carleson embedding for $j=1, 2$.

\smallskip
We present the corresponding Carleson embeddings from the cases \eqref{item-Local l2} and \eqref{item-outside-local-L2} in the next Propositions:

\subsubsection*{\textbf{The case $2<p=q$:}}
In the ``local $L^2$" case (when $2< p_j=q_j \leq \infty$), as mentioned in \eqref{item-Local l2}, we discard the size and obtain immediately, for $q>2$ and a function $f$ satisfying $|f| \leq \one_E$, the following:
\[
\|F \|_{\ii L^{q}_{mock}} \lesssim |E|^\frac{1}{q}, \quad \text{where  } \frac{1}{q}=\frac{1-\theta}{2} <\frac{1}{2}.
\]

We can formulate this as:
\begin{proposition}
\label{propo:Carleson-embedding-mock-above-L2}
Let $2<q \leq \infty$ and let $f$ be a function so that $|f| \leq \one_E$. Then
\[
\|F \|_{\ii L^{q}_{mock}} \lesssim |E|^\frac{1}{q}.
\]
\end{proposition}

This is an $L^q \mapsto \ii L^q_{mock}$ Carleson embedding ``above $L^2$" for restricted-type functions.

\subsubsection*{\textbf{The case $p<2$:}} 
This situation is especially interesting because it arises in the study of the bilinear Hilbert transform, when at least one of $p_1, p_2 <2$, while $\frac{1}{p_1}+\frac{1}{p_2} > \frac{3}{2}$, and we can't do  without using the sizes as well as the energies.

\begin{proposition}
\label{prop:L-mock-below-L2}
 Assume that $|f|\leq \one_E$ and $\sssize_{\rr P}(f) \leq \# |E|$. Then
\[
\|F \|_{\ii L^q_{mock}} \lesssim \#^{\frac{1}{p}-\frac{1}{q}} |E|^\frac{1}{p},
\]
whenever $q' \leq p<2$. That is, we have an $L^p \mapsto \ii L^q_{mock}$ Carleson embedding ``below $L^2$", for restricted-type functions, provided $q' \leq p <2$.
\begin{proof}
The proof is straightforward: since $|f|\leq \one_E$, we actually have $\ds \sssize_{\rr P}f \lesssim \min (1, \# |E|)$, and in consequence, 
\[
\|F \|_{\ii L^q_{mock}} \lesssim  (\sssize_{\rr P}f)^{a} \, (\eenergy_{\rr P} f)^{1-\theta} \lesssim \min (1, \# |E|) ^ a \cdot |E|^{\frac{1-\theta}{2}},
\]
for any $0 \leq a \leq \theta=\frac{1}{q'}-\frac{1}{q}$ (since $\frac{1}{q}=\frac{1-\theta}{2}$). If we set 
\[
a=\frac{1}{p}-\frac{1}{q} \leq \frac{1}{q'}-\frac{1}{q} 
\]
we obtain the conclusion. Note that $0 \leq a \leq \frac{1}{q'}-\frac{1}{q}$ is equivalent to $q' \leq p < 2 <q$.
\end{proof}
\end{proposition}

Now, if we want to use the Carleson embedding below $L^2$ in order to deduce \eqref{eq:tril-carleson-embeddings} from \eqref{eq:tril-form-L-mock}, we note that we still have the constraint 
\[
\big\vert \frac{1}{p_1} -\frac{1}{p_2}\big\vert \leq \frac{1}{2},
\]
which is similar to \eqref{eq:silly-eq}. This is because we assume $\frac{1}{q_1} \leq \frac{1}{p_1} < \frac{1}{q_1'}, \frac{1}{q_2} \leq \frac{1}{p_2} < \frac{1}{q_2'}$, and in consequence $\vert \frac{1}{p_1} -\frac{1}{p_2}\vert \leq \frac{1}{q_3} < \frac{1}{2}$ (all the $\ii L^q_{mock}$ spaces are well defined only for $q>2$).

Using the localization, however, we are able to obtain directly the full range of boundedness for $BHT$. This was presented in Section \ref{sec:BHT-avoid-interpolation}.

In Proposition \ref{prop:L-mock-below-L2}, the assumption $\sssize_{\rr P}(f) \leq \# \, |E|$ is a remnant of the fact that historically we assume all the tiles in $\rr P$ are away from the exceptional set. But the following is also true (and the proof is identical):

\begin{thmbis}{prop:L-mock-below-L2}
 Assume that $|f|\leq \one_E$ and $\sssize_{\rr P}(f) \leq \tilde{ \#}$. Then
\begin{equation}
\label{eq:prop-2-prim}
\|F \|_{\ii L^q_{mock}} \lesssim  \min (\tilde{\#}, 1)^{1-\frac{2}{q}} |E|^\frac{1}{q}.
\end{equation}
\end{thmbis}

In order to have an $L^p$ norm of $\one_E$ on the right hand side of \eqref{eq:prop-2-prim}, it is natural to consider $\tilde{\#}= \# \, |E|$, in which case we recover the result of Proposition \ref{prop:L-mock-below-L2}.

\subsubsection*{\textbf{Localized Carleson embeddings}}

Recall that in order to prove direct estimates for $BHT$, the ``global" size and energy lemma was inefficient, and we used instead (infinitely many times, through three additional stopping times) a ``local" version of it. Consequently, there is also a natural``local Carleson embedding" inherently associated to it and which we will present next.

We let $I_0$ be a fixed dyadic interval and $\rr P(I_0)$ represents 
\[
\rr P(I_0):=\lbrace P \in \rr P: I_P \subseteq I_0  \rbrace.
\]
We recall that $\rr P, \,\rr P(I_0)$ are collection of tiles, and $F : \rr P \mapsto \rr C$ is a function defined on $\rr P$. Then 
\[
\| F  \|_{\ii L^q_{mock; I_0}}:=\| F \cdot \one_{\rr P\left( I_0 \right)} \|_{\ii L^q_{mock}}.
\]

\begin{proposition}
\label{prop:localized-embedding-mock}
If $I_0$ is a fixed dyadic interval, $f$ a function satisfying $|f| \leq \one_E$ and $\ds \sssize_{\rr P(I_0)}(f) \leq \# |E|$, then
\[
\|F\|_{\ii L^q_{mock; I_0}} \lesssim\big( \# |E| \big)^\frac{1}{p} |I_0|^\frac{1}{q},
\]
for any $q' \leq p$.
\end{proposition}
\begin{proof}
Following the result in Lemma \ref{lemma-local-gen-size-energy}, we have that
\[
\eenergy_{\rr P(I_0)}(f) \lesssim \| f \cdot \ci_{I_0}  \|_2 \lesssim \big(\sssize_{\rr P(I_0)}(f) \cdot |I_0| \big)^\frac{1}{2}.
\]
This implies
\begin{align*}
\|F\|_{\ii L^q_{mock; I_0}} &\lesssim \big(  \sssize_{\rr P(I_0)} \one_E \big)^\frac{1+\theta}{2} \cdot |I_0|^\frac{1-\theta}{2}=  \big(  \sssize_{\rr P(I_0)} \one_E \big)^\frac{1}{q'} \cdot |I_0|^\frac{1}{q}\\
& \lesssim \big(  \sssize_{\rr P(I_0)} \one_E \big)^\frac{1}{p} \cdot |I_0|^\frac{1}{q} \lesssim \big( \# |E| \big)^\frac{1}{p} \cdot |I_0|^\frac{1}{q}.
\end{align*}
Hence we are done: upon localization onto the interval $I_0$, there is a way of controlling $\| F\|_{\ii L^q_{mock};I_0}$, which is only defined for $q>2$, by an $L^p$ norm of $f$ (represented by $|E|^\frac{1}{p}$) for $q' \leq p <2$, provided the collection $\rr P(I_0)$ also satisfies $\sssize_{\rr P(I_0)}(\one_E) \leq \# \, |E|$.
\end{proof}

A similar result is available if we assume that $\sssize_{\rr P}(f) \lesssim \tilde{\#}$:

\begin{thmbis}{prop:localized-embedding-mock}
If $I_0$ is a fixed dyadic interval, $f$ a function satisfying $|f| \leq \one_E$ and $\ds \sssize_{\rr P(I_0)}(f) \leq \tilde{\#}$, then
\[
\|F\|_{\ii L^q_{mock; I_0}} \lesssim\min (\tilde{\#}, 1)^\frac{1}{q'}\, |I_0|^\frac{1}{q}.
\]
\end{thmbis}

\subsubsection*{\textbf{Outer measure spaces}}
Let $F: \rr P \to \rr C$ be defined by $F(P)=\langle f, \phi_P \rangle$, as before. From definition \eqref{def-Lq-mock}, it is apparent that $\| F \|_{\ii L^q_{mock}}$ represents an interpolation quantity between $\sssize_{\rr P}(f)$ and $\eenergy_{\rr P}(f)$, and none of them suffices individually for obtaining the whole range of boundedness for $BHT$. They can also be regarded as $L^\infty$ and respectively $L^{2, \infty}$ ``norms" of $F$, the function defined on the collection $\rr P$ of tiles. This can be made precise, as it was done in \cite{outer_measures}, where the outer measure $\ii L^q(\rr P, \sigma, S)$ spaces were introduced.

Supporting the point of view that averages rather than pointwise values describe the behavior of a function, the \emph{$L^p$ theory for outer measure spaces} of \cite{outer_measures} starts out with predefined averages over the generating sets of the outer measure. These predefined averages are very much related to the \emph{sizes} and \emph{energies} of the previous sections: $\| F\|_{\ii L^\infty\left(\rr P, \sigma, S\right)}$ represents a substitute for $\ssize_{\rr P}(f)$, and $\| F\|_{\ii L^{2,\infty} \left(\rr P, \sigma, S\right)}$ replaces $\eenergy_{\rr P}(f)$. 

Moreover, for $ 2< q <\infty$, $\ii L^q(\rr P, \sigma, S)$ serves as an interpolation space between $\ii L^{2, \infty}(\rr P, \sigma, S)$ and $\ii L^\infty(\rr P, \sigma, S)$ and satisfies 
\[
\|F\|_{\ii L^{q} \left( \rr P, \sigma , S \right)} \lesssim \|F\|_{\ii L^{\infty} \left( \rr P, \sigma , S \right)} ^\theta \cdot \|F\|_{\ii L^{2, \infty} \left( \rr P, \sigma , S \right)} ^{1-\theta}, \quad \text{where} \quad \frac{1}{q}=\frac{1-\theta}{2} +\frac{\theta}{\infty}.
\]

Via an outer measure H\"older inequality, the study of the trilinear form from \eqref{eq:tril-form-BHT} can be reduced to that of \emph{Carleson embeddings}. This was the case also earlier, when we were working with the $\ii L^q_{mock}$ formalism. Hence we will be focusing in the remainder of the section mostly on Carleson embeddings, presenting the three cases discussed previously in \eqref{item-Local l2} - \eqref{item-local}.

Here $\rr P$ is a (possibly finite) collection of time-frequency tiles $P=I_P \times \omega_P$, as described in the beginning of Section \ref{sec:generic-size-energy-est}. The collection generating the outer measure is
\[
\mathbf E := \lbrace T \subseteq \rr P : T \text{  is  a lacunary tree }   \rbrace,
\]
and for any tree $T \in \mathbf E$, we set $\sigma(T)=|I_T|$. Then the outer measure generated by $\sigma$, denoted $\mu$, is defined for any $\rr P' \subseteq \rr P$ by:
\[
\mu(\rr P'):= \inf \lbrace \sum_{T \in \mathbf{T'} \subseteq \mathbf{E} }|I_T| : \rr P' \subseteq \rr P_{\mathbf{T'}}:= \bigcup_{T \in \mathbf{T'}} T   \rbrace.
\]
That is, we consider the infimum over all collections of lacunary trees that cover $\rr P'$. It turns out this is comparable to 
\[
\mu_*(\rr P'):=\sup \lbrace \sum_{T \in \mathbf{T'} \subseteq \mathbf{E} }|I_T| : \rr P' = \bigcup_{T \in \mathbf{T'}} T , \text{ and } \mathbf{T'} \text{  is a collection of disjoint trees}  \rbrace,
\]
which would be the equivalent of an \emph{inner measure} on $\rr P$. If $F : \rr P \to \rr C$, then its $L^2$ size is a function on $\mathbf{E}$ defined by
\[
S_2(F)(T):= \big(  \frac{1}{|I_T|} \sum_{P \in T} |F(P)|^2  \big)^{\frac{1}{2}}= \| F\|_{\ell^2_T \left( |I_T|^{-1} \right)}.
\]
Also, $\ds S_\infty(F)(T):=\sup_{P \in T}\dfrac{|F(P)|}{|I_P|^{1/2}}$. 

The size, i.e. the function which determines the predefined averages on the generating collection $\mathbf E$ of the outer measure space, used in \cite{outer_measures}, \cite{Carleson-embedding-belowL2}, is $S:=S_2+S_\infty$. This points out to a combination of $L^2$ information and $L^1$ maximal averages.

If $F: \rr P \to \rr C$ and $\lambda >0$, the \emph{super level measure} $\mu(S(F)>\lambda)$ is defined by
\begin{equation}
\label{eq:def-size-out}
\mu(S(F)>\lambda)=\inf \lbrace \mu(\rr P') : \big(  \frac{1}{|I_T|} \sum_{P \in T \cap \left(\rr P '\right)^c} |F(P)|^2  \big)^{\frac{1}{2}} +  \sup_{P \in T \cap \left(\rr P '\right)^c}\dfrac{|F(P)|}{|I_P|^{1/2}} \leq \lambda \text{  for all trees  }T \in \mathbf E  \rbrace.
\end{equation}
Above, the infimum is taken over all subsets $\rr P' \subseteq \rr P$, the complement of which doesn't contain any trees $T$ of size $S(F)(T) > \lambda$.

Finally, the \emph{outer $L^p$ quasi-norms are defined in the following way}:
\[
\|F\|_{\ii L^\infty \left( \rr P, \sigma , S \right)}:= \sup_{T \in \mathbf E} S(F)(T),
\]
and, for $0< p<\infty$
\[
\|F\|_{\ii L^{p,\infty} \left( \rr P, \sigma , S \right)}:= \sup_{\lambda >0} \lambda \, \mu (S(F)>\lambda)^\frac{1}{p},
\]
\[
\|F\|_{\ii L^{p} \left( \rr P, \sigma , S \right)}:= \big(  \int_{0}^\infty p \lambda^{p-1} \mu (S(F)>\lambda) d \lambda \big)^\frac{1}{p}.
\]

A Marcinkiewicz interpolation theorem is available for quasi-sublinear operators taking values in outer measure spaces (Proposition 3.5 of \cite{outer_measures}). Moreover, if the size function $S$ defined on $\mathbf{E}$ is subadditive, which is the case with $S_2, S_\infty$ and $S=S_2+S_\infty$, restricted type interpolation in the context of outer measure spaces is almost identical to the classical case, modulo technical difficulties. On that account, there is no loss of information in studying in the first place how an operator acts on functions bounded above by characteristic functions. 
\smallskip
The application 
\[
f \mapsto F \qquad \text{defined by} \quad F(P)=\langle f, \phi_P \rangle
\]
can be regarded as a linear operator from the space of functions on $\rr R$ to the space of functions on $\rr P$. Using the notation of the previous section, we notice that
\[
\ssize_{\rr P} (f) = \|F\|_{\ii L^\infty \left( \rr P, \sigma , S_2 \right)} \leq  \|F\|_{\ii L^\infty \left( \rr P, \sigma , S \right)}, \quad \eenergy_{\rr P}(f) \lesssim \|F\|_{\ii L^{2,\infty}\left( \rr P, \sigma , S_2 \right)} \leq \|F\|_{\ii L^{2,\infty}\left( \rr P, \sigma , S \right)}.
\]

In fact, we have an equivalence:
\begin{equation}
\label{eq:equiv-size-en}
\sssize_{\rr P} (f) \sim \|F\|_{\ii L^\infty \left( \rr P, \sigma , S \right)}, \qquad \eenergy_{\rr P}(f) \sim \|F\|_{\ii L^{2,\infty}\left( \rr P, \sigma , S \right)}.
\end{equation}

The second identity will make the object of Lemma \ref{lemma-energy-control}, while the first one is a consequence of John-Nirenberg (and the trivial observation that any tile $P$ can be regarded as a lacunary tree). Hence the expression $\ds \big( \sssize_{\rr P} (f) \big)^\theta \cdot \big( \eenergy_{\rr P}(f) \big)^{1-\theta}$, which was defined before as $\ii L^q_{mock}$ (provided $\frac{1}{q}=\frac{1-\theta}{2}+\frac{\theta}{\infty}$), can be thought of as an interpolation quasi-norm between $\| \cdot\|_{\ii L^\infty \left( \rr P, \sigma , S \right)}$ and $\| \cdot \|_{\ii L^{2,\infty}\left( \rr P, \sigma , S \right)}$.

The general size and energy estimate \eqref{eq:gen-size-energy}, which was reformulated as \eqref{eq:tril-form-L-mock} using the $\ii L^q_{mock}$ spaces, becomes, due to an outer H\"older inequality, 
\[
\vert \Lambda_{BHT; \rr P}(f_1, f_2, f_3)  \vert \lesssim \|F_1\|_{\ii L^{q_1}(\rr P, \sigma, S)}  \|F_2\|_{\ii L^{q_2}(\rr P, \sigma, S)}\|F_3\|_{\ii L^{q_3}(\rr P, \sigma, S)},
\]
where $\ds \frac{1}{q_1}+\frac{1}{q_2}+\frac{1}{q_3}=\frac{1-\theta_1}{2}+\frac{1-\theta_2}{2}+\frac{1-\theta_3}{2}=1$ and $F_j(P)=\langle f_j, \phi_P  \rangle$ for all $P \in \rr P, 1 \leq j \leq 3$.

If we knew that $\|F_j\|_{\ii L^{q_j}(\rr P, \sigma, S)} \lesssim \|f_j\|_{p_j}$ (such an inequality will represent an $L^{p_j} \mapsto \ii L^{q_j}(\rr P, \sigma, S)$ Carleson embedding), then H\"older's inequality above would imply that
\[
\vert \Lambda_{BHT; \rr P}(f_1, f_2, f_3) \lesssim \|f_1\|_{p_1} \|f_2\|_{p_2} \|f_3\|_{p_3},
\]
where, as usual, we require that $\frac{1}{p_1}+\frac{1}{p_2}+\frac{1}{p_3}=1$. If $p_3'<1$, i.e. $L^{p_3'}$ is a quasi-Banach space, the estimate above can be reformulated by dualizing $\| \cdot \|_{L^{p_3', \infty}}$ as in Lemma 2.6 of \cite{multilinear_harmonic}.

Next, we claim that

\begin{proposition}
\label{prop:Lq-lessLq-mock}
Let $\rr P, \sigma, S$ be as above. Then
\begin{equation}
\label{eq:mock-larger}
\|F\|_{\ii L^{q}(\rr P, \sigma, S)} \lesssim \| F\|_{\ii L^{q}_{mock}}, \qquad \text{for all     } 2<q \leq \infty.
\end{equation}
\end{proposition}

So we can use the estimates from the previous section in order to deduce the $L^{p_j} \mapsto \ii L^{q_j}(\rr P, \sigma, S)$ Carleson embeddings, at least for restricted type functions. For this, we need to prove the end-point estimate corresponding to $\theta=0$: $\ds \|F\|_{\ii L^{2, \infty}(\rr P, \sigma, S)} \lesssim \eenergy_{\rr P}(f)$. Assuming this to be true, we show how it implies the inequality \eqref{eq:mock-larger}.

\smallskip
\begin{proof}[Proof of Proposition \ref{prop:Lq-lessLq-mock}]
Note that 
\[
\| F\|_{\ii L^{q}(\rr P, \sigma, S)}^q \lesssim \sum_{n \in \rr Z} 2^{nq} \mu \big( S(F)> 2^n  \big),
\]
and we only need to consider those values of $n$ for which $2^n \leq \| F\|_{\ii L^{\infty}(\rr P, \sigma, S)}$. Then, given that $\frac{1}{q}=\frac{1-\theta}{2}$, we have
\begin{align*}
\| F\|_{\ii L^{q}(\rr P, \sigma, S)}^q &\lesssim \sum_{n \leq \| F\|_{\ii L^{\infty}(\rr P, \sigma, S)}} 2^{n \left( q-2 \right)}  \, 2^{2n}\mu \big( S(F)> 2^n  \big) \\
&\lesssim  \| F\|_{\ii L^{2, \infty}(\rr P, \sigma, S)}^2   \sum_{n \leq \| F\|_{\ii L^{\infty}(\rr P, \sigma, S)}} 2^{n \left( q-2 \right)} \lesssim \big( \sssize_{\rr P} f  \big)^{q-2} \, \big( \eenergy_{\rr P} f \big)^2:= \| F\|_{\ii L^{q}_{mock}}^q.
\end{align*}
\end{proof}

We are left with proving the following:

\begin{lemma}
\label{lemma-energy-control}
Let $\rr P$ be a finite family of tiles. Then for any function $f \in \ic S(\rr R)$, we have 
\[
\|F\|_{\ii L^{2, \infty}(\rr P, \sigma, S)} \lesssim \eenergy_{\rr P}(f).
\]
\begin{proof}
Recall that  
\[
\|F\|_{\ii L^{2, \infty}(\rr P, \sigma, S)}^2 =\sup_{\lambda>0} \, \lambda^2 \mu(S(F)>\lambda)=\sup_{\lambda>0} \lambda^2 \, \inf \lbrace \mu(\rr P') : S(F \cdot \one_{\left( \rr P' \right)^c})(T) \leq \lambda  \quad \forall \, T \in \mathbf{E}  \rbrace. 
\]

Let $\lambda>0$ and $\rr P' \subseteq \rr P$ be so that $S(F \cdot \one_{\left( \rr P' \right)^c})(T) \leq \lambda  \quad \forall \, T \in \mathbf{E}$.  Since we are trying to minimize $\mu(\rr P')$, we can assume that for any tree $T \subset \rr P'$ we have $S(F)(T) > \lambda$.

A classical stopping-time argument (see Proposition 6.15 of \cite{multilinear_harmonic}) allows us to decompose $\rr P'$ into trees in the following way:
\[
\rr P':= \bigcup_{n} \bigcup_{T \in \mathbf{T}_n} T,
\]
where for every $n$ as above, $\mathbf{T}_n$ represents a collection of strongly disjoint trees with
\[
S_2(F)(T) \geq 2^{n-1}\quad \forall \, T \in \mathbf{T}_n, \quad S_2(F)(T') \leq 2^{n}\quad \forall \, T' \subseteq T \in \mathbf{T}_n.
\]
In particular, since every tile $P$ is also a tree, we have that $S_\infty(F)(T) \leq 2^n$ and $S_{2}(F)(T) \sim 2^n$ for all $T \in \mathbf{T}_n$. Consequently, $\mathbf{T}_n \neq \emptyset$ only if $2^{n} \geq \lambda$.

Then we obtain
\begin{align*}
\lambda^2 \mu( \rr P') & \leq \lambda^2 \sum_{2^{n} \geq \lambda} \sum_{T \in \mathbf{T}_n} |I_T| \leq \sum_{2^{n} \geq \lambda} \lambda^2 \, 2^{-2n} \cdot \big( 2^{2n} \sum_{T \in \mathbf{T}_n} |I_T| \big) \\
&\leq \sum_{2^{n} \geq \lambda} \lambda^2 \, 2^{-2n} \cdot \big( \eenergy_{\rr P}(f)  \big)^2 \lesssim \big( \eenergy_{\rr P}(f)  \big)^2.
\end{align*}

But $\lambda >0$ and $\rr P'$ were arbitrary; hence we can deduce that $\|F\|_{\ii L^{2, \infty}(\rr P, \sigma, S)} \lesssim \eenergy_{\rr P}(f)$.
\end{proof}
\end{lemma}

Now we take a look at the $L^p \mapsto \ii L^q$ Carleson embeddings in the following situations: $p=q>2$ (corresponding to \eqref{item-Local l2}), $q' \leq p< 2< q$ (this is \eqref{item-outside-local-L2}), as well as a localized version (situation \eqref{item-local}).

\subsubsection*{\textbf{The case $2<p=q$ for outer measures:}}
We restate the Carleson embedding ``above $L^2$" in the case of a function bounded above by a characteristic function of a set of finite measure (the general case can be deduced through restricted type interpolation).

\begin{proposition}
\label{prop:Carleson-embedding-outer-meas-above-L2}
Let $2<q \leq \infty$ and let $f$ be a function so that $|f| \leq \one_E$. Then
\[
\|F \|_{\ii L^{q}\left( \rr P, \sigma, S  \right)} \lesssim |E|^\frac{1}{q}.
\]
\begin{proof}
Given that $\ds \|F\|_{\ii L^{q}(\rr P, \sigma, S)} \lesssim \| F\|_{\ii L^{q}_{mock}}$ for any $2<q \leq \infty$, the result above follows from Proposition \ref{propo:Carleson-embedding-mock-above-L2}.
\end{proof}
\end{proposition}

\subsubsection*{\textbf{The case $p<2$ for outer measures:}}

For an $L^p \mapsto \ii L^q(\rr P, \sigma, S)$ Carleson embedding with $p<2$, we need to make an extra assumption on $\sssize_{\rr P}(f)$, which is similar to supposing that the maximal function of $f$ is bounded above.
\begin{proposition}
\label{prop:Lq-outer-below-L2}
Assume that $|f|\leq \one_E$ and $\sssize_{\rr P}(f) \leq \tilde{\#}$. Then
\[
\|F\|_{\ii L^q(\rr P, \sigma, S)} \lesssim \min(\tilde{\#, 1})^{1-\frac{2}{q}}\, |E|^\frac{1}{q}.
\]
In particular, if $\tilde{\#}=\# \, |E|$, 
\[
\|F\|_{\ii L^q(\rr P, \sigma, S)} \lesssim \#^{\frac{1}{p}-\frac{1}{q}} |E|^\frac{1}{p},
\]
whenever $q' \leq p <2$.
\begin{proof}
Again, this follows from Propositions \ref{prop:L-mock-below-L2} and \ref{prop:Lq-lessLq-mock}.
\end{proof}
\end{proposition}

If we look at the triples $(p_1, p_2, p_3')$ of Lebesgue exponents for which an estimate $BHT: L^{p_1} \times L^{p_2} \to L^{p_3'}$ can be deduced by using the above Carleson embedding, we again have the constraint
\[
\big| \frac{1}{p_1} -\frac{1}{p_2} \big|<\frac{1}{2}.
\]
This is because 
\[
\big| \frac{1}{p_1} -\frac{1}{p_2} \big| \leq \max \big( \big| \frac{1}{q_1'}-\frac{1}{q_2}  \big|, \,  \big| \frac{1}{q_2'}-\frac{1}{q_1}  \big|  \big) = 1-\frac{1}{q_1}-\frac{1}{q_2}=\frac{1}{q_3}<\frac{1}{2}.
\]

\subsubsection*{\textbf{Localized Carleson embeddings for outer measures}}

Here we want to formulate a local version of the Carleson embedding. We assume that all the tiles have their spatial information contained inside a fixed dyadic interval (i.e., $I_P \subseteq I_0 \, \forall \, P$).

\begin{proposition}
\label{prop:localized-embedding-outer-measures}
If $I_0$ is a fixed dyadic interval, $\rr P(I_0)$ is a collection of tiles so that $I_P \subseteq I_0$ for all $P \in \rr P(I_0)$, $f$ a function satisfying $|f| \leq \one_E$ and $\ds \sssize_{\rr P(I_0)}(f) \leq \tilde{\#}$, then
\[
\|F\|_{\ii L^q \left( \rr P(I_0), \sigma, S \right)} \lesssim \min( \tilde{\#}, 1)^\frac{1}{q'} |I_0|^\frac{1}{q}.
\]
In particular, if $\tilde{\#}= \# \, |E|$, we have
\[
\|F\|_{\ii L^q \left( \rr P(I_0), \sigma, S \right)} \lesssim\big( \# |E| \big)^\frac{1}{p} |I_0|^\frac{1}{q},
\]
for any $q' \leq p$.
\begin{proof}
All that changes through the localization is the collection of tiles (and in consequence also the generating set $\mathbf{E}$). But we still have 
\[
\|F\|_{\ii L^q(\rr P\left(I_0 \right), \sigma, S)} \lesssim \|F\|_{\ii L^q_{mock; I_0}} \quad \forall \, 2< q \leq \infty.
\]
Hence the conclusion follows from Proposition \ref{prop:localized-embedding-mock}.
\end{proof}
\end{proposition}

The Carleson embeddings presented in Propositions \ref{prop:Carleson-embedding-outer-meas-above-L2} and \ref{prop:Lq-outer-below-L2} can be formulated directly for general functions, not only for those bounded above by characteristic functions of sets of finite measure. In the local $L^2$ case, this is precisely Theorem 5.1 of \cite{outer_measures}, and a corresponding Carleson embedding outside local $L^2$ was obtained in \cite{Carleson-embedding-belowL2}.

\section{An inquiry into the variational Carleson operator}
\label{sec:var-Carleson}

The variational Carleson operator defined by
\begin{equation}
\label{def:varCarleson}
\ic{C}^{var, r}(f)(x):=\sup_{K} \sup_{n_0< \ldots < n_K} \big( \sum_{\kappa=1}^K  \big|  \int_{a_{n_{\kappa-1}}}^{a_{n_{\kappa}}} \hat{f}(\xi) \, e^{2 \pi i x \xi} d \xi \, \big|^r    \big)^\frac{1}{r}
\end{equation}
is a generalization of the classical Carleson operator
\[
\ic{C}(f)(x)=  \sup_{N} \big| \int_{-\infty}^N \, \hat{f}(\xi) \, e^{2 \pi i x \xi} d \xi \,   \big|.
\]

The Carleson operator $\ic C$ of \cite{initial_Carleson} played an important role in establishing pointwise convergence of Fourier series of generic functions in $L^p(\rr T)$, via a transference principle. The variational Carleson on the other hand represents a tool for measuring the rate of convergence of Fourier series in $L^p(\rr T)$, and in particular implies their pointwise convergence.

The boundedness of $\ic{C}^{var, r}$ was proved in \cite{variational_Carleson}. In the present section, we will provide a shorter and somehow simpler proof based on \emph{localization}. The local result is interesting in itself, and it implies by means of the helicoidal method both sparse domination and vector-valued estimates, as will be explained in Sections \ref{sec:method-proof} and \ref{sec:sparse-dom}. 

We start by recalling a few results and concepts related to $\ic{C}^{var, r}$.

\begin{theorem}[\cite{variational_Carleson}]\label{thm:C-var}
Let $r>2$. Then $\ds \ic{C}^{var, r} : L^p(\rr R) \to L^p(\rr R)$ for any $r'<p<\infty$, and $\ds \ic{C}^{var, r}: L^{r', 1}(\rr R) \to L^{r', \infty}(\rr R)$. Both conditions $r>2$ and $p>r'$ are necessary.
\end{theorem}

The study of the variational Carleson can be reduced to that of a certain model operator. First, we consider $K \in \rr{Z}^+$ to be fixed and $\lbrace  a_{\kappa}(x)\rbrace_{0 \leq \kappa \leq K}$ measurable complex valued functions so that $ \big(\sum\limits_{\kappa} | a_{\kappa}(x)|^{r'}  \big)^\frac{1}{r'}=1$ and
\[
\ic{C}^{var, r}(f)(x)=\sum_{\kappa=1}^K a_\kappa(x) \, \int_{\rr R} \one_{\left[ \xi_{\kappa-1}\left( x\right), \xi_{\kappa}\left( x \right) \right]}(\xi) \, \hat{f}(\xi) e^{2 \pi i x \xi} d \xi.
\]

Then this can be approximated by a discretized model operator
\[
\ic{C}_{\rr P}^{var, r}(f)(x):=\sum_{P \in \rr P} \langle f, \phi_P \rangle \phi_P(x) a_P(x),
\]
where every tile $P \in \rr P$ is of the form $P=I_P \times \omega_P$ and each $\omega_P$ is the union of three frequency intervals: $\omega_l, \omega_u, \omega_h$. If $\xi_\kappa(\cdot)$ were fixed, the intervals $\omega_u$ would represent a Whitney decomposition of the frequency interval $(\xi_{\kappa-1}, \xi_{\kappa})$. Since this is not the case, the intervals $\omega_l$ and $\omega_h$ are employed in order to capture the low and high frequency information. That is, if $x \in I_P$, there exists at most an index $\kappa$ with $1 \leq \kappa \leq K$ so that $\xi_{\kappa-1}(x ) \in \omega_l$ (and $\xi_\kappa(x) \in \omega_h$). Then $a_P(x)=a_\kappa(x)$ if such such an index exists; otherwise, $a_P(x)=0$. The wave packet $\phi_P$ is associated to the ``Heisenberg box" $I_P \times \omega_u$.

The connection between the frequency intervals $\omega_l, \omega_u$ and $\omega_h$ associated to a tile $P=I_P \times \omega_P$ can be formulated in the following way: there exist constants $ 1 \leq C_3 <C_2<C_1$, such that
\[
\supp (\hat{\phi_P}) \subset C_3\, \omega_u, \quad C_2 \, \omega_u \cap C_2\, \omega_l = \emptyset, \quad C_2 \, \omega_u \cap C_2\, \omega_h = \emptyset, \quad C_2 \, \omega_l \subset C_1\, \omega_u, \quad  C_2 \, \omega_u \subset C_1 \, \omega_l.
\]

As usual, we can assume that $\rr P$ is a finite collection of multi-tiles. $\Xi^{\text{top}}$ denotes the admissible collection of tree tops, and it can be assumed to be finite. The approach will be similar to the one employed for the bilinear Hilbert transform: the tiles will be organized into collections of trees, which are `generating sets" easier to estimate.

\begin{definition}
\label{def:tree-var-Carleson}
A tree $T=(T, I_T, \xi_T)$ is represented by a collection $T \subset \rr P$ of multi-tiles, a spatial top interval $I_T$ and the frequency information $\xi_T \in \Xi^{\text{top}}$, with the property that $P \in T$ provided
\begin{equation}
\label{eq:def-tree-var-C}
I_P \subseteq I_T, \qquad \omega_T:=\big[ \xi_T -\frac{C_2-1}{4 |I_T|}, \xi_T+\frac{C_2 -1}{4 |I_T|}  \big) \subseteq \omega_m:= \text{conv}\, (C_2 \, \omega_l \cup C_2 \, \omega_u).
\end{equation} 

A tree $(T, I_T, \xi_T)$ is called \emph{$l$-overlapping} if $\xi_T \in C_2 \, \omega_l$ for all $P \in T$, and \emph{$l$-lacunary} if $\xi_T \notin C_2 \, \omega_l$ for all $P \in T$.
\end{definition}

Next, we introduce the suitable notions of sizes. Since our method is based on localization, we will not deal with energies as in \eqref{eq:def-energy}. What we refer to as $\ssize^{\mathscr{e}}$ and $\ssize^{\mathscr{m}}$ are called in \cite{variational_Carleson} \emph{energy} and \emph{density}, respectively.

\begin{definition}
Let $\rr P$ be a finite collection of multi-tiles. Then
{\fontsize{10}{10}\begin{equation}
\label{def:size-f-var-C}
\ssize^{\mathscr{e}}_{\rr P}(f):=\sup_{\substack{T \subset \rr P \\ $l$-\text{overlapping tree}}} \big( \frac{1}{|I_T|} \sum_{P \in T} |\langle f, \phi_P   \rangle|^2  \big)^\frac{1}{2}
\end{equation}}
and 
{\fontsize{10}{10}\begin{equation}
\label{def:size-g-var-C}
\ssize^{\mathscr{m}}_{\rr P}(g):= \sup_{P \in \rr P} \sup_{\substack{P' \in \bar{\rr P} \\ P' \geq P, I_{P'} \subset 9 I_P}} \big( \frac{1}{|I_{P'}|} \int_{\rr R} |g(x)|^{r'} \ci_{I_{P'}}^M(x) \cdot \sum_{\kappa=1}^K |a_{\kappa}(x)|^{r'} \cdot \one_{\omega_{P'}}(\xi_{\kappa-1}(x)) \, dx   \big)^{\frac{1}{r'}}.
\end{equation}}
Above, the supremum is taken over all admissible tiles $P'=I_{P'} \times \omega_{P'}$, such that $I_P \subseteq I_{P'}$ and $\omega_{P'} \subseteq \omega_P$, where
\[
\omega_{P'}:=\big[ \xi_{P'}-\frac{C_2-1}{4 |I_{P'}|}, \xi_{P'}+\frac{C_2-1}{4 |I_{P'}|}  \big), \quad \text{for some   } \xi_{P'} \in \Xi^{\text{top}}.
\]
\end{definition}

\begin{rremark}
We have that 
\[
\ssize_{\rr P}^{\mathscr{e}}(f)=\ssize_{\rr P}(f) \leq \sssize_{\rr P}(f),
\]
where $\ssize_{\rr P}(f) $ is the size defined in \eqref{eq:def-size} for the bilinear Hilbert transform. Also, since $\sum_{\kappa=1}^K \vert a_{\kappa}(x)\vert ^{r'} \leq 1$, $\ssize^{\mathscr{m}}_{\rr P}(g)$ can be controlled by the $L^{r'}$ size of (see definition \eqref{def:sssize}):
\[
\ssize^{\mathscr{m}}_{\rr P}(g) \lesssim \sssize^{r'}_{\rr P}(g).
\]
\end{rremark}

These notions of sizes are motivated, as usual, by the \emph{tree estimate}. Though the formulation is slightly different, this is precisely Proposition 5.1 of \cite{variational_Carleson}.

\begin{proposition}\label{prop:tree-est-C-var} Let $T \subset \rr P$ be a tree. Then 
\begin{equation}
\label{eq:tree-est-C-var}
\big \vert  \int_{\rr R} \sum_{P \in T} \langle f, \phi_P  \rangle \phi_P(x) a_P(x) g(x) dx \big \vert \lesssim \ssize^{\mathscr{e}}_{T}(f) \cdot \ssize^{\mathscr{m}}_{T}(g) \cdot \vert I_T \vert
\end{equation}
and also,  
\begin{equation}
\label{eq:tree-est-Lr'}
\big \|\sum_{P \in T} \langle f, \phi_P  \rangle \phi_P\, a_P\, g \big \|_{L^{r'}} \lesssim \ssize^{\mathscr{e}}_{T}(f) \cdot \ssize^{\mathscr{m}}_{T}(g) \cdot \vert I_T \vert^{\frac{1}{r'}}.
\end{equation}
Furthermore, for $\ell \geq 0$, we have 
\[
\big \|\sum_{P \in \rr P} \langle f, \phi_P  \rangle \phi_P \,a_P \, g \big \|_{L^1(\rr R \setminus 2^{\ell} I_T)} \lesssim 2^{-\ell (N-10)} \ssize^{\mathscr{e}}_{T}(f) \cdot \ssize^{\mathscr{m}}_{ T}(g) \cdot \vert I_T \vert
\]
and 
\[
\big \|\sum_{P \in \rr P} \langle f, \phi_P  \rangle \phi_P \, a_P \, g \big \|_{{L^{r'}(\rr R \setminus 2^{\ell} I_T)}} \lesssim 2^{-\ell (N-10)} \ssize^{\mathscr{e}}_{T}(f) \cdot \ssize^{\mathscr{m}}_{ T}(g) \cdot \vert I_T \vert^{\frac{1}{r'}}.
\]
\end{proposition}

The next step consists in decomposing the collection $\rr P$ of tiles, according to $\ssize_{\rr P}^{\mathscr{e}}(f)$ and $\ssize_{\rr P}^{\mathscr{m}}(g)$. This procedure is achieved by iterating the next two lemmas below. As mentioned previously, we work in a localized setting (here $I_0$ is a fixed dyadic interval, and $\rr P(I_0)$ are the tiles such that $I_P \subseteq I_0$); the general case will be obtained as a result of another stopping time. Though we double the number of stopping times, the exceptional set becomes much simpler than the one in \cite{variational_Carleson}. In some sense, instead of removing the trees where the \emph{counting function} $\sum_{ T \in \mathbf{T}} \one_T$ is too large (this is described in Section 7 of \cite{variational_Carleson}), we rearrange the multi-tiles (and hence the trees) through the stopping times.

We proceed with the two decomposition lemmas, which are similar to Propositions 4.3 and 4.4 of \cite{variational_Carleson}. 
\begin{lemma}
\label{lemma:dec-size-f-var-C-local}
Let $I_0$ be a fixed interval, $f$ a locally integrable function such that $\ssize^{\mathscr{e}}_{\rr P\left( I_0 \right)}(f) \leq \mathscr{E}$. Then there exists a decomposition $\rr P(I_0)=\rr P'(I_0) \cup \rr P''(I_0)$ so that $\ssize^{\mathscr{e}}_{\rr P \left( I_0 \right)}(f) \leq \mathscr{E}/2$ and $\rr P''(I_0)$ can be written as a union of disjoint trees $\rr P ''(I_0)=\bigcup\limits_{T \in \mathbf{T}} T$ with the property that
\begin{equation}
\label{eq:decay-count-func-f}
\sum_{T \in \mathbf{T}} |I_T| \lesssim \mathscr{E}^{-2} \, \| f \cdot \ci_{I_0}  \|_2^2.
\end{equation}
\end{lemma}

\begin{lemma}
\label{lemma:dec-size-g-vac-C-local}
Let $I_0$ be a fixed interval, $g$ a locally $L^{r'}$-integrable function such that $\ssize_{\rr P \left( I_0 \right)}^{\mathscr{m}}(g) \leq \lambda$. Then there exists a decomposition $\rr P(I_0)=\rr P'(I_0) \cup \rr P''(I_0)$ so that $\ssize^{\mathscr{m}}_{\rr P \left( I_0 \right)}(g) \leq \lambda/2$ and $\rr P''(I_0)$ can be written as a union of disjoint trees $\rr P ''(I_0)=\bigcup\limits_{T \in \mathbf{T}} T$ with the property that
\begin{equation}
\label{eq:decay-count-func-g}
\sum_{T \in \mathbf{T}} |I_T| \lesssim \lambda^{-r'} \, \| g \cdot \ci_{I_0}  \|_{r'}^{r'}.
\end{equation}
\end{lemma}

The proofs of the above decomposition results follow by classical arguments. In fact, since we don't employ the $BMO$ estimate of $\sum_{T \in \mathbf{T}} \one_T$, Lemma \ref{lemma:dec-size-f-var-C-local} is more similar to Proposition 6.5 of \cite{multilinear_harmonic} than to Proposition 4.4 of \cite{variational_Carleson}. The decay on the right hand side of \eqref{eq:decay-count-func-f} is due to an argument that can be found in Lemma\ref{lemma-local-gen-size-energy}, while in \eqref{eq:decay-count-func-g} we only use the spatial decay of the wave packets, together with the arguments of Proposition 4.4 of \cite{variational_Carleson}.

Now we are ready to present our localization result for $\mathcal{C}^{var, r}$. It is visible from Definition \ref{def:size-g-var-C} that $\ssize_{\rr P}^{\mathscr{m}}(g)$ is an $L^{r'}$ quantity, but in the end we want to obtain an $L^1$ maximal average. The key observation here is that any locally integrable function $g$ can be written as the product of an $L^{r'}$-integrable function $g_2$ and an $L^r$ integrable function $g_1$ simply by setting 
\[
g=g_1 \cdot g_2,\quad \text{ with    }  g_1(x)=|g(x)|^\frac{1}{r}, \quad g_2(x)=\frac{g(x)}{g_1(x)} \text{  if    }g(x) \neq 0 , \quad g_2(x)=0 \text{   if    } g(x)=0.  
\]

Alternatively, we could also work with restricted-type functions (i.e. assume that $|g(x)| \leq \one_G$), but this causes a loss of information for the sparse domination.

\begin{proposition}
\label{prop:local-C-var}
Let $I_0$ be a fixed dyadic interval and $\rr P(I_0)$ a collection of multi-tiles with $I_P \subseteq I_0$. Assume that $F$ is a set of finite measure, $f$ a function satisfying $|f(x)| \leq \one_F(x)$ for a.e. $x$, and $g$ a locally integrable function. Then
\begin{equation}
\label{eq:local-form-C-var}
\big| \Lambda_{\ic C^{var, r}; \rr P(I_0)} (f, g) \big| \lesssim \big(  \sssize_{\rr P \left(I_0\right)} \one_F \big)^{\frac{1}{r'}-\epsilon} \, \big(  \sssize_{\rr P \left( I_0 \right)} g \big) \, |I_0|.
\end{equation}
\begin{proof}

We start by writing $g=g_1 \cdot g_2$, with $g_1 \in L^{r}_{loc}(\rr R)$ and $g_2 \in L^{r'}_{loc}(\rr R)$, as explained previously. Next, we apply iteratively the decomposition procedure of Lemma \ref{lemma:dec-size-f-var-C-local}, obtaining in this way $\ds \rr P(I_0)=\bigcup_{n_1} \bigcup_{T \in \mathbf{T}_{n_1}}T$, where for every $T \in \mathbf{T}_{n_1}$ we have $\ssize_{T}^{\mathscr{e}}(f) \sim 2^{-n_1} \leq \ssize_{\rr P(I_0)}^{\mathscr{e}}(f) \lesssim \sssize_{\rr P(I_0)} \one_F$, and moreover
\[
\sum_{T \in \mathbf{T}_{n_1}} |I_T| \lesssim 2^{2 n_1} \, \| f \cdot \ci_{I_0}  \|_2^2 \lesssim 2^{2 n_1} \, \| \one_F \cdot \ci_{I_0}  \|_2^2.
\]
Similarly, Lemma \ref{lemma:dec-size-g-vac-C-local} applied to the function $g_2$ yields the decomposition $\rr P(I_0):=\bigcup_{n_2} \bigcup_{T \in \mathbf{T}_{n_2}} T$, where for each tree $T \in \mathbf{T}_{n_2}$ we have $\ssize_T^{\mathscr{m}} g_2 \sim 2^{-n_2}$ and 
\[
\sum_{T \in \mathbf{T}_{n_2}} |I_T| \lesssim 2^{r' n_2} \|  g_2 \cdot \ci_{I_0} \|_{r'}^{r'}=2^{r' n_2} \|  g \cdot \ci_{I_0} \|_1.
\] 
Above, we use the fact that $|g_2(x)|^{r'}=|g(x)|$, and for simplicity, we don't explicitly write how $\ci_{I_0}$ changes as well.

Now we are ready to estimate the bilinear form associated to $\ic C^{var, r}_{\rr P(I_0)}$:
\begin{align*}
& \big| \int_{\rr R} \sum_{P \in \rr P(I_0)} \langle f, \phi_P \rangle  \phi_P\, a_P\, g dx   \big| =  \big| \int_{\rr R} \sum_{P \in \rr P(I_0)} \langle f, \phi_P \rangle  \phi_P\, a_P\, g_1 \cdot g_2 dx   \big| \\
&\lesssim \sum_{n_1, n_2} \sum_{T \in \mathbf{T}_{n_1} \cap \mathbf{T}_{n_2}}  \big| \int_{\rr R} \big( \sum_{P \in T} \langle f, \phi_P \rangle  \phi_P\, a_P\, g_1 \cdot g_2 \big) \cdot \one_{I_T} dx   \big| \\
&+ \sum_{\ell \geq 0} \sum_{n_1, n_2} \sum_{T \in \mathbf{T}_{n_1} \cap \mathbf{T}_{n_2}}  \big| \int_{\rr R} \big( \sum_{P \in T} \langle f, \phi_P \rangle  \phi_P\, a_P\, g_1 \cdot g_2 \big) \cdot \one_{2^{\ell +1 }I_T \setminus 2^\ell I_T} dx   \big|:= (I)+(II).
\end{align*}

The second term is a technical variation of the first one and can be dealt with as a result of the fast decay of $\ic C^{var, r}_T$ away from the tree $T$. For that reason, we only focus on $(I)$. For every $T \in \mathbf{T}_{n_1} \cap \mathbf{T}_{n_2}$,
\[
\big| \int_{\rr R} \big( \sum_{P \in T} \langle f, \phi_P \rangle  \phi_P\, a_P\, g_1 \cdot g_2 \big) \cdot \one_{I_T} dx \big| \lesssim \|g_1 \cdot \one_{I_T} \|_r \cdot \| \big( \sum_{P \in T} \langle f, \phi_P \rangle  \phi_P\, a_P\, g_2 \big) \cdot \one_{I_T} \|_{r'}.
\]

Since $g_1$ was defined as $g_1(x)=|g(x)|^\frac{1}{r}$, we have that 
\[
\|g_1 \cdot \one_{I_T} \|_r= \|g_1 \cdot \one_{I_T} \|_1^{\frac{1}{r}} \lesssim \big( |I_T| \cdot \sssize_{\rr P(I_0)} g \big)^\frac{1}{r}.
\]
Also, since $T \in \mathbf{T}_{n_1} \cap \mathbf{T}_{n_2}$, Proposition \ref{prop:tree-est-C-var} yields 
\[
\| \big( \sum_{P \in T} \langle f, \phi_P \rangle  \phi_P\, a_P\, g_2 \big) \cdot \one_{I_T} \|_{r'} \lesssim 2^{-n_1} \, 2^{-n_2} \, |I_T|^\frac{1}{r'}.
\]

This implies 
{\fontsize{10}{10}\[
(I) \lesssim \sum_{n_1, n_2} \sum_{T \in \mathbf{T}_{n_1} \cap \mathbf{T}_{n_2}}  2^{-n_1} \, 2^{-n_2} \, |I_T|^\frac{1}{r'} \big( |I_T| \cdot \sssize_{\rr P(I_0)} g \big)^\frac{1}{r} = \sum_{n_1, n_2} \sum_{T \in \mathbf{T}_{n_1} \cap \mathbf{T}_{n_2}}  2^{-n_1} \, 2^{-n_2} \, |I_T| \, \big( \cdot \sssize_{\rr P(I_0)} g \big)^\frac{1}{r}.
\]}

For the sum of the tops of the trees we have the inequality
\[
\sum_{T \in \mathbf{T}_{n_1} \cap \mathbf{T}_{n_2}} |I_T| \lesssim \min ( \sum_{T \in \mathbf{T}_{n_1}} |I_T|, \sum_{T \in \mathbf{T}_{n_2}} |I_T| ) \lesssim \min( 2^{2 n_1} \| \one_F \cdot \ci_{I_0} \|_2^2, 2^{r' n_2} \| g \cdot \ci_{I_0} \|_1).
\]
We use interpolation, in the usual way: if $0 \leq \theta_1, \theta_2 \leq 1$ with $ \theta_1+\theta_2=1$, then 
{\fontsize{10}{10} \[
 \sum_{T \in \mathbf{T}_{n_1} \cap \mathbf{T}_{n_2}} |I_T| \lesssim \big(  2^{2 n_1} \| \one_F \cdot \ci_{I_0} \|_2^2 \big)^{\theta_1} \,\big( 2^{r' n_2} \| g \cdot \ci_{I_0} \|_1)^{\theta_2} \lesssim \big(  2^{2 n_1}\, |I_0|\, \sssize_{\rr P(I_0)} \one_F \big)^{\theta_1} \,\big( 2^{r' n_2}\, |I_0| \, \sssize_{\rr P(I_0)} g )^{\theta_2}.
 \]}

Hence
\[
(I) \lesssim \sum_{n_1, n_2} 2^{-n_1 \left( 1- 2 \theta_1 \right)} 2^{-n_2 \left( 1- r' \theta_2  \right)} \big( \sssize_{\rr P(I_0)}\one_F \big)^{\theta_1} \, \big(\sssize_{\rr P(I_0)}g  \big)^{\theta_2 +\frac{1}{r}} \cdot |I_0|.
\]

We recall that the decomposition procedures of Lemmas \ref{lemma:dec-size-f-var-C-local} and \ref{lemma:dec-size-g-vac-C-local} continue provided $2^{-n_1} \lesssim \sssize_{\rr P(I_0)} \one_F$  and $2^{-n_2} \lesssim \sssize^{r'}_{\rr P(I_0)}g_2 \lesssim \big( \sssize_{\rr P(I_0)}g  \big)^\frac{1}{r'}$.

As a consequence, if $\theta_1, \theta_2$ are so that $1- 2 \theta_1  >0$ and $1- r' \theta_2 >0$, we obtain
\begin{align*}
(I) &\lesssim \big( \sssize_{\rr P(I_0)} \one_F \big)^{1- 2 \theta_1} \cdot \big( \sssize_{\rr P(I_0)}g  \big)^{\frac{1}{r'} \left( 1- r' \theta_2 \right)  } \cdot  \big( \sssize_{\rr P(I_0)}f  \big)^{\theta_1} \, \big(\sssize_{\rr P(I_0)}g  \big)^{\theta_2 +\frac{1}{r}} |I_0| \\
&\lesssim \big( \sssize_{\rr P(I_0)} \one_F \big)^{1- \theta_1} \cdot \big( \sssize_{\rr P(I_0)}g  \big) \cdot |I_0|.
\end{align*}

The conditions on $\theta_1$ and $\theta_2$ can be summed up as $\ds \frac{1}{r}< \theta_1 < \frac{1}{2}$; by taking $\theta_1=\frac{1}{r} +\epsilon$, we obtain \eqref{eq:local-form-C-var}, which ends the proof.
\end{proof}
\end{proposition}

\begin{rremark}
Notice that Proposition \ref{prop:local-C-var} is stating that the bilinear form associated to $\ic C^{var, r}$, localized to the interval $I_0$, is controlled by the product of two maximal averages, which are themselves adapted to the interval $I_0$ and the information contained in $\rr P(I_0)$. Both $\ssize_{\rr P}^{\mathscr{e}}$ and $\ssize_{\rr P}^{\mathscr{m}}$ are replaced by the $L^1$ maximal average $\sssize_{\rr P}$. 
\end{rremark}

Once we have the local result of Proposition \ref{prop:local-C-var}, we obtain the global result of Theorem \ref{thm:C-var} by performing a double stopping time (for each of the functions $f$ and $g$, both of which are now bounded above by characteristic functions of certain sets). The analysis is similar to the one presented in Section \ref{sec:BHT-avoid-interpolation} for the bilinear Hilbert transform, and for that reason, we don't write down the details. 

We note however that the local estimate of Proposition \ref{prop:local-C-var} is sufficient for obtaining further multiple vector-valued estimates, as well as sparse domination. The presentation of the exact scheme makes the subject of the next two sections. As a consequence of the sparse domination, we have a Fefferman-Stein inequality for the variational Carleson operator (the case $r=\infty$ corresponds to the Carleson operator of \cite{initial_Carleson}):

\begin{corollary}
\label{cor:Feff-Stein-varCarleson}
For any $2<r \leq \infty$, $0<p<\infty$, $\epsilon>0$, and any $m$-tuple $R=(r^1, \ldots, r^m)$ with $r'< r^j<\infty$, we have 
\[
\big\| \big\| \ic C^{var, r} f(x, \cdot) \big\|_{ L^R(\ii W, \mu)}  \big\|_p \lesssim \big\|  \ic M_{r'+\epsilon} \big( \| f(x, \cdot)\|_{L^R(\ii W, \mu)} \big)\big\|_p.
\]
\end{corollary}

\smallskip

\section{The method of the proof: vector-valued inequalities}
\label{sec:method-proof}

We now present the ideas behind the proof of the vector-valued estimates, focusing on the $BHT$ operator. First, we study the Banach case: $1 < r_1^j, r_2^j  \leq \infty , 1\leq r^j <\infty$ and $\dfrac{1}{r_1^j}+ \dfrac{1}{r_2^j}=\dfrac{1}{r^j}$ for all $1 \leq j \leq n$.

The interpolation theory extends without difficulty to the vector-valued setting: if $\ds \vec f : \rr R \to L^{R_1}\left( \ii W, \mu \right) , \vec g : \rr R \to L^{R_2}\left( \ii W, \mu \right)$ and $\ds\vec h : \rr R \to L^{R'}\left( \ii W, \mu \right)$ are so that
\[ \| \vec f_{w}(x)  \|_{L^{R_1}\left( \ii W, \mu \right)} \leq \one_F(x), \quad  \| \vec g_{w}(x)  \|_{L^{R_2}\left( \ii W, \mu \right)} \leq \one_G(x), \quad  \| \vec h_{w}(x)  \|_{L^{R'}\left( \ii W, \mu \right)} \leq \one_{H'}(x),
\]
then it is enough to check that 
\[
\vert  \int_{\ii W}\Lambda_{BHT(\rr P)}(\vec f_w, \vec g_w, \vec h_w)  d w\vert \lesssim \vert F \vert^{\alpha_1} \cdot \vert G \vert^{\alpha_2} \cdot  \vert H \vert^{\alpha_3},
\]
for $(\alpha_1, \alpha_2, \alpha_3)$ arbitrarily close to $(\frac{1}{p}, \frac{1}{q}, \frac{1}{s'})$, with $\alpha_1+\alpha_2+\alpha_3=1$. By scaling invariance, we can assume that $|H|=1$.

\begin{rremark}
If we are not in the Banach setting, the only possibility is that $r^j<1$ for some $1 \leq j \leq n$. In that case, it is enough to prove
\[
\big\|   \big\|  BHT (\vec f_w, \vec g_w)  \big\|_{L^{R}_w\left( \ii W, \mu \right)}\big\|_{L^{\tilde s, \infty}_x} \lesssim \vert F\vert^{\frac{1}{s_1}}  \cdot \vert G \vert^{\frac{1}{s_2}},
\]
for any $(s_1, s_2, \tilde s)$ H\"older tuple arbitrarily close to $(p, q, s)$, provided $\ds \| \vec f_{w}(x)  \|_{L^{R_1}\left( \ii W, \mu \right)} \leq \one_F(x), \quad  \| \vec g_{w}(x)  \|_{L^{R_2}\left( \ii W, \mu \right)} \leq \one_G(x)$.

The interpolation theory and the suitable dualization of $L^{\tilde s, \infty}$ are explained in detail in \cite{quasiBanachHelicoid}.
\end{rremark}

The result in Theorem \ref{thm:main-thm-BHT} is proved by induction: assuming a local vector-valued inequality of ``depth $n-1$" (if $n=1$, that amounts to the scalar-valued result) with sharp estimates for the operatorial norm, we can prove the vector-valued inequality of depth $n$ (global or local). 

Localizations play a very important role in our proof, as they also do in the proof of the sparse domination.

Given a fixed dyadic interval $I_0$, we recall the definition
\[
\rr P(I_0):= \lbrace P \in \rr P: I_P \subseteq I_0  \rbrace.
\]

Also, since the interval $I_0$ and the collection $\rr P(I_0)$ are fixed, we use the notation
\[
\sssize_{I_0} \, f:=\sssize_{\rr P(I_0)} \, f = \max \big( \frac{1}{\vert I_0 \vert} \int_{\rr R} \vert f(x)   \vert  \cdot \ci_{I_0}(x) dx ,\, \sup_{P \in \rr P(I_0)}\frac{1}{\vert I_P \vert} \int_{\rr R} \vert f(x)   \vert  \cdot \ci_{I_P}(x) dx \big).
\]

Ultimately, the vector-valued inequality reduces to proving inductively the following two statements for $\ds \Lambda_{BHT; \rr P(I_0)}^{F, G, H}(f, g, h):=\Lambda_{BHT; \rr P(I_0)} (f \cdot \one_F, g \cdot \one_G, h \cdot \one_{H'})$:
{\fontsize{10}{10}\begin{align}
\label{eq-helM-n}
\tag*{$\ii P (n)$}
 \lft \Lambda_{BHT; \rr P(I_0)}^{F, G, H'}(\vec f,\vec g, \vec h) \rg &\lesssim \left( \sssize_{I_0} \one_F \right)^{\frac{1+\theta_1}{2}-\frac{1}{r_1}-\epsilon} \left( \sssize_{I_0} \one_G \right)^{\frac{1+\theta_2}{2}-\frac{1}{r_2}-\epsilon} \left( \sssize_{I_0} \one_{H'} \right)^{\frac{1+\theta_3}{2} -\frac{1}{r'}-\epsilon}\\
& \qquad \cdot  \big\| \big\|  \vec f\big\|_{L^{R^n_1}} \cdot \ci_{I_0} \big\|_{r_1} \big\| \big\| \vec g  \big\|_{L^{R^n_2}}\cdot \ci_{I_0} \big\|_{r_2} \big\| \big\| \vec h \big\|_{L^{\left(R^n\right)'}}\cdot \ci_{I_0} \big\|_{r'} \nonumber
\end{align}}
and, for functions $\vec f, \vec g, \vec h$ satisfying $\big\|\vec f(x) \big\|_{L^{R^n_1}}\leq \one_F(x), \big\| \vec g(x) \big\|_{L^{R^n_2}}\leq \one_G(x)$ and $\big\| \vec h(x) \big\|_{L^{\left(R^n\right)'}}\leq \one_{H'}(x)$ respectively, 
{\fontsize{10}{10}\begin{align}
\label{eq-helM-n-sizes-only}
\tag*{$\ii P^* (n)$} \lft \Lambda_{BHT; \rr P(I_0)}^{F, G, H'}(\vec f, \vec g, \vec h) \rg \lesssim \left( \sssize_{I_0} \one_F \right)^{\frac{1+\theta_1}{2}-\epsilon} \left( \sssize_{I_0} \one_G \right)^{\frac{1+\theta_2}{2}-\epsilon} \left( \sssize_{I_0} \one_{H'} \right)^{\frac{1+\theta_3}{2}-\epsilon} \cdot \lft I_0 \rg.
\end{align}}

In $\ii P^*(n)$, we need the exponents of the sizes to be positive:
\[
\frac{1+\theta_1}{2}-\frac{1}{r_1} >0, \quad \frac{1+\theta_2}{2}-\frac{1}{r_2}>0, \quad \frac{1+\theta_3}{2}-\frac{1}{r'}>0.
\]
Hence our restrictions for $\ii D_{r_1, r_2, r'}$ and $\ii D_{R_1, R_2, R'}$. The conditions above are automatically satisfied in the local $L^2$ case, when $2 \leq r_1, r_2, r' \leq \infty$, and in consequence, in that situation, we have $\ii D_{r_1, r_2, r'}=Range(BHT)$.

The proof of Theorem \ref{thm:main-thm-BHT} follows by induction: $\ii P(n) \Rightarrow \ii P^*(n+1)$ is a consequence of H\"older's inequality, and $\ii P^*(n) \Rightarrow \ii P(n)$ follows from restricted-type interpolation. A more accurate version of $\ii P^*(n)$ is the following: if $\vec f, \vec g, \vec h$ are vector-valued functions satisfying $\big\|\vec f(x) \big\|_{L^{R^n_1}}\leq \one_{E_1}(x), \big\| \vec g(x) \big\|_{L^{R^n_2}}\leq \one_{E_2}(x)$ and $\big\| \vec h(x) \big\|_{L^{\left(R^n\right)'}}\leq \one_{E_3'}(x)$ respectively, then
{\fontsize{9.5}{10}\begin{align}
\label{eq-helM-n-sizes-only-fns}
\tag*{$\tilde{\ii P^{*}} (n):$}  \lft \Lambda_{BHT; \rr P(I_0)}^{F, G, H'}(\vec f,\vec g, \vec h) \rg &\lesssim \left( \sssize_{I_0} \one_F \right)^{{\small \frac{1+\theta_1}{2}-\frac{1}{r_1}-\epsilon}} \left( \sssize_{I_0} \one_G \right)^{{\small \frac{1+\theta_2}{2}-\frac{1}{r_2}-\epsilon}} \left( \sssize_{I_0} \one_{H'} \right)^{{\small \frac{1+\theta_3}{2} -\frac{1}{r'}-\epsilon}}\\
& \qquad \cdot |E_1|^\frac{1}{r_1}\,  |E_2|^\frac{1}{r_2}\,  |E_3|^\frac{1}{r'}\,  \nonumber.
\end{align}}
All of the above statements will be made precise in the vector-valued case (that is, for depth-$1$ vector-valued inequalities); the multiple vector-valued case follows by iteration.

\smallskip
We note that $\ii P^*(0)$ is precisely the statement of Lemma \ref{lemma-local-gen-size-energy} in Section \ref{sec:localization} and it is an immediate consequence of \eqref{eq:BHT-en-size}, where the $L^2$ norms are localized as well: if $\vert f(x) \vert \leq \one_F(x), \vert g(x) \vert \leq \one_G(x)$ and $\vert h(x) \vert \leq \one_{H'}(x)$, then \eqref{eq:BHT-en-size} implies that 
{\fontsize{10}{10}\begin{align*}
\vert \Lambda_{BHT; \rr P(I_0)}(f, g, h) \vert &\lesssim \big( \ssize_{\rr P} \,f \big)^{\theta_1} \cdot \big( \ssize_{\rr P}\, g \big)^{\theta_2} \cdot  \big( \ssize_{\rr P} \, h \big)^{\theta_3} \cdot  \|f \cdot \ci_{I_0}\|_2^{1-\theta_1} \cdot  \|g \cdot \ci_{I_0}\|_2^{1-\theta_2} \cdot   \|h \cdot \ci_{I_0}\|_2^{1-\theta_3} \\
&\lesssim  \left( \sssize_{I_0} \one_F \right)^{\frac{1+\theta_1}{2}-\epsilon} \left( \sssize_{I_0} \one_G \right)^{\frac{1+\theta_2}{2}-\epsilon} \left( \sssize_{I_0} \one_{H'} \right)^{\frac{1+\theta_3}{2}-\epsilon} \cdot \lft I_0 \rg,
\end{align*}}
since $\theta_1+\theta_2+\theta_3=1$.

However, for proving vector-valued estimates we need the stronger inequality $\ii P(0)$, which requires a more careful analysis: we need to estimate in an optimal way the operatorial norm associated to $\Lambda_{BHT; \rr P(I_0)}^{F, G, H'}$.

In fact, $\ii P(0)$ can be obtained from $\ii P^*(0)$ via generalized restricted type interpolation, making use of a triple stopping time which allows us to recover $L^p$ norms from sizes. Given $E_1, E_2$ and $E_3$ sets of finite measure, first we construct a major subset $E_3'$ of $E_3$ by setting $E_3':=E_3 \setminus \tilde \Omega$, where
\[
\tilde \Omega:=\big\lbrace  x: \ic M(\one_{E_1})> C \frac{\vert E_1 \vert}{\vert E_3 \vert} , \, \ic M(\one_{E_2})> C \frac{\vert E_2 \vert}{\vert E_3 \vert}   \big\rbrace.
\]
Then, for any functions $f, g, h$ so that  $\vert f(x) \vert \leq \one_{E_1}, \vert g(x) \vert \leq \one_{E_2}$ and $\vert h(x) \vert \leq \one_{E_3'}$, we want to prove that 
\begin{equation}
\label{eq:restr-L}
\vert \Lambda_{BHT; \rr P(I_0)}^{F, G, H'}(f, g, h) \vert \lesssim \| \Lambda_{I_0}^{F, G, H'} \| \cdot \vert E_1 \vert^{\frac{1}{r_1}}\cdot \| \vert E_2 \vert^{\frac{1}{r_2}} \cdot \| \vert E_3 \vert^{\frac{1}{r'}},
\end{equation}
where $\|\Lambda_{I_0}^{F, G, H'}\|$ represents the ``operatorial norm"
\[
\|\Lambda_{I_0}^{F, G, H'}\|:=\left( \sssize_{I_0} \one_F \right)^{\frac{1+\theta_1}{2}-\frac{1}{r_1}-\epsilon} \left( \sssize_{I_0} \one_G \right)^{\frac{1+\theta_2}{2}-\frac{1}{r_2}-\epsilon} \left( \sssize_{I_0} \one_{H'} \right)^{\frac{1+\theta_3}{2} -\frac{1}{r'}-\epsilon}.
\]

In the model operator for $BHT$, we can assume that 
\begin{equation}
\label{eq:assumption-d}
1+\frac{\dist(I_P, \tilde \Omega^c)}{\vert I_P \vert } \sim 2^d, \quad \text{for  } d\geq 0.
\end{equation}
This will allow us to say that 
\[
\sssize_{\rr P}\, \one_{E_1} \lesssim 2^d \frac{|E_1|}{|E_3|}, \quad \sssize_{\rr P} \,\one_{E_2} \lesssim 2^d \frac{|E_2|}{|E_3|}, \quad \sssize_{\rr P} \,\one_{E'_3} \lesssim 2^{-Md},
\]
and summing in $d$ will not be an issue because of the decay in $\sssize_{\rr P} \one_{E'_3}$.

The triple stopping time (similar to that of Section \ref{sec:localization}) will allow us to write $\ds \rr P(I_0):=\bigcup_{n_1, n_2, n_3} \bigcup_{I \in \ii I^{n_1, n_2, n_3}} \rr P(I)$, where for each subcollection $\rr P(I)$ we know
\[
\sssize_{\rr P(I)} \one_{E_1} \lesssim \min( 2^{-n_1},2^d \frac{|E_1|}{|E_3|}), \, \sssize_{\rr P(I)} \one_{E_2} \lesssim \min( 2^{-n_2},2^d \frac{|E_2|}{|E_3|}), \,\sssize_{\rr P(I)} \one_{E'_3} \lesssim \min( 2^{-n_3},2^{-Md} ).
\]

Now we use $\ii P^*(0)$ on the interval $I \subseteq I_0$, with the sets $F, G, H$ replaced by $F \cap E_1, G \cap E_2$ and $H \cap E_3$ respectively:
\begin{align*}
 \lft \Lambda_{BHT; \rr P(I)}^{F, G, H'}(\vec f,\vec g, \vec h) \rg &\lesssim \left( \sssize_{I} \one_F \right)^{\frac{1+\theta_1}{2}-\alpha_1-\epsilon} \left( \sssize_{I} \one_G \right)^{\frac{1+\theta_2}{2}-\alpha_2-\epsilon} \left( \sssize_{I} \one_{H'} \right)^{\frac{1+\theta_3}{2} -\alpha_3-\epsilon}\\
& \qquad \cdot 2^{-n_1 \cdot \alpha_1} \cdot  2^{-n_2 \cdot \alpha_2} \cdot  2^{-n_3 \cdot \alpha_3} \cdot |I|,
\end{align*}
for some $0 \leq \alpha_j \leq \dfrac{1+\theta_j}{2}$.

Using the monotonicity of the $\sssize$ ( since $I \subseteq I_0$ and $\rr P(I) \subseteq \rr P(I_0)$, we have $\sssize_I \,f \leq \sssize_{I_0} f$), and an estimate of  $\ds \sum_{I \in \ii I^{n_1, n_2, n_3}} |I|$, we can recover \eqref{eq:restr-L}, with an extra $2^{-10 d}$ decay as a consequence of assumption \eqref{eq:assumption-d}.

Now we sketch the proof of the implication $\ii P(0) \Rightarrow \ii P^*(1)$. For clarity, we assume in the presentation that the vector-spaces involved are $\ell^{r_1}, \ell^{r_2}$ and $\ell^r$. Then, for sequences of function $\lbrace f_k \rbrace_k, \lbrace g_k \rbrace_k$ and $\lbrace h_k \rbrace_k$ satisfying 
\[
\|\lbrace f_k \rbrace\|_{\ell^{r_1}} \leq \one_F(x), \quad \|\lbrace g_k \rbrace\|_{\ell^{r_2}} \leq \one_G(x), \quad \|\lbrace h_k \rbrace\|_{\ell^{r'}} \leq \one_{H'}(x), \quad 
\]
we want to prove
\[
\vert \sum_k \Lambda_{BHT; \rr P(I_0)} (f_k, g_k, h_k) \vert \lesssim \left( \sssize_{I_0} \one_F \right)^{\frac{1+\theta_1}{2}-\epsilon} \left( \sssize_{I_0} \one_G \right)^{\frac{1+\theta_2}{2}-\epsilon} \left( \sssize_{I_0} \one_{H'} \right)^{\frac{1+\theta_3}{2}-\epsilon} \cdot \lft I_0 \rg.
\]

First, we note that $\Lambda_{BHT; \rr P(I_0)} (f_k, g_k, h_k) =\Lambda_{BHT; \rr P(I_0)} ^{F, G, H'}(f_k, g_k, h_k)$ and using $\ii P(0)$ we have
{\fontsize{10}{10}\begin{align*}
\vert \Lambda_{BHT; \rr P(I_0)} (f_k, g_k, h_k)  \vert & \lesssim \left( \sssize_{I_0} \one_F \right)^{\frac{1+\theta_1}{2}-\frac{1}{r_1}-\epsilon} \left( \sssize_{I_0} \one_G \right)^{\frac{1+\theta_1}{2}-\frac{1}{r_2}-\epsilon} \left( \sssize_{I_0} \one_{H'} \right)^{\frac{1+\theta_1}{2} -\frac{1}{r'}-\epsilon} \\
& \qquad \cdot  \big\|   f_k \cdot \ci_{I_0} \big\|_{r_1} \big\|   f_k \cdot \ci_{I_0} \big\|_{r_2} \big\|   h_k \cdot \ci_{I_0} \big\|_{r'}.
\end{align*}}

Summing in $k$ via H\"older's inequality the expressions in the second line, we obtain
\[
\frac{ \|  \one_{F} \cdot \ci_{I_0} \|_{r_1}}{\vert I_0 \vert^\frac{1}{r_1}} \cdot \frac{ \|  \one_{G} \cdot \ci_{I_0} \|_{r_2}}{\vert I_0 \vert^\frac{1}{r_2}} \cdot \frac{ \|  \one_{H'} \cdot \ci_{I_0} \|_{r'}}{\vert I_0 \vert^\frac{1}{r'}} \cdot \vert I_0 \vert,
\]
which is bounded above by
\[
\left( \sssize_{I_0} \one_F \right)^{\frac{1}{r_1}} \left( \sssize_{I_0} \one_G \right)^{\frac{1}{r_2}} \left( \sssize_{I_0} \one_{H'} \right)^{\frac{1}{r'}} \cdot \lft I_0 \rg.
\]
This end the proof of ``$\ii P(0) \Rightarrow \ii P^*(1)$".

\subsection{The quasi-Banach case}
If $r^j<1$ for some $1 \leq j \leq n$, the inductive statements for the localized trilinear form $\Lambda_{BHT; \rr P(I_0)}^{F, G, H'}$ are replaced by statements for the localized operator
\[
BHT_{I_0}^{F, G, H'}(f, g)(x):=BHT_{I_0}(f \cdot \one_F, g \cdot \one_G)(x) \cdot \one_{H'}(x).
\]
Then, the two inductive statements are:
{\fontsize{9}{10}
\begin{align}
\label{Pn-bht}
\tag*{$\ii P(n):$} 
\big\|  \big\|   BHT_{\rr P \left( I_0 \right)}^{F, G, H'} \left(\vec f, \vec g  \right)\big \|_{L^{R^n}} \big\|_s & \lesssim \left( \sssize_{\rr P\left( I_0 \right)} \one_F  \right)^{\frac{1+\theta_1}{2}-\frac{1}{p}-\epsilon} \left( \sssize_{\rr P\left( I_0\right)} \one_G  \right)^{\frac{1+\theta_2}{2}-\frac{1}{q}-\epsilon} \left( \sssize_{\rr P\left( I_0 \right)} \one_{H'}  \right)^{\frac{1+\theta_3}{2}-\frac{1}{s'}-\epsilon} \\
&\cdot \quad \big \|  \big\|  \vec f \big\|_{L^{R^n_1}} \cdot \ci_{I_0} \big\|_p \big \|  \big\|  \vec g\big\|_{L^{R^n_2}} \cdot \ci_{I_0} \big\|_q. \nonumber
\end{align}}
Also, whenever $\big\|  \vec f (x)\big\|_{L^{R^n_1}} \leq \one_{F}(x)$ and $\big\|  \vec g(x)\big\|_{L^{R^n_2}} \leq \one_G(x)$, we have
{\fontsize{9}{10}\begin{align}
\label{Pn-star-bht}
\tag*{$\ii P^*(n):$} 
\big\|  \big\|   BHT_{\rr P \left( I_0 \right)}^{F, G, H'} \left(\vec f, \vec g  \right)\big \|_{L^{R^n}} \big\|_s & \lesssim \left( \sssize_{\rr P\left( I_0 \right)} \one_F  \right)^{\frac{1+\theta_1}{2}-\epsilon} \left( \sssize_{\rr P\left( I_0\right)} \one_G  \right)^{\frac{1+\theta_2}{2}-\epsilon} \left( \sssize_{\rr P\left( I_0 \right)} \one_{H'}  \right)^{\frac{1+\theta_3}{2}-\frac{1}{s'}-\epsilon} \cdot \lft I_0 \rg^{1/s}.
\end{align}}
 
Notice that we obtain a gain in the exponent for $\sssize_{\rr P(I_0)} \one_{H'}$, if $s<1$, compared to the similar estimate for the trilinear form. This is a key fact in obtaining also the quasi-Banach valued extensions. 
 
\section{The method of the proof: sparse domination} 
\label{sec:sparse-dom}

In Section \ref{sec:localization}, we have obtained implicitly that
{\fontsize{9}{9}\begin{align*}
\vert \Lambda_{BHT; \rr P} (f, g, h)  \vert \lesssim \sum_{\substack{S_1 \in \ic S_1 \\ S_2 \in \ic S_2 \\S_3 \in \ic S_3}}   \big( \sssize_{\rr P (S_1)} \one_F  \big)^\frac{1+\theta_1}{2} \cdot   \big(  \sssize_{\rr P (S_2)} \one_G  \big)^\frac{1+\theta_2}{2} \cdot  \big( \sssize_{\rr P (S_3)} \one_{H'}  \big)^\frac{1+\theta_3}{2} \cdot |S_1 \cap S_2 \cap S_3|,
\end{align*}}where the functions $f, g, h$ satisfy $|f| \leq \one_F, |g| \leq \one_G, |h| \leq \one_{H'}$, and  $0 \leq \theta_1, \theta_2, \theta_3 <1$, with $\theta_1+\theta_2 +\theta_3 =1$.

Such estimates are sometimes called in literature ``sparse estimates", since the collections of intervals $\ic S_1, \ic S_2$ and $\ic S_3$ are sparse. In fact, we've seen that each of them verifies \eqref{eq:Carleson-condition}, which is called a $\tilde C$-``Carleson condition", and which in turn is equivalent to $\frac{1}{\tilde C}$-sparse property of Definition \ref{def-sparse-collection}, as shown in \cite{Nazarov-Lerner-DyadicCalculus}.

In applications, ideally one would like to obtain similar estimates, but with only one sparse set rather than three, and also one would like it to be available for arbitrary functions, not only for those bounded above by characteristic functions of sets. The goal of this section is to describe our recent work in \cite{sparse-hel}, which realizes this task.

The local estimate \eqref{eq-helM-n-sizes-only} of Section \ref{sec:method-proof} is formulated for functions that are bounded above by characteristic functions of finite sets. However, at the cost of an $\epsilon$ in the exponent of the sizes, we obtain a similar estimate for general functions.
 
\begin{definition}
If $\rr P$ is a collection of time-frequency tiles, then we define the $L^q$ size to be
\[
\sssize_{\rr P}^q(f) := \sup_{P \in \rr P} \Big( \frac{1}{|I_P|} \int_{\rr R} |f(x)|^q \cdot \ci_{I_P} dx  \Big)^\frac{1}{q}.
\]

In the case of a collection $\rr P(I_0)$ localized to a fixed dyadic interval $I_0$ (all the tiles in the collection satisfy $I_P \subseteq I_0$ ), we have 
\[
\sssize_{\rr P(I_0)}^q (f) := \max \Big(  \sup_{P \in \rr P(I_0)} \Big( \frac{1}{|I_P|} \int_{\rr R} |f(x)|^q \cdot \ci_{I_P} dx  \Big)^\frac{1}{q}, \Big( \frac{1}{|I_0|} \int_{\rr R} |f(x)|^q \cdot \ci_{I_0} dx  \Big)^\frac{1}{q} \Big).
\]
\end{definition} 
We note that the $L^1$ size coincides with the classical $\sssize_{\rr P}(f)$, and also, 
\[ \sssize_{\rr P}^q(f) = \big(\sssize_{\rr P}( |f |^q) \big)^\frac{1}{q}, \quad \sssize_{\rr P(I_0)}^q(f) = \big(\sssize_{\rr P(I_0)}( |f |^q) \big)^\frac{1}{q}. 
\]
 
Using an argument similar to interpolation we can prove the following
\begin{proposition}[Mock interpolation]
\label{prop:restricted->general}
Let $\rr P(I_0)$ be a collection of tiles and $\Lambda_{\rr P(I_0)}$ a trilinear form satisfying, for any sets of finite measure $F, G$ and $H$, and any functions functions $f, g, h$ so that $|f| \leq \one_F, |g| \leq \one_G, |h| \leq \one_{H}$ the estimate
\[
\vert \Lambda_{\rr P(I_0)} (f, g, h)  \vert \lesssim \sssize_{\rr P(I_0)}^{q_1} (\one_F) \cdot \sssize_{\rr P(I_0)}^{q_2} (\one_G) \cdot  \sssize_{\rr P(I_0)}^{q_3} (\one_H) \, |I_0|.
\]
Then, for any $\epsilon>0$, and any locally integrable functions $f, g, h$ we have 
\[
\vert \Lambda_{\rr P(I_0)} (f, g, h)  \vert \lesssim_\epsilon \sssize_{\rr P(I_0)}^{q_1+\epsilon} (f) \cdot \sssize_{\rr P(I_0)}^{q_2+\epsilon} (g) \cdot  \sssize_{\rr P(I_0)}^{q_3+\epsilon} (h) \, |I_0|.
\]
\begin{proof}
First, note that the trilinear form above is not the one associated to the model operator of $BHT$. Instead, it can be defined on a collection of dyadic intervals, since $\sssize_{\rr P(I_0)} (f)$ only depends on the spatial intervals. For simplicity, we can assume that the functions are Schwartz. The statement remains true for more general collections $\rr P$ which need not be localized on an interval $I_0$.

Then we decompose $f=\sum_{\kappa_1} 2^{\kappa_1} \cdot f_{\kappa_1}=\sum_{\kappa_1} 2^{\kappa_1} \cdot f_{\kappa_1} \cdot \one_{F_{\kappa_1}}$, where $$F_{\kappa_1}=\lbrace 2^{\kappa_1-1} \leq |f| \leq 2^{\kappa_1}  \rbrace.$$
That is, $ 2^{\kappa_1} f_{\kappa_1}=f \cdot \one_{F_{\kappa_1}}$ represents the part of $f$ where $|f| \sim 2^{\kappa_1}$, and $|f_{\kappa_1}| \leq \one_{F_{\kappa_1}}$, so it verifies the hypotheses of the Proposition. 

Similarly, $g=\sum_{\kappa_2} 2^{\kappa_2} \cdot g_{\kappa_2}$ and $h= \sum_{\kappa_3} 2^{\kappa_3} \cdot h_{\kappa_3}$, while
\begin{align*}
\vert \Lambda_{\rr P(I_0)} (f, g, h)  \vert &\lesssim \sum_{\kappa_1, \kappa_2, \kappa_3} 2^{\kappa_1}\, 2^{\kappa_2}\, 2^{\kappa_3} \, \vert \Lambda_{\rr P(I_0)} (f_{\kappa_1}, g_{\kappa_2}, h_{\kappa_3})  \vert \\
& \lesssim \sum_{\kappa_1, \kappa_2, \kappa_3} 2^{\kappa_1}\, 2^{\kappa_2}\, 2^{\kappa_3} \, \sssize_{\rr P(I_0)}^{q_1} (\one_{F_{\kappa_1}}) \cdot \sssize_{\rr P(I_0)}^{q_2} (\one_{G_{\kappa_2}}) \cdot  \sssize_{\rr P(I_0)}^{q_3} (\one_{H_{\kappa_3}}) \cdot |I_0|.
\end{align*}

Then it suffices to prove that 
\[
\sum_{\kappa} 2^{\kappa}  \sup_{P \in \rr P(I_0) }\big( \frac{1}{|I_P|} \int_{\rr R} \one_{F_{\kappa}} \cdot \ci_{I_P} dx  \big)^\frac{1}{q_1} \lesssim \sup_{P \in \rr P(I_0) }\big( \frac{1}{|I_P|} \int_{\rr R} \big| f \cdot \ci_{I_P} \big|^{q_1+\epsilon} dx  \big)^{\frac{1}{q_1+\epsilon}}.
\]

Let $K$ be so that $\ds 2^K \sim  \sup_{P \in \rr P(I_0) }\big( \frac{1}{|I_P|} \int_{\rr R} \big| f \cdot \ci_{I_P} \big|^{q_1+\epsilon} dx  \big)^{\frac{1}{q_1+\epsilon}}$. Since $ 2^{\kappa-1} \cdot \one_{F_{\kappa}} \leq |f(x)|  \one_{F_{\kappa}}(x) \leq  2^{\kappa}\cdot \one_{F_{\kappa}} $, we have $\one_{F_{\kappa}}(x) \lesssim 2^{-\kappa \, q_1} |f(x)|^{q_1}$ and hence
\begin{equation*}
\frac{1}{\vert I_P \vert} \int_{\rr R }\one_{F_{\kappa}} \cdot \ci_{I_P} dx  \lesssim 2^{-\kappa \, \left(q_1 +\epsilon \right)} \frac{1}{\vert I_P \vert} \int_{\rr R }\big| f(x) \cdot \ci_{I_P}\big|^{q_1+\epsilon} dx \lesssim 2^{\left( K -\kappa \right) \left( q_1+\epsilon \right)}. 
\end{equation*}

We actually have more than that; namely, 
\[
\sup_{P \in \rr P (I_0)}\big( \frac{1}{|I_P|} \int_{\rr R} \one_{F_{\kappa}} \cdot \ci_{I_P} dx  \big)^\frac{1}{q_1} \lesssim \min\big(1, 2^{\left( K -\kappa \right) \left( q_1+\epsilon \right)/{q_1}}\big).
\]

We split the sum over $\kappa \leq K$ and $\kappa >K$. In the first case, we obtain 
\[
\sum_{\kappa \leq K} 2^{\kappa}  \sup_{P \in \rr P(I_0)}\big( \frac{1}{|I_P|} \int_{\rr R} \one_{F_{\kappa}} \cdot \ci_{I_P} dx  \big)^\frac{1}{q_1} \lesssim \sum_{\kappa \leq K} 2^{\kappa} \lesssim 2^K.
\]
Summation over $\kappa >K$ yields
\[
\sum_{\kappa >K} 2^\kappa  \sup_{P \in \rr P(I_0)}\big( \frac{1}{|I_P|} \int_{\rr R} \one_{F_{\kappa}} \cdot \ci_{I_P} dx  \big)^\frac{1}{q_1} \lesssim 2^{K} \sum_{\kappa >K} 2^{\left(\kappa-K\right)\left( 1-  \frac{q_1+\epsilon}{q_1}\right)} \lesssim 2^{K}.
\]
Combining the two cases together, we obtain the conclusion. The last step also shows that we cannot avoid the $\epsilon$ loss.
\end{proof}
\end{proposition}
 
Using Proposition \ref{prop:restricted->general}, \ref{eq-helM-n-sizes-only} implies  
 {\fontsize{10}{9}
\begin{equation}
\label{eq:local-sizes-etc} \tag*{$\ii P^{**}(n):$} 
\vert \Lambda_{BHT;\rr{P}(I_0)}(\vec f, \vec g, \vec h) \vert \lesssim  \left(\sssize_{\rr P(I_0)} \| \vec f(x, \cdot) \|_{L^{R^n_1}}^{s_1} \right)^{\frac{1}{s_1}} \cdot \left(\sssize_{\rr P(I_0)} \| \vec g(x, \cdot) \|_{L^{R^n_2}}^{s_2}\right)^{\frac{1}{s_2}} \cdot \left(\sssize_{\rr P(I_0)} \|\vec h(x, \cdot) \|_{L^{(R^n)'}}^{s_3} \right)^{\frac{1}{s_3}} \cdot \vert I_0 \vert,
\end{equation}} provided there exist $\theta_1, \theta_2, \theta_3$ so that $0 \leq \theta_1, \theta_2, \theta_3 <1,  \theta_1+\theta_2+\theta_3  =1$ and 
\[
\frac{1}{r_1^j}, \frac{1}{s_1} <\frac{1+\theta_1}{2}, \quad \frac{1}{r_2^j}, \frac{1}{s_2} <\frac{1+\theta_2}{2} \quad \frac{1}{r^j}, \frac{1}{s_3} <\frac{1+\theta_3}{2} \quad \forall 1 \leq j \leq n.
\]

\begin{rremark}
\label{remark:expl-name-hel}
As mentioned earlier, the $(n+1)$-depth estimate $\ii P^{**}(n+1)$ is obtained through a careful stopping time (which identifies the intervals realizing the maximum in the inequality above) from $\ii P^{**}(n)$ and H\"older's inequality, the latter allowing us to move to the next level of iteration. If instead we don't have the extra variable to apply H\"older's inequality to and we perform the usual stopping time procedure starting from $\ii P^{**}(n)$, at the end we obtain exactly the same estimate $\ii P^{**}(n)$, as if moving in circles. This too accounts for our choice of the name \emph{the helicoidal method}.
\end{rremark}

For the exponents $s_1, s_2, s_3$, we don't assume any H\"older-type condition. In fact, they are chosen to be as small as possible, but we still require them to satisfy for the same $\theta_1, \theta_2, \theta_3$ as above and an $\epsilon>0$ which can be as small as we wish, the condition 
\[
 \frac{1}{s_1} <\frac{1+\theta_1}{2}-\epsilon, \quad \frac{1}{s_2} <\frac{1+\theta_2}{2}-\epsilon \quad \frac{1}{s_3} <\frac{1+\theta_3}{2}-\epsilon.
\] 
 
Now we claim that a local estimate such as \eqref{eq:local-sizes-etc} implies a vector-valued sparse domination:

\begin{theorem}
\label{thm:local->sparse}
Let $\rr P$ be a collection of tiles and assume that 
{\fontsize{9}{9}\begin{equation*}
\vert \Lambda_{BHT;\rr{P}(I_0)}(\vec f, \vec g, \vec h) \vert \lesssim  \left(\sssize_{\rr P(I_0)} \| \vec f(x, \cdot) \|_{L^{R^n_1}}^{s_1} \right)^{\frac{1}{s_1}} \cdot \left(\sssize_{\rr P(I_0)} \| \vec g(x, \cdot) \|_{L^{R^n_2}}^{s_2}\right)^{\frac{1}{s_2}} \cdot \left(\sssize_{\rr P(I_0)} \|\vec h(x, \cdot) \|_{(L^{R^n)'}}^{s_3} \right)^{\frac{1}{s_3}} \cdot \vert I_0 \vert,
\end{equation*}}holds for any dyadic interval $I_0$ and any vector-valued functions $\vec f, \vec g, \vec h$ so that $\| \vec f(x, \cdot) \|_{L^{R^n_1}}, \| \vec g(x, \cdot) \|_{L^{R^n_2}}, \| \vec h(x, \cdot) \|_{L^{(R^n)'}}$ are locally integrable. Then, there exists a sparse family of dyadic intervals $\ic S$, depending on $\vec f, \vec g, \vec h$ and on the Lebesgue exponents $s_1, s_2, s_3$ so that 
{\fontsize{9}{9}
\[
\big| \Lambda_{\rr P}(\vec f, \vec g, \vec h) \big|\lesssim \sum_{Q \in \ic S} \big( \frac{1}{|Q|} \int_{\rr R} \| \vec f(x, \cdot) \|_{L^{R^n_1}}^{s_1} \cdot \ci_{Q} dx \big)^\frac{1}{s_1}  \big( \frac{1}{|Q|} \int_{\rr R} \| \vec g(x, \cdot) \|_{L^{R^n_2}}^{s_2} \cdot \ci_{Q} dx \big)^\frac{1}{s_2}\big( \frac{1}{|Q|} \int_{\rr R} \| \vec h(x, \cdot) \|_{L^{(R^n)'}}^{s_3} \cdot \ci_{Q} dx \big)^\frac{1}{s_3} \cdot |Q|.
\]}
\begin{proof}
We only treat the scalar valued case, but the general one follows in the same way. We have, for any interval $I_0,$
\begin{equation}
\label{eq:local-est-P}
\vert \Lambda_{\rr{P}(I_0)}(f, g, h) \vert \lesssim  \left(\sssize_{ \rr P(I_0)} \,\vert f \vert^{s_1} \right)^{1/{s_1}} \cdot \left(\sssize_{\rr P(I_0)} \,\vert g \vert^{s_2} \right)^{1/{s_2}} \cdot \left(\sssize_{ \rr P(I_0)} \, \vert h\vert^{s_3} \right)^{1/{s_3}} \cdot \vert I_0 \vert,
\end{equation}
and we want to construct a sparse collection of dyadic intervals $\ic S$ for which 
{\fontsize{9}{9}\begin{equation}
\label{eq:sparse-BHT}
\big| \Lambda_{\rr P}(f, g, h) \big|\lesssim \sum_{Q \in \ic S} \big( \frac{1}{|Q|} \int_{\rr R} \vert f \vert^{s_1} \cdot \ci_{Q} dx \big)^{\frac{1}{s_1}}  \big( \frac{1}{|Q|} \int_{\rr R} \vert g \vert^{s_2} \cdot \ci_{Q} dx \big)^{\frac{1}{s_2}}\big( \frac{1}{|Q|} \int_{\rr R} \vert h \vert^{s_3} \cdot \ci_{Q} dx \big)^{\frac{1}{s_3}} \cdot |Q|.
\end{equation}}

The sparse collection will be structured iteratively as $\ic S =\bigcup_{k \geq 0} \ic S_k$, and the intervals on the $k+1$ level $\ic S_{k+1}$ will consist of the ``descendants" of the intervals on the $k$-th level $\ic S_k$. That is, for every $k \geq 0$, and every $Q_0 \in \ic S_k$, we have $ch_{\ic S}(Q_0)$ the collection of direct descendants of $Q_0$ in $\ic S$, and moreover,
\[
\ic S_{k+1} =\bigcup_{Q_0 \in \ic S_k} ch_{\ic S}(Q_0).
\]
For each $Q_0 \in \ic S$, we have an associated subcollection of tiles $\rr P_{Q_0} \subseteq \rr P(Q_0)$, which induces a partition of $\rr P$:
\[
\rr P:=\bigsqcup_{Q_0 \in \ic S} \rr P_{Q_0}.
\]

The level $0$, $\ic S_0$, will consist of the maximal spatial supports of the collection $\rr P$ of tiles:
{\fontsize{10}{10}\[
\ic S_0= \lbrace I \text{ dyadic interval }: I=I_P \text{ for some } P\in \rr P \text{ and if } I \subseteq I_{P'} \text{for another } P' \in \rr P, \text{ then } I_P=I_{P'}   \rbrace.
\] }

Now we assume $\ic S_k$ was constructed and for every $Q_0 \in \ic S_k$ we define $ch_{\ic S}(Q_0)$, the union of which will represent $\ic S_{k+1}$, and $\rr P_{Q_0}$, the subcollection of tiles. For each $Q_0 \in \ic S_{k}$, the descendants $ch_{\ic S}(Q_0)$ will consist of all the maximal dyadic intervals $Q \subseteq Q_0$ so that there exists at least one tri-tile $P \in \rr P$ with $I_P \subset Q$ and so that one of the following holds:
\[
\big( \frac{1}{\vert Q \vert} \int_{\rr R} \vert f(x) \vert^{s_1} \cdot \ci_{Q}(x) dx  \big)^{1/s_1} > C \cdot \big( \frac{1}{\vert Q_0 \vert} \int_{\rr R} \vert f(x) \vert^{s_1} \cdot \ci_{Q_0}(x) dx  \big)^{1/s_1} \qquad \text{or}
\]
\[
\big( \frac{1}{\vert Q \vert} \int_{\rr R} \vert g(x) \vert^{s_2} \cdot \ci_{Q}(x) dx  \big)^{1/s_2} > C \cdot \big( \frac{1}{\vert Q_0 \vert} \int_{\rr R} \vert g(x) \vert^{s_2} \cdot \ci_{Q_0}(x) dx  \big)^{1/s_2} \qquad \text{or}
\] 
\[
\big( \frac{1}{\vert Q \vert} \int_{\rr R} \vert h(x) \vert^{s_3} \cdot \ci_{Q}(x) dx  \big)^{1/s_3} > C \cdot \big( \frac{1}{\vert Q_0 \vert} \int_{\rr R} \vert h(x) \vert^{s_3} \cdot \ci_{Q_0}(x) dx  \big)^{1/s_3} . \qquad       \text{              }  
\]

Then we have that 
{\fontsize{9}{9}
\begin{align*}
&\bigcup_{Q \in ch_{\ic S}\left( Q_0 \right)} Q \subseteq \lbrace x: \ic M_{s_1} (f \cdot \ci_{Q_0}) >  C \cdot \big( \frac{1}{\vert Q_0 \vert} \int_{\rr R} \vert f(x) \vert^{s_1} \cdot \ci_{Q_0}(x) dx  \big)^\frac{1}{s_1} \rbrace  \\ & \cup \lbrace x: \ic M_{s_2} (g \cdot \ci_{Q_0}) >  C \cdot \big( \frac{1}{\vert Q_0 \vert} \int_{\rr R} \vert g(x) \vert^{s_2} \cdot \ci_{Q_0}(x) dx  \big)^\frac{1}{s_2} \rbrace \cup \lbrace x: \ic M_{s_3} (h \cdot \ci_{Q_0}) >  C \cdot \big( \frac{1}{\vert Q_0 \vert} \int_{\rr R} \vert h(x) \vert^{s_3} \cdot \ci_{Q_0}(x) dx  \big)^\frac{1}{s_3} \rbrace,
\end{align*}}
and in consequence 
\[
\sum_{Q \in ch_{\ic S}\left( Q_0 \right)} |Q| \lesssim \frac{1}{2} |Q_0|.
\]

The last inequality implies the sparseness property, as described in Definition \ref{def-sparse-collection}, of the collection $\ic S$

For every $Q_0$ in $\ic S_k$, the associated collection of tiles $\rr P_{Q_0} \subseteq \rr P(Q_0)$ consists of all tiles $P \in \rr P$ so that 
\[
I_P \subseteq Q_0, \text{ but  } I_P \nsubseteq Q, \quad \forall \, Q \in ch_{\ic S}(Q_0).
\]

For that reason, $\sssize_{\rr P_{Q_0}}^{s_1} (f) \lesssim  \big( \frac{1}{\vert Q_0 \vert} \int_{\rr R} \vert f(x) \vert^{s_1} \cdot \ci_{Q_0}(x) dx  \big)^\frac{1}{s_1}$, and similarly for the functions $g$ and $h$. Then we have, due to \eqref{eq:local-est-P},
\begin{align*}
&\vert \Lambda_{\rr P}(f, g, h) \vert \lesssim \sum_{Q \in \ic S} \vert \Lambda_{\rr P_Q}(f, g, h) \vert \\
& \lesssim  \sum_{Q \in \ic S} \left(\sssize_{\rr P_Q} \,\vert f \vert^{s_1} \right)^{1/{s_1}} \cdot \left(\sssize_{\rr P_Q} \,\vert g \vert^{s_2} \right)^{1/{s_2}} \cdot \left(\sssize_{\rr P_q} \, \vert h\vert^{s_3} \right)^{1/{s_3}} \cdot \vert Q \vert,
\end{align*}
which further implies \eqref{eq:sparse-BHT}.
\end{proof}
\end{theorem}

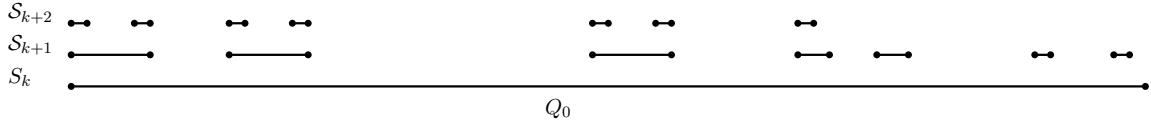
\begin{figure}
\psscalebox{.75 .75} 
{
\begin{pspicture}(0,-0.985)(20.16,0.985)
\psline[linecolor=black, linewidth=0.04, dotsize=0.06cm 1.6]{*-*}(1.12,-0.425)(20.16,-0.425)
\rput[bl](9.52,-0.985){$Q_0$}
\psline[linecolor=black, linewidth=0.04, dotsize=0.06cm 1.6]{*-*}(1.12,0.135)(2.52,0.135)
\psline[linecolor=black, linewidth=0.04, dotsize=0.06cm 1.6]{*-*}(3.92,0.135)(5.32,0.135)
\psline[linecolor=black, linewidth=0.04, dotsize=0.06cm 1.6]{*-*}(10.36,0.135)(11.76,0.135)
\psline[linecolor=black, linewidth=0.04, dotsize=0.06cm 1.6]{*-*}(14.0,0.135)(14.56,0.135)
\psline[linecolor=black, linewidth=0.04, dotsize=0.06cm 1.6]{*-*}(15.4,0.135)(15.96,0.135)
\psline[linecolor=black, linewidth=0.04, dotsize=0.06cm 1.6]{*-*}(19.88,0.135)(19.6,0.135)(19.6,0.135)
\psline[linecolor=black, linewidth=0.04, dotsize=0.06cm 1.6]{*-*}(18.48,0.135)(18.2,0.135)
\psline[linecolor=black, linewidth=0.04, dotsize=0.06cm 1.6]{*-*}(1.12,0.695)(1.4,0.695)
\psline[linecolor=black, linewidth=0.04, dotsize=0.06cm 1.6]{*-*}(2.52,0.695)(2.24,0.695)
\psline[linecolor=black, linewidth=0.04, dotsize=0.06cm 1.6]{*-*}(3.92,0.695)(4.2,0.695)
\psline[linecolor=black, linewidth=0.04, dotsize=0.06cm 1.6]{*-*}(5.32,0.695)(5.04,0.695)
\psline[linecolor=black, linewidth=0.04, dotsize=0.06cm 1.6]{*-*}(10.36,0.695)(10.64,0.695)
\psline[linecolor=black, linewidth=0.04, dotsize=0.06cm 1.6]{*-*}(11.76,0.695)(11.48,0.695)
\psline[linecolor=black, linewidth=0.04, dotsize=0.06cm 1.6]{*-*}(14.0,0.695)(14.28,0.695)
\rput[bl](0.0,0.695){$\Huge \mathcal{S}_{k+2} \quad  $}
\rput[bl](0.0,0.135){$\Huge \mathcal{S}_{k+1}   $}
\rput[bl](0.0,-0.425){$\Huge S_{k}$}
\end{pspicture}
}
\caption{The sparse collection of intervals $\ic S= \bigcup_{k \geq 0}\ic S_k$}
\label{fig:sparse-dom}
\end{figure}

\begin{rremark}
\label{remark:stopping-times}
We can see from Figure \ref{fig:sparse-dom} that the stopping time yielding the sparse domination collection ( which we call \emph{stopping time of type $SST$}) $\ic S$ is a bottom-up procedure: we start with large spatial intervals $Q_0$ and construct the descendants $ch_{\ic S}(Q_0)$, and in doing so, the averages $\sssize_{\rr P_{Q}}^{s_1}(f)$ are increasing.

On the contrary, the stopping time for the vector-valued estimates (referred to as \emph{of type $VVST$}) is a top-down approach: we start with largest possible size, which is concentrated on smallest possible intervals, and as the procedure continues, the sizes are decreasing and the spatial intervals are becoming larger. This is illustrated in Figure \ref{fig:sparse}.
\end{rremark}

Lastly, in order to obtain the full result of Theorem \ref{thm:sparse-BHT}, i.e. a sparse domination of $\|BHT_{\rr P}(f, g) \cdot v \|_q^q$ (rather that just a sparse domination of the trilinear form), we need to rephrase Theorem \ref{thm:local->sparse}. We note that such a sparse domination is well defined in the quasi-Banach case as well.

\begin{proposition}
\label{prop:Lq:local->sparse}
Let $\rr P$ be a collection of tiles and assume that the bilinear operator $ T_{\rr P}$ satisfies
{\fontsize{9}{9}\begin{equation*}
\big\| \| T_{\rr P(I_0)}(\vec f, \vec g)  \|_{L^{R}(\ii W, \mu)}  \cdot v     \big\|_q^q  \lesssim  \left(\sssize_{\rr P(I_0)} \| \vec f(x, \cdot) \|_{L^{R^n_1}}^{s_1} \right)^{\frac{q}{s_1}} \cdot \left(\sssize_{\rr P(I_0)} \| \vec g(x, \cdot) \|_{L^{R^n_2}}^{s_2}\right)^{\frac{q}{s_2}} \cdot \left(\sssize_{\rr P(I_0)} |v|^{s_3} \right)^{\frac{q}{s_3}} \cdot \vert I_0 \vert,
\end{equation*}}for any dyadic interval $I_0$, any vector-valued functions $\vec f, \vec g$ so that $\| \vec f(x, \cdot) \|_{L^{R^n_1}}, \| \vec g(x, \cdot) \|_{L^{R^n_2}}$ are locally integrable, and any $q$-integrable function $v$. Then, provided $\ds \| \cdot \|_{{L^{R}(\ii W, \mu)}}^q$ is subadditive, there exists a sparse family of dyadic intervals $\ic S$, depending on $\vec f, \vec g, v$ and on the Lebesgue exponents $s_1, s_2, s_3$ so that 
{\fontsize{9}{9}
\[
\big\| \| T_{\rr P}(\vec f, \vec g)  \|_{L^{R}(\ii W, \mu)}  \cdot v  \big\|_q^q  \lesssim \sum_{Q \in \ic S} \big( \frac{1}{|Q|} \int_{\rr R} \| \vec f(x, \cdot) \|_{L^{R^n_1}}^{s_1} \cdot \ci_{Q} dx \big)^\frac{q}{s_1}  \big( \frac{1}{|Q|} \int_{\rr R} \| \vec g(x, \cdot) \|_{L^{R^n_2}}^{s_2} \cdot \ci_{Q} dx \big)^\frac{q}{s_2}\big( \frac{1}{|Q|} \int_{\rr R} |v|^{s_3} \cdot \ci_{Q} dx \big)^\frac{q}{s_3} \cdot |Q|.
\]}
\end{proposition}

The subadditivity condition is fulfilled if $\ds q\leq \min_{j} r^j$. If $\ds \| \cdot \|_{{L^{R}(\ii W, \mu)}}^q$ is not subadditive, we pick $\tau< q$ so that $\ds \| \cdot \|_{{L^{R}(\ii W, \mu)}}^\tau$ is subadditive (and a sparse domination of $\big\| \| T_{\rr P}(\vec f, \vec g)  \|_{L^{R}(\ii W, \mu)}  \cdot v  \big\|_\tau^\tau$ is at our disposal) and we construct a sparse domination of $\big\| \| T_{\rr P}(\vec f, \vec g)  \|_{L^{R}(\ii W, \mu)}  \cdot v  \big\|_q^q$:

\begin{proposition}
\label{prop:sparse->sparse->no-subadditivity}
Let $s_1, s_2, s_3$ and $\tau$ be so that $\ds \| \cdot \|_{{L^{R}(\ii W, \mu)}}^\tau$ is subadditive and for any vector-valued functions $\vec f, \vec g$ with the property that $\| \vec f(x, \cdot) \|_{L^{R^n_1}}, \| \vec g(x, \cdot) \|_{L^{R^n_2}}$ are locally integrable, and any $\tau$-integrable function $v$, there exists a sparse family of dyadic intervals $\ic S$, depending on $\vec f, \vec g, v$ and on the Lebesgue exponents $s_1, s_2, s_3$ so that 
{\fontsize{9}{9}
\[
\big\| \| T_{\rr P}(\vec f, \vec g)  \|_{L^{R}(\ii W, \mu)}  \cdot v  \big\|_\tau^\tau  \lesssim \sum_{Q \in \ic S} \big( \frac{1}{|Q|} \int_{\rr R} \| \vec f(x, \cdot) \|_{L^{R^n_1}}^{s_1} \cdot \ci_{Q} dx \big)^\frac{\tau}{s_1}  \big( \frac{1}{|Q|} \int_{\rr R} \| \vec g(x, \cdot) \|_{L^{R^n_2}}^{s_2} \cdot \ci_{Q} dx \big)^\frac{\tau}{s_2}\big( \frac{1}{|Q|} \int_{\rr R} |v|^{s_3} \cdot \ci_{Q} dx \big)^\frac{\tau}{s_3} \cdot |Q|.
\]}

Then for any $q>\tau$ and any vector-valued functions $\vec f, \vec g$ with the property that $\| \vec f(x, \cdot) \|_{L^{R^n_1}}, \| \vec g(x, \cdot) \|_{L^{R^n_2}}|$ are locally integrable, and any $q$-integrable function $v$, there exist Lebesgue exponents $\tilde s_1, \tilde s_2, \tilde s_3$ and a sparse family of dyadic intervals $\ic S$, depending on $\vec f, \vec g, v$ and on the Lebesgue exponents $ s_1, s_2, s_3$ so that 
{\fontsize{9}{9}
\begin{equation}
\label{eq:sparse-dom-non-subadd}
\big\| \| T_{\rr P}(\vec f, \vec g)  \|_{L^{R}(\ii W, \mu)}  \cdot v  \big\|_q^q  \lesssim \sum_{Q \in \ic S} \big( \frac{1}{|Q|} \int_{\rr R} \| \vec f(x, \cdot) \|_{L^{R^n_1}}^{\tilde s_1} \cdot \ci_{Q} dx \big)^\frac{q}{\tilde s_1}  \big( \frac{1}{|Q|} \int_{\rr R} \| \vec g(x, \cdot) \|_{L^{R^n_2}}^{\tilde s_2} \cdot \ci_{Q} dx \big)^\frac{q}{\tilde s_2}\big( \frac{1}{|Q|} \int_{\rr R} |v|^{\tilde s_3} \cdot \ci_{Q} dx \big)^\frac{q}{\tilde s_3} \cdot |Q|.
\end{equation}}
Moreover, $\tilde s_1=s_1, \tilde s_2=s_2$, while $\tilde s_3$ is so that
\[
\frac{1}{\tilde s_3} <\frac{1}{s_3}-\frac{1}{\tau}+\frac{1}{q}.
\]
\begin{proof}
First note that  
\[
\big\| \| T_{\rr P}(\vec f, \vec g)  \|_{L^{R}(\ii W, \mu)}  \cdot v  \big\|_q^q=\big\| \big| \| T_{\rr P}(\vec f, \vec g)  \|_{L^{R}(\ii W, \mu)}  \cdot v \big| ^\tau  \big\|_{\frac{q}{\tau}}^{\frac{q}{\tau}},
\]
and since $\ds \frac{q}{\tau}>1$, we can find $u \in L^{\left( \frac{q}{\tau} \right)'}$ with $\ds \|u \|_{\left( \frac{q}{\tau} \right)'}=1$ such that
\[
\big\| \big| \| T_{\rr P}(\vec f, \vec g)  \|_{L^{R}(\ii W, \mu)}  \cdot v \big| ^\tau  \big\|_{\frac{q}{\tau}} = \int_{\rr R} \big| \| T_{\rr P}(\vec f, \vec g)  \|_{L^{R}(\ii W, \mu)}  \cdot v \big| ^\tau u dx=\big\| \| T_{\rr P}(\vec f, \vec g)  \|_{L^{R}(\ii W, \mu)}  \cdot v \cdot u^\frac{1}{\tau}  \big\|_\tau^\tau.
\]

Now we use the sparse domination for $\| \cdot \|_\tau^\tau$:
\begin{align*}
&\big\| \| T_{\rr P}(\vec f, \vec g)  \|_{L^{R}(\ii W, \mu)}  \cdot v \cdot u^\frac{1}{\tau}  \big\|_\tau^\tau \\
 &\lesssim \sum_{Q \in \ic S} \big(\aver{Q} \| \vec f(x, \cdot) \|_{L^{R^n_1}}^{s_1}   \big)^\frac{\tau}{s_1} \, \big(\aver{Q} \| \vec g(x, \cdot) \|_{L^{R^n_2}}^{s_2}   \big)^\frac{\tau}{s_2} \, \big(\aver{Q} |v \cdot u^\frac{1}{\tau}|^{s_3}   \big)^\frac{\tau}{s_3} \, |Q| \\
 &\lesssim \sum_{Q \in \ic S} \big(\aver{Q} \| \vec f(x, \cdot) \|_{L^{R^n_1}}^{s_1}   \big)^\frac{\tau}{s_1} \, \big(\aver{Q} \| \vec g(x, \cdot) \|_{L^{R^n_2}}^{s_2}   \big)^\frac{\tau}{s_2} \, \big(\aver{Q} |v |^{\tilde s_3}   \big)^\frac{\tau}{\tilde s_3} \, |Q| ^{\frac{1}{(q/\tau)}} \, \big(\aver{Q} |u |^\frac{\alpha}{\tau}   \big)^\frac{\tau}{\alpha} |Q| ^{\frac{1}{(q/\tau)'}},
\end{align*}
where we used H\"older's inequality with $\ds \frac{1}{s_3}=\frac{1}{\tilde s_3}+\frac{1}{\alpha}$. Now again we use H\"older's inequality in $\ell^{\frac{q}{\tau}}$ and $\ell^{\left(\frac{q}{\tau} \right)'}$ spaces indexed after the collection $\ic S$; the first term is precisely the right hand side of \eqref{eq:sparse-dom-non-subadd} with $\tilde s_1=s_1$ and $\tilde s_2=s_2$, while the second one can be estimated, using the sparseness properties of $\ic S$, by 
\[
\Big(  \sum_{Q \in \ic S}  \big(\aver{Q} |u |^\frac{\alpha}{\tau}   \big)^{ \frac{\tau}{\alpha} \cdot \left( \frac{q}{\tau}  \right)' }|Q|   \Big)^{\frac{1}{(q/\tau)'}} \lesssim \big\| \ic M_{\frac{\alpha}{\tau}} ( u)  \big\|_{\left(\frac{q}{\tau}\right)'}.
\]

Hence we can deduce the desired sparse domination from \eqref{eq:sparse-dom-non-subadd} provided we have
\[
\big\|  \ic M_{\frac{\alpha}{\tau}} (u) \big\|_{\left( \frac{q}{\tau} \right)'} \lesssim \big\|  u \big\|_{\left( \frac{q}{\tau} \right)'}=1.
\]

The last estimate is contingent upon
\[
\frac{\alpha}{\tau} < \left( \frac{q}{\tau} \right)' \Leftrightarrow \frac{1}{\tau} -\frac{1}{q} < \frac{1}{s_3} -\frac{1}{\tilde s_3},
\]
which is the condition we had on $\tilde s_3$.
\end{proof}
\end{proposition}

\begin{rremark}
\label{remark:rangeFeffStein}
\begin{enumerate}[label=(\alph*), ref=\alph*]
\item We note that the sparse collection $\ic S$ depends on the functions involved and $\tau, s_1, s_2$ and $s_3$, rather than on $q, \tilde s_1, \tilde s_2$ and $\tilde s_3$.
\item \label{remark:rangeFS-item} For the bilinear Hilbert transform operator, we require the tuples $R_1, R_2$ and $R$ to satisfy the assumptions \eqref{eq:cond-r_j-theta} of Theorem \ref{thm:sparse-BHT} (i.e. they depend on certain $\theta_1, \theta_2$ and $\theta_3$ which are contained in $[0, 1)$ and satisfy $\theta_1+\theta_2+\theta_3=1$).

If $q \leq 1$ and $\ds \| \cdot \|_{{L^{R}(\ii W, \mu)}}^q$ is subadditive, then $s_1, s_2$ and $s_3$ need to satisfy
\[
\frac{1}{s_1}<\frac{1+\theta_1}{2}, \qquad \frac{1}{s_2}<\frac{1+\theta_2}{2}, \qquad \frac{1}{s_3}<\frac{1}{q'}-\frac{1+\theta_3}{2}=\frac{1}{q}-\frac{\theta_1+\theta_2}{2}.
\]
On the other hand, if $\ds \| \cdot \|_{{L^{R}(\ii W, \mu)}}^q$ is not subadditive, then we deduce the sparse domination result by using Proposition \ref{prop:sparse->sparse->no-subadditivity} and a $\tau$ for which $\ds \| \cdot \|_{{L^{R}(\ii W, \mu)}}^\tau$ is subadditive. For such a $\tau$, we have a sparse domination result with averages $s_1(\tau), s_2(\tau)$ and $s_3(\tau)$, where
\[
\frac{1}{s_1(\tau)}<\frac{1+\theta_1}{2}, \qquad \frac{1}{s_2(\tau)}<\frac{1+\theta_2}{2}, \qquad \frac{1}{s_3(\tau)}<\frac{1}{\tau'}-\frac{1+\theta_3}{2}=\frac{1}{\tau}-\frac{\theta_1+\theta_2}{2}.
\]

And this implies a sparse domination result for $q$, with $s_1, s_2$ and $s_3$ averages, where $s_1=s_1(\tau), s_2=s_2(\tau)$ and $\ds \frac{1}{s_3}<\frac{1}{s_3(\tau)}-\frac{1}{\tau}+\frac{1}{q}$. We want to minimize $s_3$ (and hence $s_3(\tau)$ as well), so we can pick $s_3(\tau)$ so that
\[
\frac{1}{s_3(\tau)}=\frac{1}{\tau} -\frac{\theta_1+\theta_2}{2} -\epsilon,
\]
for a certain $\epsilon$ small enough. Then $s_3$ can be chosen similarly:
\[
\frac{1}{s_3}=\frac{1}{s_3(\tau)}-\frac{1}{\tau}+\frac{1}{q}-\epsilon=\frac{1}{q} -\frac{\theta_1+\theta_2}{2} -2\epsilon.
\]

So even if we fail to have subadditivity, the conditions on $s_1, s_2$ and $s_3$ are given by \eqref{eq:condition-on-s} also in this case:
\[
\frac{1}{s_1}<\frac{1+\theta_1}{2}, \qquad \frac{1}{s_2}<\frac{1+\theta_2}{2}, \qquad \frac{1}{s_3}<\frac{1}{q}-\frac{\theta_1+\theta_2}{2}.
\]

\item In Corollary \ref{corollary-Fefferman-Stein}, we disregard the value of $s_3$ by setting $v \equiv 1$; we need nevertheless to have $\ds 0< \frac{1}{s_3}<\frac{1}{q}-\frac{\theta_1+\theta_2}{2}$, which implies that
\[
\frac{1}{s_1}+\frac{1}{s_2}<1+\frac{\theta_1+\theta_2}{2} < \min \left( \frac{3}{2}, 1+\frac{1}{q} \right).
\]

\end{enumerate}
\end{rremark}

\bibliographystyle{alpha}
\bibliography{TheHelicoidalMethod.bbl}
\end{document}